\documentclass[10pt, 
amssymb,psfig]{amsbook}
\newtheorem{thm}{Theorem}[section]
\newtheorem{cor}[thm]{Corollary}
\newtheorem{lem}[thm]{Lemma}
\newtheorem{prop}[thm]{Proposition}
\newtheorem{definition}[thm]{Definition}

\theoremstyle{remark}
\newtheorem{rem}[thm]{Remark}

\theoremstyle{exercise}

\newtheorem{example}[thm] {Example}
{{\em Idea of the proof}.} 

\numberwithin{equation}{section}

\theoremstyle{definition}

\font\nt=cmr7

\def\be{\begin{equation}}
\def\ee{\end{equation}}

\def\ssk{\smallskip}
\def\msk{\medskip}

\def\nin{\noindent}

\DeclareMathOperator{\divi}{div}

\DeclareMathOperator{\ord}{ord}

\DeclareMathOperator{\spe}{sp}

\newcommand{\Jul}{{\mathcal{J}}}
\newcommand{\Fatou}{{\mathcal{F}}}

\newcommand{\Kfilled}{{\mathcal {K} }}

\newcommand{\di}{\partial}
\newcommand{\dibar}{\bar\partial}

\newcommand{\ra}{\rightarrow}

\newcommand{\dra}{\dashrightarrow}

\newcommand{\hor}{{\mathrm{hor}}}
\newcommand{\ver}{{\mathrm{ver}}}

\def\sm{\smallsetminus}

\newcommand{\inter}{\operatorname{int}}

\newcommand{\supp}{\operatorname{supp}}
\newcommand{\id}{\operatorname{id}}
\newcommand{\Id}{\id}

\newcommand{\const}{\mathrm{const}}
\def\loc{{\mathrm{loc}}}

\renewcommand{\max}{{\operatorname{max}}}

\newcommand{\eps}{{\varepsilon}}
\newcommand{\De}{{\Delta}}
\newcommand{\de}{{\delta}}

\newcommand{\la}{{\lambda}}

\newcommand{\Si}{{\Sigma}}
\newcommand{\Om}{{\Omega}}
\newcommand{\om}{{\omega}}

\newcommand{\Cheb} {{T}}   
\newcommand{\cheb} {{t}}

\newcommand{\Zhuk}{{\mathrm{Zh}}}

\newcommand{\shift}{{\tau}}

\newcommand{\DD}{{\mathcal D}}

\newcommand{\GG}{{\mathcal G}}

\newcommand{\HH}{{\mathcal H}}

\newcommand{\LL}{{\mathcal L}}
\newcommand{\MM}{{\mathcal M}}

\newcommand{\TT}{{\mathcal T}}

\newcommand{\C}{{\Bbb C}}
\newcommand{\bC}{{\hat{\Bbb C}}}
\newcommand{\hC}{{\hat{\Bbb C}}}

\newcommand{\I}{{\Bbb I}}

\newcommand{\M}{{\Bbb M}}
\newcommand{\N}{{\Bbb N}}

\renewcommand{\P}{{\Bbb P}}
\newcommand{\R}{{\Bbb R}}
\newcommand{\T}{{\Bbb T}}

\newcommand{\Z}{{\Bbb Z}}

\newcommand{\bP}{\mathbb{P}}

\newcommand{\Space} {{\mathfrak{R}}}

\newcommand{\p}{{\mathbf p}}

\DeclareMathOperator{\jsp}{jsp}

\DeclareMathOperator{\Sym}{Sym}

\newcommand{\Jac}{\operatorname{Jac}}

\newcommand{\Aut}{\operatorname{Aut}}

\newcommand{\GL}{ {\mathrm{GL}} }

\newcommand{\RP}{ {\Bbb{RP}}   }
\newcommand{\Fix}{{\mathrm{Fix}}}

\def\tr{{\text{tr}}}

\renewcommand{\Re}{\operatorname{Re}}
\renewcommand{\Im}{\operatorname{Im}}

\newcommand{\base}{{\boldsymbol{\circ}}}

\newcommand{\cond}{ \, |\,}

\renewcommand{\Im}{\operatorname{Im}}

\catcode`\@=12

\def\Empty{}
\newcommand\oplabel[1]{
  \def\OpArg{#1} \ifx \OpArg\Empty {} \else
  	\label{#1}
  \fi}
		
\newcommand{\comm}[1]{}
\newcommand{\comment}[1]{}

\newtheorem*{Fundamental Theorem of 2D Topology}{Fundamental Theorem of 2D Topology}
\newtheorem*{Jordan Theorem}{Jordan Theorem}
\newtheorem*{Alexander Trick}{Alexander Trick}
\newtheorem*{Lifting Criterion}{Lifting Criterion}
\newtheorem*{Non-Crossing Principle}{Non-Crossing Principle}
\newtheorem*{Small Overlapping Principle}{Small Overlapping Principle}

\newtheorem*{Bers Lemma}{Bers' Lemma}
\newtheorem*{Weyl's Lemma}{Weyl's Lemma}
\newtheorem*{Riemann Mapping Theorem}{Riemann Mapping Theorem}
\newtheorem*{Schwarz Lemma}{Schwarz Lemma}
\newtheorem*{Koebe 1/4-Theorem}{Koebe 1/4-Theorem}
\newtheorem*{Riemann - Hurwitz formula}{Riemann - Hurwitz formula}
\newtheorem*{Series Law}{Series Law}
\newtheorem*{Parallel Law}{Parallel Law}
\newtheorem*{First  Caratheodory Theorem}{Carath\'eodory Boundary Theorem}
\newtheorem*{Second Caratheodory Theorem}{Carath\'eodory-Torhorst Theorem}
\newtheorem*{Conformal Schonflies Theorem}{Conformal Sch\"onflies Theorem}
\newtheorem*{Mean Value Property}{Mean Value Property}
\newtheorem*{Max/Min Principle}{Maximum/Minimum Principle}
\newtheorem*{Maximum Principle}{Maximum Principle}
\newtheorem*{Poisson Formula}{Poisson Formula}
\newtheorem*{la-lemma (Extension)}{First $\la$-lemma}
\newtheorem*{qc la-lemma}{Second $\la$-lemma}
\newtheorem*{Perron-Frobenius Theorem}{Perron-Frobenius Theorem}
\newtheorem*{Thurston Realization Theorem}{Thurston Realization Theorem}

\newtheorem*{Quasi-Additivity Law}{Quasi-Additivity Law}
\newtheorem*{General Quasi-Additivity Law}{General Quasi-Additivity Law}
\newtheorem*{Covering Lemma}{Covering Lemma}  

\def\note#1
{\marginpar
{\nt $\leftarrow$
\par
\hfuzz=20pt \hbadness=9000 \hyphenpenalty=-100 \exhyphenpenalty=-100
\pretolerance=-1 \tolerance=9999 \doublehyphendemerits=-100000
\finalhyphendemerits=-100000 \baselineskip=6pt
#1}\hfuzz=1pt}

\usepackage{epsfig,amsfonts,amssymb,color,amsmath,amscd,url}
\usepackage{graphicx}
\usepackage[all,cmtip]{xy}
\usepackage{import}
\usepackage{chngcntr}
\counterwithout{section}{chapter}
\usepackage{hyperref}
\usepackage{color}

\title[groups and dynamics] {
  Self-similar groups and holomorphic dynamics:\\
Renormalization, integrability, and spectrum. \\
 \texttt{ \today} } 
\date{\today}
\author {N.-B. Dang \\  Nguyen-Bac.Dang@stonybrook.edu \\ Stony-Brook University \and R. Grigorchuk \\
 grigorch@math.tamu.edu \\
 {Texas A\&M} \and M. Lyubich  \\
mikhail.lyubich@stonybrook.edu \\
 Stony-Brook University}

\begin{document}



\begin{abstract}
  In this paper, 
we explore the spectral measures  of the Laplacian on Schreier graphs for several self-similar groups
(the Grigorchuk,  Lamplighter, and Hanoi  groups)
from the dynamical and algebro-geometric  viewpoints.
For these graphs,  classical Schur renormalization transformations
act on appropriate  spectral parameters as  rational maps in two variables.  We show that the
spectra in question can be interpreted as asymptotic distributions of
slices by a line of  iterated pullbacks
of certain algebraic curves under the corresponding  rational maps
(leading us to a notion of  a  \textit{{spectral current}}).
We follow up with 
a dynamical criterion for discreteness of the spectrum. 
In case of discrete spectrum,
the precise rate of convergence of finite-scale approximands
to the limiting spectral measure is given. 
For the three groups under consideration,
the corresponding rational maps happen to be fibered over   polynomials in one variable.
We reveal the algebro-geometric nature of this integrability phenomenon.
 \end{abstract}

\maketitle

\setcounter{tocdepth}{1}
 
\tableofcontents

\section{Introduction}

Spectral theory of Laplacian is a classical area of Mathematical Physics,
with deep connections to Geometry, Probability, Dynamics, Geometric Group Theory,
and Number Theory. From the point of view of Quantum Mechanics,
it describes the observable energy spectrum of a free particle moving in the
space under consideration. In this interpretation, the dichotomy between  discrete and
continuous spectrum roughly corresponds to the difference between
insulating and conducting states of matter.

In a series of works by Bartholdi, Grigorchuk, Nekrashevich, Suni\'c, Zuk, 
and others
\cite{
grigorchuk_zuk_99,grigorchuk_nekrashevych_self_similar_operator,
bartholdi_grigorchuk_hecke,grigorchuk_sunic_hanoi,
grigorchuk_nekrashevych_sunic_self_similar,
grigorchuk_leemann_nagnibeda,
dudko_grigorchuk,
grigorchuk_yang,
grigorchuk_lenz_nagnibeda_schroedinger}
over the past 20 years,  the spectral problem for Cayley,
and more generally Schreier graphs
has been explored  for 
discrete self-similar groups.
Homogeneity and self-similarity of the corresponding spaces
leads  to invariance of the spectrum under
{\em Schur Renormalization transformations}, which sometimes happen to be  rational maps in two variables.
This allowed the authors to describe the spectrum of the corresponding Schreier graphs 
in three remarkable cases:  
the {\em Grigorchuk $\mathcal{G}$,  Lamplighter $\mathcal{L}$}, and  {\em Hanoi} $\mathcal{H}$ group. In particular, the spectrum turned out to be absolutely
continuous in the former case and discrete in the latter two. 

%

In this paper we bring ideas from Holomorphic Dynamics and Algebraic Geometry 
to give a new insight into the above spectral phenomena.
Namely, we take a close look at the dynamics of the corresponding
renormalization transformations and relate the spectral results
to the {\em equidistribution theory} for dynamical pullbacks of  holomorphic
curves.
We also analyze the nature of {\em integrability}
of these transformations (that happen to be  fibered over 
one-dimensional maps). In particular,  we give a general algebro-geometric
criterion for integrability (in the spirit of Diller and Favre
\cite{diller_favre}) that can be applied to each of the groups in question.  
This allows us to put all the previous results in a general framework.



To set up the renormalization scheme (for the above three self-similar groups),
one needs to introduce an extra spectral parameter
and the corresponding two-parameter pencil of operators.
In the  $n$th scale this pencil is reduced to a pencil of matrices of size
$d^n$ (where $d=2$ for the groups $\mathcal{G}$ , $\mathcal{L}$ and $d=3$
for $\mathcal{H}$). 
Letting $P_n \in \C[\la,\mu]$  be the characteristic polynomial of
that matrix pencils, one obtains the following spectral relation
between two consecutive scales:
\begin{equation} \label{eq_main_formula}
P_{n+1}  = Q^{d^{n}} \cdot ( P_{n} \circ R ), 
\end{equation}
where $Q\in \C[\la, \mu]$, $P_ 0$ is linear,
and  $R\colon \C^2 \dashrightarrow  \C^2$ is the renormalization rational map. 
For the  groups $\mathcal{G}$, $\mathcal{H}$ and  $\mathcal{L}$,
the transformation $R$ is given by the following explicit expressions,   respectively:  
\begin{equation*}
R_{\mathcal{G}} (\lambda, \mu ):=    \left( \frac {2\la^2} {4-\mu^2}, \
    \mu+ \frac  {\mu \la^2} {4-\mu^2} \right),
\end{equation*}
\begin{equation*}
R_{\mathcal{L}}(\lambda,\mu) := \left (- \dfrac{\lambda^2 - \mu^2 -2 }{\mu - \lambda} , -\dfrac{2}{\mu - \lambda} \right ),
\end{equation*}
\begin{equation*}
R_{\mathcal{H}}(\lambda,\mu):=  \left( \lambda + \dfrac{2\mu^2 (-\lambda^2 + \lambda + \mu^2)}{ (\lambda -1 -\mu )(\lambda^2 -1 + \mu -\mu^2 )} , \dfrac{\mu^2 (\la-1+\mu)}{(\la -1 -\mu )(\la^2 -1 + \mu -\mu^2 )} \right ).
\end{equation*}   
%
It shows a clear  connection between the spectral  algebraic
curves $\Gamma_n = \{ P_n = 0\}$ and iterated pullbacks of the initial line $\Gamma_0$  
by the rational map $R$.  Passing to a limit, we obtain  the {\em spectral current} of a 
  pencil of operators that consists of some  curves  accumulating on the Julia set of $R$ in the Hanoi case and Grigorshuk case or on the set of non-wandering points in the Lamplighter case. 
  
The desired spectral measure $\om$ for the Laplacian is the slice of this current by an
appropriate line $\{ \lambda = \const \}$ and is called the
\emph{density of states},
or is refered in \cite{grigorchuk_zuk_ihara,
  grigorchuk_nekrashevych_sunic_self_similar}
as the \emph{KNS spectral measure}
(after Kesten, Von Neumann and Serre).
 

 
The density of states for the groups $\mathcal{G}$, $\mathcal{H}$ and $\mathcal{L}$ (naturally acting on the corresponding regular trees)
were described in the papers
\cite{bartholdi_grigorchuk_hecke},
 \cite{grigorchuk_sunic_hanoi} and
 \cite{grigorchuk_zuk_lamplighter}
respectively. 
In this paper we will interprete these results
from the outlined dynamical viewpoint,
and,  for the latter two groups, give the rate of convergence
$\om_n\to\om$,
where $\om_n$ is the  counting measure
for the corresponding eigenvalues in $n$-th scale.

\medskip

\noindent \textbf{Theorem A}. \textit{
 The following properties hold.}

 \ssk \nin {\rm (i)}  
\textit{ The 
 density of states $\om$ associated with $\mathcal{G}$ is
  absolutely continuous with respect to the Lebesgue measure
(with an explicit density)
supported on the union of two intervals.}
\textit{This measure is the pushforward by $\mu \mapsto (\mu+1)/4$ of the slice of the Green current of the renormalization
map $R_\GG$ by the appropriate line. Its support is the image by the above affine transformation of the slice of the
Julia set of $R_\GG$ (see below). }

\ssk \nin {\rm (ii)}\textit{ The 
density of states  associated with the  group $\mathcal{L}$ is
atomic, and}
$$
     \omega_n - \omega \sim - n/2^{n-1}\, m 
$$
\textit{where $m$ is the Lebesgues measure on the interval. 
This interval (equal to the spectrum of $R_\LL$)  is the slice 
of the ``elliptic cylinder''   (see below)  
of the renormalization transformation $R_\LL$. }

\ssk \nin {\rm (iii)} \textit{The 
density of states  associated with the group $\mathcal{H}$ is atomic
as well, and}
$$
     \om_n  - \om \sim - (2^n/3^{n-1}) \, m  
$$
\textit{where $m$ is a Bernoulli measure on a Cantor set $K$. 
This Cantor set is the slice  of the Julia set   (see below) 
of the renormalization transformation $R_\HH$
by the appropriate line. 
Moreover, the spectrum of $\HH$ consists of a countable set of eigenvalues
accumulating on $K$. }

\bigskip 

\begin{rem} 
It turns out that the Hanoi group can be realized as the iterated monodromy group of the rational function $z^2 + 16 /27 z $ whose Julia set is homeomorphic to the Sierpinski gasket \cite{grigorchuk_sunic_hanoi}. 
A similar notion of density of states was defined for various fractal sets in \cite{strichartz,kigami}. 
It has been intensely studied for the Sierpinski gasket \cite{rammal,kigami_harmonic,fukushima_shima,teplyaev_gasket} and other fractals  \cite{sabot_electrical,sabot_smf,teplyaev_rodgers_basilica}. 
It would be interesting to explore if there is a direct connection between assertion (iii) and these results.
\end{rem}

\begin{rem}
  The Julia set for the map $R_\GG$
(and for closely related map $R_\DD$ for the infinite dihedral  group)
  was independently studied by
  Goldberg and Yang \cite{goldberg_yang} (see the discussion in \S~ \ref{section_fatou} for more details).
\end{rem}  

%


 
What makes these results 
quite easy from the dynamical viewpoint
is the ``integrability'' of the corresponding renormalization transformations.
The respective  integrals were explicitly given in
\cite{bartholdi_grigorchuk_hecke,
  grigorchuk_zuk_lamplighter,grigorchuk_sunic_hanoi} and communicated to us privately by Vorobets \cite{vorobets_notes}.
They lead to the following simple dynamical  models:

\medskip

\noindent \textbf{Theorem B}. 
%
 {\rm (i)} 
\textit{Near the Julia set, the Grigorchuk  map $R_{\mathcal{G}}$ is
 conjugate via a biholomorphic map to the following  direct product:}
\begin{equation*}
(\lambda,\mu) \mapsto (\lambda, \mu^2).
\end{equation*}
\textit{In this model, the Julia set of $R_ \GG$ is equal to the direct product $\C\times \T$.
The original Julia set of $R_\GG$  is foliated by complex conics parametrized by an interval.
}

\ssk\nin {\rm (ii)}   
\textit{\cite{grigorchuk_zuk_lamplighter}
The Lamplighter  map $R_{\mathcal{L}}$ is  conjugate via an invertible rational map to the following skew product:}
\begin{equation*}
(\lambda, \mu) \mapsto \left (\lambda ,  \dfrac{ \lambda \mu - 4}{\mu}\right ).
\end{equation*}
\textit{In this model, the recurrent part of the dynamics is supported by 
the fixed points locus  and the elliptic cylinder.
}

\ssk\nin {\rm (iii)}  
\textit{The Hanoi map $R_{\mathcal{H}}$ is conjugate via an invertible rational map  to the following skew product: }
\begin{equation*}
(\lambda,\mu) \mapsto \left (\la^2 - \la - 3 ,  \dfrac{(\la-1)(\la+2)}{\la+3} \mu \right ).
\end{equation*}
\textit{In this model, the Julia set of $R_\HH$ is equal to the product of the
Julia set of $\la^2-\la-3$ (which is a hyperbolic Cantor set)   times $\C$.}

\bigskip



\begin{rem}
 In \cite{grigorchuk_zuk_spectral},
 the authors asked whether the spectrum is also discrete in the case of the  {\em Basilica} group, which is the iterated monodromy group for  the Julia set of $z^2-1$.
 Though  the corresponding renormalization
is not integrable,  our criterion for discreteness  (formulated below)  is still
applicable due  to the fact that the  dynamical degree
(calculated by Eric Bedford)
turns out to be non-integer in this case.  
We will discuss it in a forthcoming paper (this problem was independently studied in \cite{teplyaev_rodgers_basilica}).
\end{rem}


\msk
Our main focus in this paper is to analyze the nature of this
integrability phenomenon,
i.e., to identify from general principles invariant fibrations for the
maps under consideration. Note with this respect, that
though 
meromorphic surface maps  preserving fibrations are classified 
(see \cite{diller_favre,dabija_jonsson,favre_pereira},
there is no general method of identifying  an invariant fibration for a
{\em given}  non-invertible rational surface map.   

We provide two ways to identify 
the above fibrations:

\ssk\nin $\bullet$ 
Either by  considering some explicit invariant  pencils 
of conics passing through special points of the maps $R =R_{\mathcal{G}}, R_{\mathcal{L}}, R_{\mathcal{H}}$,
namely certain points of indeterminacy  
and certain fixed/prefixed points;

\ssk\nin $\bullet$ 
Or else, by means of a systematic algebro-geometric approach.

\ssk \nin   
In the latter approach, inspired by \cite{gizatullin,cantat_k3,diller_favre},  
we   develop
an algebraic  criterion  to detect  presence of an invariant (rational) fibration
and give a method to calculate 
an explicit semi-conjugacy. 
Let us explain briefly the ideas behind our criterion.
To construct an explicit semi-conjugacy, one has to find 
a rational map  $\pi \colon \C^2 \to \C$ which  semi-conjugates  $R$ to a
one dimensional map.
To this end  
we apply some ideas from the Minimal Model Program \cite{kollar_mori}
which provides a setting 
in which one can contract a rational curve to a point.
A natural condition,  due to Mori,
is to ask that those curves intersect negatively
the first Chern class of the canonical bundle in our space. 
When this happens, we obtain a map $\pi\colon \C^2 \to \C^k$ where $k$ is
either $0$, $1$ or $2$. We then add an additionnal condition on the
contracted curves so that the Riemann-Roch-Hirzebruch formula rules out the cases $ k= 0,2$.

In order to state our result, we interpret the contracted curves as
particular holomorphic sections of a holomorphic line bundle whose
first Chern class is cohomologically equivalent to the integration
along these sections and compute the intersection of classes as a
cup-product in the deRham cohomology of $\P^2$.

\medskip

\noindent \textbf{Theorem C}. \textit{Let $R \colon  \P^2 \dashrightarrow \P^2$ be a dominant rational map\footnote{whose image is not contained in a curve}. 
Suppose that there exists a surface $X$  obtained from $\P^2$ by a finite sequence of blow-ups of $\P^2$,
an integer $d \geqslant 1$,
and  a line bundle $L$ on $X$ whose first Chern class $c_1(L) \in H^2(X,\Z)$ satisfies the following conditions.}

\ssk \nin {\rm (i)} 
$c_1(L) \cdot c_1(L) = 0$ in $H^4(X,\Z) \simeq \Z $. 

\ssk \nin {\rm (ii)}  
  \textit{For any curve $C$ on $X$, the intersection $[C] \cdot c_1(L)  \in H^{4}(X,\Z) $ is non-negative where $[C]$ denotes the cohomology class in $H^2(X,\Z)$ induced by $C$.
}

\ssk \nin {\rm (iii)} 
\textit{$ c_1(L) \cdot K_X < 0$ in $H^{4}(X,\Z)$ where $K_X$ is the first Chern class of the canonical bundle on $X$.
}
\ssk \nin {\rm (iv)}  
\textit{The pullback of the line bundle $R^* L$ by $R$ is isomorphic to the line bundle $ L^{\otimes d}$.}

\ssk 
\textit{Then the rational map $R$ is rationally  semi-conjugate to a degree $d$ rational map on a curve.}
\bigskip 
  
Theorem C produces a  particular semi-conjugacy whose fibers are rational curves.
Our proof follows closely the (non-dynamical) construction of a contraction morphism on a ruled surface \cite[Theorem 1.28 (2)]{kollar_mori}.   
 We then show that our criterion applies to the three maps, $R_{\mathcal{G}}, R_{\mathcal{L}}, R_{\mathcal{H}}$,
 under consideration.

\bigskip

Once Theorem B is proved, then one proves successively the two assertions of Theorem A.

For the first assertion, $R = R_{\mathcal{G}}$, $d=2$, $P_0= 2 - \lambda - \mu $,
and we interpret the density of states 
$\omega_{\mathcal{G}}$ associated with the  group $\mathcal{G}$  as the limiting measure given by:
\begin{equation*}
\omega_\mathcal{G} = \lim_{n\rightarrow +\infty} \dfrac{1}{2^n} R_{\mathcal{G}}^{-n} \{ P_0 =0 \} \cap \{\lambda = -1 \},
\end{equation*}  
where the intersection $R_G^{-n} \{ P_0 =0 \} \cap \{\lambda = -1 \}$ is the counting measure on the line $\lambda=-1$. 
The above formula shows that the convergence to the density of states is
related to the behavior of the  iterated preimage of the curve $\{P_0=0\}$ by $R_{\mathcal{G}}^n$
which is a classical equidistribution problem in
the two-dimensional holomorphic dynamics
(\cite{bedford_smilie_polynomial_diffeo_currents,
russakovskii_shiffman,favre_jonsson_brolin,%
dinh_sibony_equidistribution_green_current,bleher_lyubich_roeder}).

In our situation, as $R_{\mathcal{G}}$  is semi-locally conjugate to
a simple model $\id\times z^2$,
it is easy to show directly that the sequence of curves
\begin{equation*}
\dfrac{1}{2^n} R_{\mathcal{G}}^{-n}\{ P_0 = 0\}
\end{equation*}
converges towards the Green current of $R_{\mathcal{G}}$,
while their slices converge to the corresponding transverse measure. 
(For more general results of this kind 
see \cite{dujardin_bifurcation,chio_roeder}.) 
We recover directly the so-called ``joint spectrum''
of a particular pencil (\cite{goldberg_yang}) by looking at the support of the Green current and our approach using currents gives a quantitative way to measure this set. In this case, one finds that the spectral current $T_{\mathcal{G}}$ is supported on the union of hyperbolas:
\begin{equation} \label{eq_current_grigor}
T_{\mathcal{G}}:= \int_{-1}^{1} [4-\mu^2+\la^2 - 4\theta \, \la =0] \  \dfrac{d\theta}{\pi \sqrt{1 - \theta^2}} ,
\end{equation}
where $[4-\mu^2+\la^2 - 4\theta \, \la =0]$ denotes the current of integration on the corresponding hyperbola.

For the Lamplighter and Hanoi group, their spectral currents $T_\mathcal{L}$ and $T_{\mathcal{H}}$ are both supported on a countable union of curves (instead of a continuum) and one obtains an asymptotic expansion:
\begin{equation*}
T_\mathcal{L} = T_{n, \mathcal{L}} + \dfrac{n}{2^n} \int_{-2}^{2} [\lambda + \mu = \eta] \   \dfrac{ 2d\eta}{\sqrt{4 - \eta^2}} + o \left ( \dfrac{n}{2^n} \right ),
\end{equation*}
\begin{equation*}
T_\mathcal{H} = T_{n , \mathcal{H}} + \dfrac{2^n}{3^n} \int_{\mathcal{J}(p)} [\lambda^2 - 1 - \lambda \mu - 2 \mu^2 = \eta \mu] \  dm_p(\eta) + o \left ( \dfrac{2^n}{3^n} \right ),
\end{equation*}
where $T_{n,\mathcal{L}}, T_{n , \mathcal{H}}$ are some currents supported  on $2^n$ and $3^n$ curves respectively and $m_p$ is the measure of maximal entropy associated to the polynomial $p = z^2 - z -3$.
\bigskip

%

The proof of the second and third statement of Theorem A is also of dynamical nature. 
The fact that the spectrum is atomic follows from a discrepancy
between the branching degree  $d$ of the regular tree
$T$ under consideration
and the {\em  first dynamical degree} of the
renormalization transformation $R_{\mathcal{L}}, R_{\mathcal{H}}, R_{\mathcal{B}}$, respectively.

The first dynamical degree $\lambda_1(R)$ is defined formally as:
\begin{equation*}
\lambda_1(R) :=\limsup_{n\rightarrow +\infty} (\deg R^n)^{1/n},
\end{equation*}
(where $\deg R^n$ denotes the algebraic degree of $R^n$)
and measures the growth  of the degree of the iterated  preimages of
generic algebraic  curves.

 For the Lamplighter group, $\lambda_1(R_{\mathcal{L}}) = 1$ whereas $d=2$,
 and for the Hanoi group, $\lambda(R_{\mathcal{H}}) = 2$ whereas $d=3$.
To understand the spectral measure, we expand the inductive formula \eqref{eq_main_formula} into:
\begin{equation*}
P_n = \left (\prod_{i=0}^{n} Q^{d^{n-1-i}} \circ R^i \right ) P_0 \circ R^n.
\end{equation*}
 Observe that there are two different contributions for the growth of the degree of $P_n$,
 one from the power of $d$ and the other from the iteration of $R$.
 We then show that  when $\lambda_1(R) < d$, then the function $1/d^n\log |P_n|$  converges  to a non-constant function
 which is equal to $-\infty$ on countably many curves
 making the density of states atomic.
 
\medskip

\noindent \textbf{Theorem D}.\textit{ Take $R\colon \C^2 \dashrightarrow \C^2$ a dominant rational transformation and take $P_n \in \C[\la, \mu]$ some polynomials of degree $d^n$ in the variable $\mu$  satisfying:}
\begin{equation*}
P_{n+1} = Q^{d^{n-k}} \cdot ( P_n \circ R ),
\end{equation*}
\textit{for all integer $n$, where $Q$ is a fixed polynomial and where $k=0,1,2$.
If $\lambda_1(R) < d$, then for any $\lambda_0 \in \C$, any limit point of the sequence of probability measures $ 1/d^n [P_n(\lambda_0, \cdot) =0]  $ is atomic.}
 
\medskip

This second statement and its proof are reminiscent of the Dichotomy Theorem by Sabot (\cite[Theorem 4.1]{sabot_spectral}),
who observed a similar phenomenon for different rational maps arising from the study of  the spectrum of the Laplacian for a class of self-similar sets. 
In our setting, Theorem D applies to the Lamplighter and Hanoi group and shows that the density of states associated to these two groups is atomic.




\begin{rem}
  This project originated at a conference in Saas-Fee in March 2016 
as a discussion (nicknamed ``Saas-Fee nightmares'') of the dynamical
interpretation of the density of states for  $\mathcal{G}$.
It was obtained shortly afterwards and was announced at a conference in the
Fields Institute in May 2019. The Lamplighter and  Hanoi groups were
studied later; the corresponding results were announced at a Luminy meeting
in January 2020. 
\end{rem}  

\section*{Acknowledgements}
The first two authors were partly supported by the Simons Foundation at the
IMS at Stony Brook and are also thankful to Dinh-Sibony for their references on their general slicing techniques. 
The third author was partly suppported by the Hagler Institute for Advanced Study and the NSF grants  DMS-1600519 and 1901357.


\section{Background on spectra of graphs and groups}


\subsection{General spectral theory}

The study of spectral properties of operators on groups and graphs is very interesting and important.
There are hundreds (if not thousand) of articles on  spectra of finite graphs (including such topics as expanders and Ramanujan graphs)
and many 
books on that subject. 

By \textbf{spectrum of a graph} $\Gamma =(V,E)$, one means the spectrum of the \textbf{Laplace operator} $L$. 
In the case where $\Gamma$ is a $d$-regular graph, then $L = I - M$ where $M = A/d$ and where $A$ is the adjacency operator (or matrix) on the vertices of $\Gamma$. 
The operator $M$ is called the \textbf{Markov operator} and corresponds to a simple random walk with uniform transition probability $1/d$ along each edge of $\Gamma$. 
One can also consider a more general concept of weighted Markov or Laplace operators when a weight $w \colon E \to \R^+$ is given.   
The weighted Laplacians are also used in various situations.

By \textbf{spectrum of a group $G$ with a system of generators $S$}, one means the spectrum of the Cayley graph $\Gamma(G,S)$.

If $G$ is finite then one may try to use the information  about irreducible unitary representations of $G$ (although this approach is often not easy to implement).
In the case of Cayley graphs or their generalization, Schreier graphs,  one chooses the weight $w \colon S \cup S^{-1} \to \R^+$ so that it is \textbf{symmetric} $w(s) = w(s^{-1}) \  \forall s\in S$. 
The symmetry of the weight is needed to keep the weighted Laplacian  $L_w$ self-adjoint.

The case of infinite graphs or groups is much harder and little is known about their spectral properties.
However,  a big progress was achieved for self-similar groups and their associated  Schreier graphs.
We give here some background in this setting.

\medskip

Let $M$ be a Markov operator on $d$-regular connected infinite graph $\Gamma = (V,E)$. It is a self-adjoint operator of norm bounded by $1$,
so its spectrum is contained in the interval $[-1, 1]$. 
By the spectral theorem for bounded self-adjoint operators, there exists a projection valued measure $P$ defined on Borel subsets of $\R$ which plays the role of the diagonalization basis of $M$ (see for instance \cite[Chapter VII]{reed_simon}). 
To each vertex $v \in V$, one associates the probability measure defined by:
\begin{equation*}
\mu_v (B)= \langle \delta_v , P(B) \delta_v \rangle, 
\end{equation*}
where $\delta_v$ is the delta function at the vertex $v$ and where $B \subset \R$ is any Borel subset.
The moments of this measure, 
$$ \int_{\R} \lambda^n d\mu_v(\lambda)= < M^n \delta_v, \delta_v> $$  
coincide  with the probabilities of returns to  $v$ 
for the random walk induced by $M$.
It was proved by Kesten (\cite[Lemma 2.1]{kesten}) that the support of $\mu_v$ 
coincides with the spectrum  of $M$ when $\Gamma$ is a Cayley Graph.

These spectral measures are often hard to determine,
so they were computed in the rare cases:
finitely generated free abelian and non-abelian groups are among those  \cite{kesten}.
For example, the spectral measure associated to the free abelian group $\mathbb{Z}^n$ is absolutely continuous with analytic density and has support in the interval $[-1,1]$. 
Its density is the pushforward of the Haar measure on the torus $\R/\Z^n$ by the map $(\theta_1, \ldots , \theta_n) \mapsto  (1/n) \sum_{i=1}^n \cos(\theta_i)$.
  Kesten showed that the free group generated by $h$ elements admits a spectral measure which is absolutely continuous, has analytic density and is supported in $[- \sqrt{2h-1}/h , \sqrt{2h-1}/h]$. Moreover, the density of this spectral measure is given by:
  \begin{equation*}
  \dfrac{\sqrt{2 h -1 -x^2h^2}}{1- x^2} dx. 
\end{equation*}   
 


\subsection{Self-similar groups} \label{section_self_similar}

The idea of self-similarity came to group theory at the beginning of 1980th in the relation to the Burnside problem on periodic group and Milnor's question on existence of groups of intermediate growth \cite{grigorchuk_80,grigorchuk_83,grigorchuk_84}. 
The first examples of self-similar groups were presented in dynamical terms, namely as groups acting on the interval $[0,1]$ or on the square $[0,1] \times [0,1]$ by Lebesgue measure preserving transformations. 
Later on, along with the development of the algebraic background of the theory of self-similar groups,
stimulating relations to  various themes in dynamical systems,
statistical mechanics, and mathematical physics
(including symbolic and holomorphic dynamics,
random Schrödinger  operator, invariant random subgroups, etc.) 
were revealed \cite{psim,just_infinite,BGS03,solved,grigorchuk_lenz_nagnibeda_subshift,
grigorchuk_lenz_nagnibeda_schroedinger,dudko_grigorchuk_diagonal}. 
%
%

Initially used for  resolving various difficult
problems in Algebra and Functional Analysis
(e.g., non-elementary amenability),
they were found later to be  naturally connected to some well-known and popular  games like the Chinese puzzle or Hanoi Towers game. 
Moreover, they can be seen from the analysis of Gray code, automatically generated sequences (like for instance Thue-Morse
sequence), Julia sets of  one dimensional polynomials (like the Basilica Julia set \cite{grigorchuk_zuk_spectral}),  higher dimensional
holomorphic dynamics, etc.
As we have already  mentioned above, the latter connection comes from  
the non-cyclic  
renormalization relating various scales of the group. 

\medskip

A self-similar group naturally acts on a regular rooted $d$-regular tree and this action respects the self-similar structure of the tree. 
Namely, for each element $g \in G$ and a vertex $v \in V(T)$, the restriction $g_v$ of $g$  on the subtree $T_v$ rooted at $v$ can be identified with an element of $G$ (using the canonical identification of $T_v$ with $T$). 
There are modifications of this definition that lead to the classes of self-replicating (or recurrent) groups, branch groups, etc. 
An account of the theory of self-similar group
can be found in the surveys
\cite{grigorchuk_nekrashevych_sunic_self_similar,G00,BGS03} and in Nekrashevych's book \cite{nekrashevych_self_similar}.


There are two main ways to describe the action of self-similar groups on such tree: either via wreath recursion or via Mealy automata.
\smallskip

Fix $d \geqslant 2$  an integer and let $T= T_d$ be the $d$-regular rooted tree whose vertices are in bijection with finite words (strings) over an alphabet of cardinality $d$ (a standard choice for $A$ is $\{ 1 , \ldots , d \}$). 
The ordering on each level is given by the lexicographic order (see Figure \ref{figure_tree} below). 

\begin{figure}[h!] 
\includegraphics[scale=0.5]{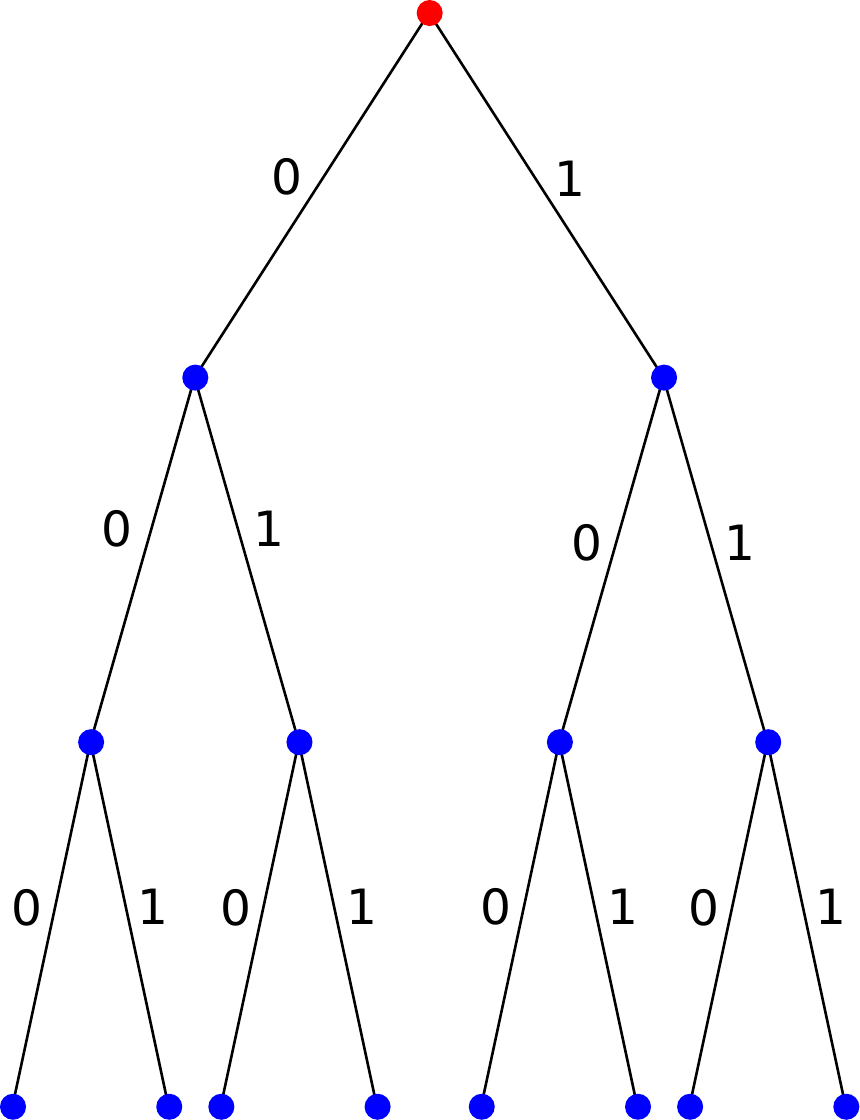}
\caption{\label{figure_tree}Picture of a 2 regular rooted tree}
\end{figure}  

The boundary of the tree, denoted $\partial T$ consists of geodesic paths joining the root with infinity. It can naturally be identified with the set $\{ 1 , \ldots , d \}^{\mathbb{N}}$, endowed with the product topology and the resulting topological space is homeomorphic to a Cantor set. 

The \textbf{group of automorphism}  $\Aut(T)$ of the (rooted) tree $T$ consists of bijection of the set of vertices of $T$ which fix the root and preserve adjacency relations.
For any vertex $v$ of $T$, let $T_v$  be the $d$-regular subtree of
$T$ whose root is $v$. There is a canonical isomorphism between $T_v$
and $T$ which is induced by a power of the left shift  $\shift$ on  the symbolic
space $\Si_d^+= \{ 1 , \ldots , d \}^{\mathbb{N}}$.   

Every automorphism $g \in \Aut(T)$ can be described  by the following data:
an element of the symmetric group $\sigma \in \Sym(d)$ which corresponds to  the restriction of $g$ on the first level of the tree, and a $d$-tuple $(g_1, \ldots , g_d )$ of elements of $\Aut(T)$ called \textbf{sections at the vertices of the first level} which encodes how $g$ acts on each rooted subtree $T_1, \ldots, T_d$ with a root at level $1$ using the canonical identification of $T_1, \ldots , T_d$ with $T$. 
More precisely, for any word $w \in \Si_d^+$,
$g_i (w ) =\shift \circ g ( \overline{i w} )$,
where $\overline{i w}$ is the concatenation of the number $i$ with the
word $w$.

Using this description, we obtain an isomorphism:
\begin{equation*}
\psi : \Aut(T) \to \Aut (T)^d \rtimes \Sym(d),
\end{equation*}
where the sign $\rtimes$ stands for the semi-direct product of groups and where $\Sym(d)$ acts on the direct product $\Aut(T)^d$ by permuting  the factors.

Now, let $G$ be a subgroup acting faithfully on $T $ by automorphisms.
We can view $G$ as a subgroup of $\Aut(T)$ and consider the restriction of $\psi$ to $G$. 
When $\Im(\psi_{|G})< G^d \rtimes \Sym(d)$, we say that the group \textbf{$G$ is a self-similar group}.

Another way to express this is as follows:
A subgroup $G \leqslant \Aut(T)$ is self-similar if its sections $g_1, \ldots, g_d$ belong to $G$.

The relation $\psi(g) = (g_1, \ldots , g_d)\,  \sigma$ is called the \textbf{wreath recursion} and is usually denoted:
\begin{equation*}
g = (g_1, \ldots , g_d)\,  \sigma.
\end{equation*}

\begin{example} Take $A = \{ 1 , \ldots, d\}$ and the wreath recursion given by:
\begin{equation*}
a = (1 , \ldots, 1 , a)\,  \sigma,
\end{equation*}
where $\sigma$ is a cyclic permutation of $A$. 
The subgroup generated by $a$ is an infinite cyclic group which is an
algebraic realization of the odometer group (called also the adding machine).
\end{example}

For the next examples, we take $A= \{ 0,1\}$ and denote by $e, \sigma$ the identity and the standard involution in $\Sym(2)$ respectively.

\begin{example} 
Consider the wreath recursion,
\begin{equation*}
a = (1,1)\, \sigma,\quad b = (a , b)\, e.
\end{equation*}
The subgroup $\langle a, b \rangle$ is isomorphic to $D_\infty$, the
infinite dihedral group (see \cite{psim}).  
\end{example}

We now present successively the three  self-similar groups of interest in this paper.
\begin{definition} Consider the wreath recursions
\begin{equation*}
a = (1,1)\,  \sigma,\quad b = (a, c)\,  e,\quad c= (a, d)\,  e , \quad d = (1, b)\,  e.
\end{equation*}
The subgroup $\mathcal{G} = \langle a,b,c,d \rangle$ is the \textbf{first Grigorchuk group} (\cite[Section 4.1]{bartholdi_grigorchuk_hecke}).
\end{definition}

\begin{definition} Consider the wreath recursions
\begin{equation*}
a = (b, a)\,  \sigma , \quad b = (b,a ) \,  e .
\end{equation*}
The subgroup $\langle a , b \rangle$ is the \textbf{Lamplighter group}
(\cite[Section 5]{grigorchuk_zuk_lamplighter}), it is the wreath product $\Z_2 \wr \Z$
and is isomorphic to the semidirect product $$  \left (\underset{\Z} {\oplus} \ \Z_2 \right ) \rtimes \Z,$$
where a generator $a$ of $\Z$ acts on
$\displaystyle {\underset{\Z}   {\oplus} \, \Z_2 } $ as the shift map.
\end{definition}


\begin{definition} Consider $A = \{ 0, 1, 2 \}$ and the wreath recursions
\begin{equation*}
a= (1,1,a) \,  \alpha,\quad b= (1, b, 1)\,  \beta,\quad c= (c, 1 ,1)\,  \gamma,
\end{equation*}
where $\alpha = (01), \beta = (0 2), \gamma = (12)$ are the three involutions in $\Sym(3)$.
The subgroup $\mathcal{H} =\langle a, b, c \rangle$ is called the \textbf{Hanoi tower group} and is associated to the  Hanoi towers game on $3$ pegs (\cite{grigorchuk_sunic_hanoi}). 
\end{definition}


The groups in the above examples are not only self-similar groups but
they are groups with finite self-similar set of generators (\cite{grigorchuk_nekrashevych_sunic_self_similar}). 
%
%
%
\subsection{Spectra of self-similar groups and density of states}
\label{section_spectra}
We have explained  in the previous section how a self-similar group  acts  on a $d$-regular rooted  tree $T$. 
Moreover, given a group $G \leqslant \Aut(T)$ with a finite generating set $S$, one associates  a sequence of finite graph
$\Gamma_n = (V_n , E_n)$ where $n=1,2, \ldots$, $|V_n| = d^n$,
and an uncountable family of graphs $\{\Gamma_\xi = (V_\xi, E_\xi)\}_{\xi \in \partial T }$ where $V_\xi$ is the $G$-orbit of the point $\xi$. 

The vertices of the  graphs $\Gamma_n$ and $\Gamma_\xi$
are level $n$ vertices of $T$ and points in the $G$-orbit of $\xi$, respectively,
and the edges are pairs of vertices of the form $(v, s\cdot v)$ where $s \in S$. 
Usually, all graphs $\Gamma_\xi$ are infinite (for instance when  the $G$ action on $T$ is transitive on each level) and they are natural limits of the graphs $\Gamma_n$.
Namely,

\begin{equation*}
(\Gamma_\xi , \xi) = \lim_{n} (\Gamma_n , v_n),
\end{equation*}
where $v_n$ is the vertex of level $n$ on the  geodesic path representing $\xi $,
and  $(\Gamma_\xi , \xi)$, $(\Gamma_n , v_n)$ are the corresponding  pointed  
graphs .
The convergence above is taken in the usual way: for all $R > 0$, the balls $B_{V_n}(R)$ of radius $R$ in $(\Gamma_n, v_n)$ converge to $B_\xi(R)$.
This leads to the idea of approximating $\spe(\Gamma_\xi)$ with $\spe(\Gamma_n)$. 

The first observation is made in \cite{grigorchuk_zuk_99}: 
let us take $\xi$ a point on the boundary of the tree, $v_n$ a vertex of level $n$ which belongs to the geodesic $\xi$, and fix a Markov operator $M$ on $l^2(G\xi)$ which induces a Markov operator on $l^2(G v_n) = l^2(\Gamma_n)$.
 If $\mu_n$ is the spectral measure  associated with the operator $M_n$ and with the delta function on the vertex $v_n$, then
\begin{equation*}
\lim_{n\rightarrow +\infty} \mu_n = \mu_\xi,
\end{equation*}
where $\mu_\xi$ is a spectral measure of a Markov operator on $\Gamma_\xi$ determined by the delta function $\delta_\xi \in l^2(G\xi)$.

Since each graph $\Gamma_{n+1}$  covers $\Gamma_n$ and is covered by $\Gamma_\xi$,
the spectrum set increases $\spe(\Gamma_n) \subset \spe(\Gamma_{n+1})$,
and  an easy argument (\cite{bartholdi_grigorchuk_hecke}) shows that
\begin{equation*}
\spe(\Gamma_\xi) \subset \overline{ \bigcup_{n=1}^{\infty}  \spe(\Gamma_n) }.
\end{equation*}
Moreover, \cite[Theorem 3.6]{bartholdi_grigorchuk_hecke} states that if the graph $\Gamma_\xi$ is  amenable then
\begin{equation} \label{eq_spectrum_exhaustion}
\spe(\Gamma_\xi) = \overline{ \bigcup_{n=1}^{\infty}  \spe(\Gamma_n) }.
\end{equation}
Recall that a graph $\Gamma_\xi$ is \textbf{amenable} if its Cheeger constant is $0$,
or equivalently $||M||=1$ (see \cite{harpe_grigorchuk}).
If a group $G$ is amenable then $\Gamma_\xi$ is amenable for all $\xi \in \partial T$.

The three groups studied in this paper,
the Grigorchuk group, the Lamplighter, and the Hanoi group,
are all amenable, so \eqref{eq_spectrum_exhaustion} applies to their Scheier graphs. 

To a finite graph $\Gamma$, one can associate the counting measure $\eta$ given by:
\begin{equation*}
\eta := \dfrac{1}{|V| } \sum_{\lambda_i \in \spe(M)} \delta_{\lambda_i},
\end{equation*} 
where $\lambda_i$ are the eigenvalues of the matrix $M$ counted with multiplicities.

If $\eta_n$ is the counting measure associated to $\Gamma_n$, we define $\eta$ as the weak limit of measures $\eta = \lim_n \eta_n$.
This measure is called the \textbf{density of states} (or  \textbf{KNS spectral measure}
 where the initials stand for Kesten, Von-Neumann, Serre).
 If $\rho$ is the uniform Bernoulli measure on $\partial T \sim \{0,1 ,\ldots,  d-1\}^\N$ and $\mu_\xi$ as before is the spectral measure associated to $\Gamma_\xi$
 with respect to the vertex $\xi \in \partial T$, then by \cite{grigorchuk_11}, one has
\begin{equation*}
\eta = \int_{\partial T} \mu_\xi d\rho(\xi),
\end{equation*} 
i.e the density of states is an average of the spectral measures $\mu_\xi$.
\medskip

To obtain the spectrum associated to the  Cayley graph of a group $G$,
an additional property is needed. 
Recall that the action of $G$ is \textbf{essentially free} with respect to the Bernoulli measure $\rho$ on $\partial T$ if $\forall g \in G \setminus \{1\}$, $\rho(\Fix(g))=0$ where $\Fix(g)$ denotes the set of fixed points of $\partial T$.
Equivalently, when $G$ is countable, this condition is equivalent to the property that the $G-$stabilizer of almost  any point $\xi \in \partial T$ is trivial.
 
Under this assumption,  $\eta$ coincides with the spectral measure associated with  $\delta_1 \in l^2(G)$. 
Thus, the computation of the density of states leads to the  spectrum of the Cayley graph of a group. 

\subsection{Schur renormalization transformations}

In this section, we will define some operators on finite matrices
called {\emph Schur complements}. These operators will allow us to
deduce inductively the spectrum of the Markov operator on the Schreier graphs
as one passes from one scale to another.

Take $H$ a finite dimensional vector space which can be decomposed as the sum of two non-zero subspaces $H = H_1 \oplus H_2$. 
If $M$ is an endomorphism of $H$, then $M$ can be expressed as a block-matrix according to this decomposition:
\begin{equation} \label{eq_block}
M := \left  ( \begin{array}{ll}
A & B \\
C& D
\end{array} \right ),
\end{equation}
where $A,D$ are  endomorphisms of $H_1$ and $H_2$ respectively, $C,D$ are linear transformations from $H_1 \to H_2$ and $H_2 \to H_1$ respectively.

\medskip

\begin{definition}
\begin{enumerate}
\item[(i)] Assume that $D$ is invertible, then the \textbf{first Schur complement}, denoted $S_1(M)$, is the endomorphism:
\begin{equation*}
S_1(M) := A - BD^{-1} C.
\end{equation*}
\item[(ii)] Assume that $A$ is invertible, then the \textbf{second Schur complement}, denoted $S_2(M)$, is the endomorphism
\begin{equation*}
D - C A^{-1} B.
\end{equation*}
\end{enumerate}
\end{definition}

The Schur complements are useful in our setting because they relate  the
invertibility of the Markov matrices in various scales, via the following classical result.

\begin{thm} (see e.g \cite[Theorem 5.1]{grigorchuk_nekrashevych_self_similar_operator}) Suppose that $D$ is invertible. Then $M$ is invertible if and only if $S_1(M)$ is invertible. Similarly, if $A$ is invertible then $M$ is invertible if and only if $S_2(M)$ is invertible.
\end{thm} 

In particular, we will exploit the relation between the determinant of $M$ and the Schur complement.
\begin{prop} \label{Schur relation}
  Suppose that $D$ is invertible, then
\begin{equation*}
\det(M) = \det(D) \det(S_1(M)).
\end{equation*}
\end{prop}  

Let us explain how the Schur complement arises in our study. 
We start with a sequence of vector spaces $H_n$ of dimension $d^n$ together with an identification:
\begin{equation*}
H_{n+1} = H_n \oplus \ldots \oplus H_n,
\end{equation*}
for all $n$, where the direct sum is taken $d$ times.
More precisely, $H_n$ will be the (Hilbert) space $l^2(V_n)$ where $V_n$ are the vertices of level $n$ of the rooted tree $T_d$ and each component in the decomposition of $H_{n+1}$ corresponds to the space of functions on a subtree.

\bigskip

For all the self-similar group $G$ treated in this paper, we will choose some generators $s_1, \ldots, s_k$, which are identified as operators on $H_n$  and we will consider  a pencil of operators (on $H_n$): 
$$M_n(z_1 ,\ldots , z_k) := z_1 (s_1 + s_1^{-1})  + \ldots + z_k (s_k + s_k^{-1}) $$  where $z= (z_1, \ldots , z_k) \in \C^k$  and where $s_i$ denotes the restriction of $s_i$ to $H_n$. 

The self-similarity of the action on the tree and Proposition \ref{Schur relation}
will lead to a relation of the form:
\begin{equation*}
\det M_{n+1}(z) = Q(z)^{d^{n-p}} \det(M_n(F(z))),
\end{equation*}   
where $p=0,1,2$, $F : \C^k \to \C^k$ is a rational map and $Q$ is a polynomial function on $\C^k$. 
The map $F$ is called the \textbf{renormalization map} associated with the spectral problem under consideration.

Under these conditions, we can now introduce the main notion of our paper.  
When it exists, we say that the limit of currents: 
\begin{equation*}
\lim_{n\rightarrow +\infty} \dfrac{1}{d^n} [\det M_n =0] ,
\end{equation*} 
is the \textbf{spectral current} associated to the group $G$, where $[M_n=0]$ denotes the current of integration on the zeros of the polynomial $\det (M_n)$ (see Appendix \ref{appendix_current}). 
\bigskip

Although we do not work directly in the case where the dimension of $H$ is infinite but let us yet explain how the renormalization map can be defined in this situation as well.
Assume again that we have a pencil $M(z)$ where $z\in \C^k$ of bounded linear operators on an infinite dimensional Hilbert space $H$. 
We define the \textbf{joint spectrum}, denoted $\jsp(M(z))$, as the subset:
\begin{equation*}
\jsp(M(z)) = \{ z\in \C^k | M(z) \text{ is not invertible} \}.
\end{equation*}
Let us consider 
$\varphi : H \to H\oplus \ldots \oplus H$ (called $d$-similarity)
where the direct sum is taken $d$-times,
a map $F : \C^k \to \C^k$  and a rational function 
$A : z \in \C^k \mapsto \mathcal{B}(H) $ with values in the space of bounded operators on $H$ such that for some $i\leqslant d$, one has:
\begin{equation*}
S_i (M(z)) = A(z) M(F(z)),
\end{equation*}
on a Zariski-open set of values of $z$.
In this case, the map $F$ is a renormalization map associated with the
problem of finding the joint spectrum $\jsp(M(z))$.
If we understand $\jsp(A(z))$ 
then  the spectral  problem for
$\jsp(M(z))$ gets reduced to a dynamical problem for $F$.

Observe that the support of the spectral current measures the locus of points  in $\C^k$ where the restriction of the operator $M(z)$ on finite dimensional subspaces is not invertible.  We thus  expect  
the support of the spectral current, when it exists, to be  equal to the joint spectrum of $M(z)$ when the group $G$ is amenable (this statement will appear in a forthcoming survey written by the second author and Tatiana Nagnibeda).

%
%
%
%
%
%

\bigskip

 \subsubsection{Schur transformations for the Grigorchuk group}
\label{section_schur_grigorchuk}


The self-similarity of the group $\mathcal{G}$ determines a morphism of algebra $\varphi : \C[\mathcal{G}] \to \M_2(\C[\mathcal{G}])$, where $\M_2(\C[\mathcal{G}])$ denotes the space of matrices with coefficient in the non-commutative group algebra $\C[\mathcal{G}]$.

We consider the pencil $M(\lambda, \mu) = - \lambda a + b + c + d - 1 - \mu  \in \C[\lambda,\mu][\mathcal{G}]$ . 

Denote by $t$ the element $(b+ c + d -1) /2$. Then $t$ and $a$ are
involutions and the recursion matrix associated to $M$
is precisely the matrix $\varphi(M(\lambda,\mu))$ given by:
\begin{equation*}
\varphi \circ M (\lambda, \mu) :=\left ( 
\begin{array}{ll}
2a - \mu & -\lambda\\
-\lambda & 2t - \mu
\end{array} \right ) .
\end{equation*}
Since $a$ and $t$ are involutions, one sees directly that the element $2a - \mu$ and $2 t - \mu$ are invertible in $\C[G]$.
The two Schur complements are given by :
\begin{equation*}
 S_1 \varphi \circ M (\lambda, \mu) = M (F(\lambda,\mu)),
 \end{equation*} 
 \begin{equation*}
 S_2 \varphi \circ M(\lambda,\mu) = M (G(\lambda, \mu)),
 \end{equation*}
where $F,G$ are the rational maps given by the formulas:
$$
     F(\la, \mu) = \left( \frac {2\la^2} {4-\mu^2}, \
    \mu+ \frac  {\mu \la^2} {4-\mu^2} \right).
$$

\begin{equation*}
G : (\lambda, \mu) \mapsto \left ( 2\dfrac{ 4 - \mu^2}{\lambda^2} ,\  - \mu \left (1 + \dfrac{4-\mu^2}{\lambda^2}\right )\right ).
\end{equation*}

\msk

\subsubsection{Schur transformations  for the Lamplighter group}
\label{section_schur_lamplighter}

The recursion for the Lamplighter group $\mathcal{L}$,
induces an algebra  morphism  $\varphi : \C[\mathcal{L}] \to
\M_2(\C[\mathcal{L}])$ as well.
%



Consider the following pencil of operators:
\begin{equation*}
M(\lambda, \mu) := a+ a^{-1} + b + b^{-1} - \lambda \Id - \mu \sigma,
\end{equation*}
where $\sigma =  b^{-1} a$ is the involution which exchanges the two
subtrees $T_1$ and $T_2$  introduced in Section \ref{section_self_similar}. 
The recursion matrix associated to $M$ takes the form:
\begin{equation}
\varphi\circ M (\lambda, \mu) := \left (\begin{array}{ll}
a + a^{-1} - \lambda & a  + b^{-1} - \mu \\
b + a^{-1} - \mu & b + b^{-1} - \lambda
\end{array}\right ).
\end{equation}
Consider the rational map $F$ given by:
\begin{equation}
F(\lambda, \mu )= \left ( \dfrac{\lambda^2 -\mu^2 -2}{\mu - \lambda} , \dfrac{2}{\lambda- \mu} \right ).
\end{equation}

The two Schur complements turn out to be the same, and are related to $F$ as follows:

\begin{prop} We have:
\begin{equation}
S_1(\varphi \circ M(\lambda,\mu)) =  S_2(\varphi \circ M(\lambda, \mu))   = M \circ F(\lambda, \mu).
\end{equation}
\end{prop}

\begin{proof}
The first Schur complement yields:
\begin{equation*}
S_1 (\varphi\circ M) (\lambda, \mu) =  a + a^{-1} - \lambda \id - \dfrac{1}{2(\mu - \lambda)}(b^{-1} - a^{-1} + \lambda- \mu) (b - a + \lambda- \mu).
\end{equation*}
Using the fact that $b^{-1}a = \sigma$, we obtain that $S_1(\varphi \circ M(\lambda,\mu)) = M(F(\lambda, \mu))$.
Similarly, the second Schur complements gives:
\begin{equation*}
S_2 (\varphi \circ  M)(\lambda,\mu) = b+ b^{-1} - \lambda \id - \dfrac{1}{2(\mu-\lambda)}(a^{-1} - b^{-1} + \lambda - \mu) (a-b +\lambda-\mu).
\end{equation*}
We then conclude that $S_2 \varphi\circ M(\lambda,\mu) = S_1 \varphi\circ M(\lambda,\mu)$, as required. 
\end{proof}


\bigskip

\subsubsection{Schur transformations  for the Hanoi group}
\label{section_schur_hanoi}

Consider the Hanoi group $\mathcal{H}$ and we consider the pencil of operator
\begin{equation*}
M(\lambda, \mu) := a + b + c - \la + (\mu-1)A \in \C[\mathcal{H}],
\end{equation*}
where $A$ is the operator given by the matrix:
\begin{equation*}
A := \left ( \begin{array}{lll}
 0 & 1 & 1 \\
 1& 0 & 1 \\
 1& 1 &0
\end{array} \right ).
\end{equation*}

The recursion matrix associated to $M$ on two levels takes the form:
\begin{equation*}
\varphi(M(\lambda,\mu)):= \left  ( \begin{array}{lll|lll|lll}
c - \la & 0 & 0 & \mu& 0 & 0 & \mu& 0 & 0 \\
0& -\la & 1 & 0 & \mu & 0 &  0 & \mu & 0 \\
0& 1 & -\la & 0 & 0 & \mu &  0 & 0 & \mu \\
\hline
\mu &0& 0 & -\la& 0& 1  & \mu& 0 & 0\\
 0& \mu& 0 & 0 &b-\la& 0& 0 & \mu & 0\\
  0 &0 &\mu & 1& 0& -\la & 0 & 0 & \mu\\
  \hline
\mu &0& 0  &  \mu& 0 & 0 & -\la& 1& 0 \\
 0& \mu& 0 &  0 & \mu & 0 & 1 &-\la& 0\\
  0 &0 &\mu &  0 & 0 & \mu & 0 &0 &a-\la\\
\end{array} \right )
\end{equation*}

The computation of the Schur complement with respect to an appropriate
corner was carried out by Grigorchuk and Suni\'c.
\begin{prop} \label{prop_renorm_hanoi} (\cite[Proposition 3.1]{grigorchuk_sunic_hanoi}) One has the following recursive formula:
\begin{equation*}
\det M_n (\lambda,\mu) = (\la^2 - (1 + \mu)^2 )^{3^{n-2}} (\la^2 - 1 + \mu - \mu^2 )^{2 \cdot 3^{n-2}} \det M_{n-1}(F(\lambda,\mu)),
\end{equation*}
where $F $ is the rational transformation
\begin{equation*}
F : (\lambda,\mu) \mapsto  \left(  \la + \dfrac{2\mu^2 (-\la^2 + \la +
    \mu^2)}{ (\la -1 -\mu )(\la^2 -1 + \mu -\mu^2 )} , \  
 \dfrac{\mu^2 (\la-1+\mu)}{(\la -1 -\mu )(\la^2 -1 + \mu -\mu^2 )} \right )
\end{equation*}
\end{prop}

%
\comm{***********

\msk
\subsubsection{Schur transformations for the Basilica group}
\label{section_schur_basilica}

For the Basilica group $\mathcal{B}$, one considers the pencil of operator:
\begin{equation*}
M(\lambda, \mu) := a + a^{-1} + \mu (b+ b^{-1}) - \lambda \in \C[\mathcal{B}].
\end{equation*}
The recursion matrix is of the form:
\begin{equation*}
\varphi \circ M(\lambda,\mu) = \left  ( \begin{array}{ll}
2 - \lambda & \mu(1 + a^{-1}) \\
\mu (1+ a) &b + b^{-1} -\lambda 
\end{array} \right ).
\end{equation*}

We obtain the following inductive formula on the determinants:
\begin{prop} (\cite[Lemma 4.2]{grigorchuk_zuk_spectral}) One has:
\begin{equation*}
\det M_{n+1} (\lambda, \mu) = \mu^{2^{n+1}} \det M_n (F(\lambda,\mu)),
\end{equation*}
where $F$ is the rational map given by:
\begin{equation*}
F : (\lambda,\mu) \mapsto \left ( -2 - \dfrac{\la (2-\la)}{\mu^2} , \dfrac{\la -2 }{\mu^2}  \right ).
\end{equation*}
\end{prop}
***************************}
\section{Background in Holomorphic Dynamics}

\subsection{Equidistribution of preimages  in dimension one}

$ $
\medskip

\subsubsection{General result}

\medskip

Let $f$ be a polynomial, it extends to a holomorphic map on the Riemann sphere $\hat\C$. The filled Julia set $\mathcal{K}(f)$ of $f$  is the set of non-escaping points in $\C$,  the Julia set $\mathcal{J}(f)$ is the boundary of that domain and the Fatou set $F(f)$ is the normality locus of $f$. 

\bigskip

\begin{thm} [\cite{brolin,L:note,
lyubich_entropy_properties_riemann_sphere,FLM}] \label{equidistr D1} 
  Let $f\colon \bC \ra \bC$ be a rational function of degree $d\geq 2$.
Then for all $z\in \bC$ except at most two points, we have:
$$
  \frac 1{d^n} \sum_{\zeta\in f^{-n} z } \de_\zeta \to \om,
$$
where $\om $ is the measure of maximal entropy for $f$. 
In the polynomial case, $\om $ coinsides with the harmonic measure on the
Julia set $\Jul (f)$.
\end{thm}



\comm{*****
The idea of the proof, in the polynomial case \cite{brolin},
is to consider the logarithmic potentials, $G_n$ and $G$, of the measures in question.
The former are subharmonic functions with logarithmic singularities
at the zeros of $f^n z$ (and at $\infty$), 
so it can be identified with  
$$
     G_n= \frac 1 {d^n} \log | f^n z|.
$$  
The latter is the Green function for the basin at $\infty$ with the
pole at $\infty$.
%
%
%
Then one can check:

\begin{lem}
     For a   
polynomial $f$, we have:  
$ G_n \to G$ in $L^1_\loc (U)$, 
where $U$ is any domain in $C$ that does not contain the exceptional
point in $\C$ (if any).%
\footnote{Such a point exists only if 
 $f$ is affinely equivalent to $z\mapsto z^d$.
For the latter, it is $0$.}
Moreover, the convergence is locally uniform on the basin of
infinity. 
\end{lem}

Taking the distributional Laplacian, we obtain the desired equidistribution. 
**********************} 

\msk

\subsubsection{Squaring map}

The doubling or squaring map is the map $f_0\colon z \mapsto z^2$. 
It has two superattracting fixed points on $\mathbb{P}^1$ corresponding to the origin and the point at infinity and its Julia set is the unit circle $\T$ in $\C$. 
The measure of maximal entropy the Lebesgue measure on the circle.

\msk
\subsubsection{Chebyshev map}
\label{Cheb sec} 

The {\em Chebyshev}  (or {\em Ulam-Neumann}) quadratic map $\cheb$  appears in several
normalizations: 
\be\label{cheb}
 \cheb\colon  z\mapsto 2z^2 -1, \quad {\mathrm{or}}\quad z\mapsto z^2-2,  \quad {\mathrm{or}}\quad
  z\mapsto 4z(1-z),
\ee
all of which are conjugate by appropriate affine changes of variable. 
Its special place in dynamics  becomes clear from the first expression,
as it satisfies the functional equation 
$$
     \cos 2z = \cheb (\cos z)  ,
$$
telling us that  $\cos$ semi-conjugates the doubling map $z\mapsto 2z$ to $\cheb$. 
In the coordinate $\zeta= e(z)\equiv e^{2\pi i z}$, it can be written as
$$
   \Zhuk (\zeta^2) = \cheb (\Zhuk (\zeta )), \quad {\mathrm{where}}\ 
    \Zhuk (z) = \frac 12 \left( z+\frac 1z \right)
$$
if the {\em Zhukovsky function}. Thus, $\Zhuk$ {\em semi-conjugate the quadratic map 
$f_0\colon z\mapsto z^2$ to $\cheb$}.  The Julia set $\Jul (f_0)$ is the unit circle $\T$,
while the Julia set $\Jul(\cheb)$ is the interval $\I= [-1, 1]$. Naturally, they are related by
the Zhukovsky function: $\Zhuk (\T) = \I$. 

Let  
$$dm  =  \frac 1{2\pi}   d\theta   $$
 be the normalized Lebesgue measure on $\T$. 
It is the measure of maximal entropy for $f_0$,
which gives the asymptotic distribution for the iterated preimages of
all points $z\in \C^*\equiv \C \sm \{0\}$. (All these are well-known
elementary statements.) Let us push this measure forward to $\I$:
\be\label{Cheb meas} 
  d\om  :=  \Zhuk_* ( dm ) =  \frac 1\pi \frac {dx} {\sqrt{1-x^2}} .  
\ee
We see that $\om $ is the measure of maximal entropy for $\cheb$,
which gives the asymptotic distribution for the iterated preimages of
all points $\zeta \in \C $.

\msk
\subsubsection{Cantor case} \label{section_cantor}

Consider the polynomial map $p \colon z \mapsto z^2 - z -3$. 
This map is called hyperbolic (see \cite[Section 14]{milnor}) since it is conjugate to $u\mapsto u^2 - 15/4$ where $u = z-1/2$ and the critical point escapes to the attracting fixed point at infinity. 
The Julia set of this map in the $u$ coordinates is a Cantor set contained in the union of intervals $[-5/2, -\sqrt{5}/2] \cup [\sqrt{5}/2 , 5/2]$. Translating back to the $z$ coordinates,  the Julia set of $p$ is a Cantor set contained in the union $[-2, (-\sqrt{5}+1)/2] \cup [(\sqrt{5}+1)/2,2]$. 
The measure of maximal entropy is the uniform self-similar measure on this Cantor set.

\msk
So, Theorem \ref{equidistr D1} is straightforward
in the three particular cases singled out above.
Incidentally, these are the only cases relevant for  this paper.


\subsection{Algebraic, topological and dynamical degrees, and algebraic stability}

Take a rational map $F \colon \P^2 \dashrightarrow \P^2$  (see Appendix \ref{section_appendix_rational_maps} for the definition of a rational map on any surface). 
The map $F$ is determined by three homogeneous polynomials $P_0, P_1, P_2 \in \C[x,y,z]$ with no common factors and with the same degree $d$, which we denote by $\deg(F)$. 
The integer  $d =\deg(F)$ is called the (algebraic) \textbf{degree of the rational map} $F$.
whereas its \textbf{topological degree} is the number of preimages counted with multiplicity of a generic point. 

A rational map $F \colon \bP^2(\C)\ra \bP^2(\C)$ is called \textbf{ dominant}
if its image is not contained in any algebraic curve. 

More generally, consider a surface $X$  obtained from  $\mathbb{P}^2$ by finitely many  blow-ups, it is called a \textbf{rational variety}.  A given rational map $F \colon \mathbb{P}^2  \dashrightarrow \mathbb{P}^2 $ lifts to a rational map $F_X$ on $X$ (see Appendix \ref{section_appendix_rational_maps}) and $F_X$ is said to be \textbf{algebraically stable} on $X$ if  there is no algebraic curve $C$ mapped under some iterate
$F_X^n$ to a point of indeterminacy.  
When $F$ is algebraically stable on $\mathbb{P}^2$, the sequence of degrees is multiplicative (see \cite{fornaess_sibony_II}): 
$$ \deg F^n = (\deg F )^n\quad  {\mathrm{for\  all}} \quad  n=1,2,\dots,$$
and if $F_X$ is algebraically stable, then  its induced  action on the Dolbeaut cohomology $H^{1,1}(X)$ (see Appendix \ref{appendix_push_full}) satisfies the relation 
$(F_X^n)^* = (F_X^*)^n$ for all integer $n$.
Note that the sequence $\deg(F^n)$ is submultiplicative:
\begin{equation*}
\deg(F^{n+m}) \leqslant \deg(F^n) \deg(F^m).
\end{equation*}
By Fekete's lemma \cite{fekete}, the \textbf{first dynamical degree} of $F$, denoted $\lambda_1(F)$ and defined by the formula:
\begin{equation*}
\lambda_1(F) = \lim_{n\rightarrow +\infty} \deg(F^n)^{1/n},
\end{equation*}
is a well defined real number satisfying $\lambda_1(F) \leqslant d$. 
When the map $F_X$ becomes algebraically stable on a surface $X$, 
one can compute the dynamical degree using the following statement.

\begin{prop} \label{prop_algebraic_stability_dynamical_deg}
Let $F\colon \mathbb{P}^2 \dashrightarrow \mathbb{P}^2$ and suppose that there exists a rational surface $X$  on which the lift $F_X$ of $F$ is algebraically stable.
Then one has: 
\begin{equation*}
\lambda_1(F) = \rho(F_X^*),
\end{equation*}
where $\rho(F_X^*)$ denotes the spectral radius of the pullback action $F_X^*$ on $H^{1,1}(X)$.
\end{prop}

The rational surface $X$ satisfying the conditions of the above Proposition is called an \textbf{algebraically stable model} for the map $F$. 
For arbitrary maps, the dynamical degree can be difficult to compute, however there are methods to determine this degree in more rigid situations. 
To do so, we state the general properties satisfied by these numbers.

\begin{thm} \label{thm_dynamical_degree} (\cite{dinh_sibony_une_borne_sup},\cite[Theorem 1]{dang_degrees}, \cite[Theorem 1.1]{dinh_nguyen_truong}) The following properties are satisfied:
\begin{enumerate}
\item[(i)] The dynamical degree is a birational invariant, i.e for any birational map $\varphi \colon \P^2 \dashrightarrow \P^2$, one has $\lambda_1(\varphi^{-1} \circ F \circ \varphi) = \lambda_1(F)$. 
\item[(ii)] If $F$ is a skew-product $F = (x,y) \mapsto (P(x), Q_x(y))$ where $P$ is a rational map of degree $p$ on $\P^1$ and $Q_x$ is a rational family of rational maps\footnote{formally $Q \in \C(x)(y)$} of $\P^1$ of degree $q$, then the dynamical degree of $F$ is given by the formula:
\begin{equation*}
\lambda_1(F) = \max (p,q).
\end{equation*} 
Moreover, the topological degree of $F$ is equal to the product $p q$.
\end{enumerate} 
\end{thm}



$ $

\ssk
\subsection{Existence of the Green currents.}

%


Let now $F\colon  \bP^2(\C)\dashrightarrow \bP^2(\C)$ be a rational map of the
projective space (with points of indeterminacy allowed). 
Then instead of taking iterated preimages of points, one should
consider iterated pullback of holomorphic curves.  Let $[C]$
stand for the {\em current of integration} over a holomorphic curve
$C$. Then the desirable result would assert that for a typical
$C$,  the normalized
currents $[(F^n)^* (C) ] $ converge to some current  $\Omega$
called {\em Green}. 
 There is an extensive literature on this subject
\cite{bedford_smilie_polynomial_diffeo_currents,russakovskii_shiffman,
favre_jonsson_brolin,
dinh_sibony_equidistribution_green_current,
bleher_lyubich_roeder}.
Below we will quote   a few sample results of this kind. 


%
%
%


\begin{thm}  {\rm (\cite[Theorem 2.2]{guedj_decay})} \label{Green current} 
Let $F: \bP^2(\C) \dashrightarrow \bP^2(\C)$ be a dominant rational map and  
let $X$ be a rational surface satisfying the following properties:  
\begin{enumerate}
\item[(i)] The lift $F_X$ of $F$ to $X$ is 
algebraically stable.
\item[(ii)] One has $\lambda_1(F) > 1$. 
\item[(iii)] There exists a constant $C>0$ such that  $\deg(F^n) \leqslant C \lambda_1(F)^n$ for all $n$. 
\item[(iv)] There exits a $\lambda_1(F)$ invariant class $\alpha \in H^{1,1}(X) $ by $F_X^*$ which is represented by a closed smooth semi-positive form.  
\end{enumerate}

  Then there exists  a unique (up to scaling)
  closed positive  $(1,1)$-current $\Om$ on $X$ representing $\alpha$ such that
  $$F_X^*(\Om) = \lambda_1(F) \Om .$$
\end{thm} 

The current is called the {\em Green} current of the rational map $F$.

 This theorem  was proved by Fornaess-Sibony
 \cite{fornaess_sibony_II,sibony_1999} in the particular case where $X=\bP^k(\C)$
 (in which case conditions (iii) and (iv) are satisfied automatically).
However,  our maps $F$ (albeit, elementary) do not fit into this
framework as they are not algebraicly stable on $\bP^k(\C)$. 
%


However, below we will show that each of them admits an
algebraicly stable model (condition (i)), and two of them 
(Grigorchuk and Hanoi)  satisfy condition (ii).
For these two maps we  will provide an explicit geometric description
of the Green current (without appealing to Theorem \ref{Green current}).

In fact, these two maps do fit into the framework of
Guedj's  Theorem. 
Indeed,  conditions (iii) and (iv)  of the theorem follow easily from the
integrability of $F$. 
%
For instance, if a map $F$ is semi-conjugate to a degree $\lambda_1$
one-dimensional map via a projection $\varphi : X \to C$ to
a smooth projective curve,
the invariant cohomology class for $F$ can be represented by the
$\varphi$-pullback of a K\"ahler form on $C$ (providing us with (iv)).

Let us note that  though the Lamplighter map does not fit into
the above frame  (as $\lambda_1=1$), it still admits an analogue of the Green current
that will be explicitly described. 

 In conclusion, let us  summarize properties of our three maps: 
 
\begin{table}[!h]
\begin{tabular}{|l|l|l|l|}
\hline
Group                           & Grigorchuk group & Lamplighter group  & Hanoi group      \\
\hline
Algebraic degree                & 3                &                    2 &           4       \\
\hline
Dynamical degree & $\lambda_1(R_{\mathcal{G}})=2$ & $\lambda_1(R_{\mathcal{L}}) = 1$ & $\lambda_1(R_{\mathcal{H}})=2$ \\
\hline
Topological degree              & $d_t(R_{\mathcal{G}})=2$          & $d_t(R_{\mathcal{L}})=1$            & $d_t(R_{\mathcal{H}}) =2$      \\
\hline
 Algebraic stability on $\bP^2(\C)$ & No & No & No\\
 \hline
  Algebraically stable model & Yes (\S~\ref{subsection_integrability_grigorchuk}) & Yes (\S~\ref{section_integrability_lamplighter}) & Yes (\S~\ref{subsection_integrability_hanoi})\\
  \hline
  Integrability  & Yes & Yes & Yes \\
  \hline
  
\end{tabular}
\end{table}

\bigskip

\subsection{Fatou, Julia sets of rational maps in higher dimension}
\label{section_fatou}

\msk
Given a rational map $F : \bP^2(\C) \dashrightarrow \bP^2(\C)$, 
the \textbf{Fatou set}   $\Fatou(F) $ is defined as  in the one-dimensional
situation:  $z\in \Fatou (F)$ if there is a neighborhood $U\ni z$ such
that the iterates $(F^n)_{n=0}^\infty$  are well defined (i.e., they
do not hit the indeterminacy points) 
and form an  equicontinuous family on $U$  (so, the  orbits near $z$ are
Lyapunov stable).%
\footnote{
Locally equicontinuous  families of maps are also called {\em normal} in
Complex Analysis.}  
There are two version of the big Julia set: 
\begin{enumerate}
\item[A] As the support of the Green current $\mathcal{J}(F)  = \supp \Om$. 
\item[B] As the complementary of the Fatou set, $\tilde{\mathcal{J}}(F) = \bP^2 \setminus \mathcal{F}(F)$.
\end{enumerate} 

For holomorphic map of $\bP^2$, $\tilde{\mathcal{J}}(F)=\mathcal{J}(F)$ but for rational maps, there could be a difference between these two sets
(see \cite[Corollary 1.6.7]{sibony_1999}). 

\bigskip

When $F $ is the renormalization  map associated to the Grigorchuk group (see \S~\ref{section_schur_grigorchuk}), $\tilde{\mathcal{J}}(F) \setminus \mathcal{J}(F)$ consists of the closure of a countable set of points. From the explicit expression of $\Omega$ in \eqref{eq_current_grigor}, the set $\mathcal{J}(F)$ is the union:
\begin{equation*}
\mathcal{J}(F) =\bigcup_{-1 \leq \theta \leq 1} \{ [\lambda, \mu , w] \ | \  4 w^2 - \mu^2 + \lambda^2 - 4 \theta \lambda w =0  \} \subset \bP^2(\C)
\end{equation*}
whereas
 the precise description of $\tilde{\mathcal{J}}(F)$ was obtained by Goldberg-Yang \cite{goldberg_yang}:
 \begin{equation*}
 \tilde{\mathcal{J}}(F) = \overline{\bigcup_{n \geqslant 0 } I(F^n)} \cup \mathcal{J}(F), 
\end{equation*}
where $I(F^n)$ are the indetermany points of $F^n$ on $\bP^2(\C)$.  
%
  
\begin{rem}
In many cases, one can also define a  ``small''  Julia set inside the big one
as the support of the measure of maximal entropy (see e.g \cite{bedford_smilie_polynomial_diffeo_currents,
bedford_lyubich_smillie_4,sibony_1999}).
However, it is not canonically defined in the cases of interest for us.   
\end{rem}

\msk
\subsection{General equidistribution results}

As we have mentioned above,
we are interested in  a result of the following type:

\newtheorem*{Desired Equidistribution Statement}
 {Desired Equidistribution Statement}

\begin{Desired Equidistribution Statement}
\textit{  Let ${\mathfrak{R}} $ be a certain class of dominant maps
of degree $d \geq 2$. 
Then for any $F \in {\mathfrak{R}} $ and a typical algebraic curve 
$C\subset \bP^2(\C)$, we have:}

\ssk\nin {\rm (i)} 
$$
     \frac {  [ (F^n)^* C ]   }     { d^n  \cdot  \deg C } \to \Om.
$$

\ssk\nin {\rm (ii)} 
\textit{For any holomorphic curve $S$, the restriction  $\Om \cond S\equiv \om_S$ is a
well defined measure $\om_S$.}
 
\ssk\nin {\rm (iii)} \textit{Letting $\nu_n$ be the probability  measure uniformly distributed
over $ (F^n)^* C \cap S $, we have $\nu_n \to  \om_S$. }
\end{Desired Equidistribution Statement}


Assertion (i) was obtained in the following situations:

\ssk\nin $\bullet$ 
  $\Space$  is the space of non-elementary  polynomial automorphisms
  of $\C^2$,  $C$ is an arbitrary affine algebraic curve 
  (Bedford and Smillie \cite{bedford_smilie_polynomial_diffeo_currents};
 
\ssk\nin $\bullet$ 
  $\Space$ is the space of proper polynomial maps%
\footnote{We assume without saying that $\deg F \geq 2$} 
 of $\C^2$,
  $C$ is a typical (in a capacity sense)  affine algebraic curve
  (Russakovskii and Shiffman \cite{russakovskii_shiffman}); 

\ssk\nin $\bullet$ 
  $\Space$ is the space of holomorphic endomorphisms of $\bP^2(\C)$,
  $C$ is an algebraic curve which is not contained in  the 
  ``exceptional subvariety'' (Favre and Jonsson \cite{favre_jonsson_brolin});

\ssk\nin $\bullet$ 
    $\Space$ is a space of dominant  rational endomorphisms of $\bP^2(\C)$
(subject of certain technical assumptions);
 $C$ is an algebraic curve which  does not pass through
 ``maximally degenerate'' periodic points 
(Bleher-Lyubich-Roeder \cite{bleher_lyubich_roeder}). 

\medskip
The validity of assertion (ii) is a consequence of  Bedford-Taylor's  intersection theory of $(1,1)$ currents (see Appendix \ref{appendix_current}). 

\medskip
Assertion  (iii) does not follow immediately from (i). 
The reason is that the intersection of currents is not continuous with respect to the weak topology.
However, this fact is known in particular situations listed below.

\ssk\nin $\bullet$ In the study of bifurcation of a holomorphic family of rational maps, Dujardin (\cite[Theorem 2.11.]{dujardin_bifurcation}) obtained these assertions when $C$ is a horizontal curve and when $S$ is the graph of a marked family of critical points.

\ssk\nin $\bullet$ When $F = (\lambda , z) \mapsto (\lambda, f_\lambda(z))$ where $f_\lambda$ is a family of rational maps of the same degree whose coefficients depend algebraically on $\lambda$ (\cite[Theorem C]{chio_roeder}).

\ssk\nin $\bullet$ In \cite{berteloot_dinh}, Berteloot and Dinh showed that the so-called bifurcation measure associated to the quadratic family $z^2 + c$ can be realized as the slice of the Julia set of a particular tangent map.



\comm{*****************

\msk
In the polynomial case,  one can  define the Green function
$$
    G(z) = \lim \frac 1 {d^n} \log \| F^n z \| , 
$$
where the limit is taken in $L^1_\loc (\C^2)$.
Then one needs to check that for a typical  polynomial $p(z)$,
$$
     \log |p(F^n z )| \to G (z)\quad  {\mathrm{in}}\ L^1_\loc (\C^2).
$$
Taking the pluripotential Laplacian $ \frac i2  \dibar\di $, 
one obtains convergence of the corresponding  currents. 

Restricting the above construction to a holomorphic curve $S$,
one obtains convergence of the corresponding  measures.  
**********************}

\bigskip 

\subsubsection{Transport of the equidistribution by conjugation}
Fix two simply connected domain $U,V$ of $\P^2$ and two dominant
rational maps $F, G$ on on $\P^2$
which preserve $U$ and $V$  respectively and take $\varphi \colon U \to V$ a biholomorphism such that $\varphi \circ F = G \circ \varphi$.

The following assertions show that
the equidistribution property is invariant under analytic
conjugacies.

\begin{lem} \label{lem_equidistr_conjugation} Take $C$ an irreducible algebraic curve in $V$.
Suppose that the following assertions holds:
\begin{enumerate}
\item[(i)] $G$ is algebraically stable on $\P^2$.
\item[(ii)] The sequence of currents $$ \dfrac{1}{ \lambda_1(G)^n} (G^n)^* [C]  $$  
converges to the Green current $\Omega_G$ of $G$.
\end{enumerate} 
Then the limit
$$ \dfrac{1}{\lambda_1(F)^n} (F^n)^*  \varphi^* [C\cap V] $$ also exists and is equal to a multiple  of the restriction of $\varphi^* \Omega_G  $ to  $U$.
\end{lem}

%


\begin{lem} \label{lem_slicing}
Fix $C_1, C_2$ two irreducible curves on $\P^2$.
 Suppose that the following properties hold.

\begin{enumerate}
\item[(i)] The map $G$ is algebraically stable on $\P^2$.
\item[(ii)] The curve $\varphi(C_1 \cap U)$ satisfies the condition of Lemma \ref{lem_equidistr_conjugation}.
\item[(iii)] The sequence of measures given by the intersection of currents
\begin{equation*}
\dfrac{1}{\lambda_1(G)^n} [\varphi(C_2 \cap U)] \wedge (G^n)^* [\varphi(C_1 \cap U)]
\end{equation*}
converges to a multiple of the measure $[\varphi(C_2 \cap U)] \wedge \Omega_G$,
where $\Omega_G$ is the Green current of $G$.
\end{enumerate}
Then the sequence of measures:
\begin{equation*}
\dfrac{1}{\lambda_1(F)^n} [C_2 \cap U] \wedge (F^n)^* [C_1 \cap U]
\end{equation*}
converges to a multiple of the measure $[C_2 \cap U] \wedge \varphi^* \Omega_G$.
\end{lem}

\subsection{Three particular direct products}


Although the existence of the Green current associated to two our maps follows from general results, the equidistribution of the preimages of curves toward this current and the precise description of the Green current will hold because our maps have a very specific form.

\medskip

\subsubsection{Direct product  $\id \times f$ related to the Grigorchuk group}

Let us consider a map
$$ 
F\colon \C^2\ra \C^2, \quad F= ( \id\times f),\quad 
        (\eta, \theta)  \mapsto (\eta, f(\theta) )
$$
where $f$ is a polynomial in one variable of degree $d\geq 2$. 
It extends to  $\P^1 \times \P^1$ as a holomorphic map.
Its filled  Julia set   $\Kfilled (F)$ (i.e the set of non-escaping points) in $\C^2$ is equal to the product 
$\C\times \Kfilled (f)$. 

 The Green function $G_F $ for $F$  depends only on the second coordinate
 and is equal to 
the  one-dimensional Green function $G_f (\theta)$ for the polynomial
$f$.  Indeed, on the basin of infinity,  $\C^2\sm \Kfilled (F)$, we have
$$
   G_F (\eta, \theta)  = \lim \frac 1{d^n}   \log  \|  F^n (\eta,
   \theta) \| = \lim \frac 1{d^n}   \log  |  f^n ( \theta) | =
   G_f(\theta),
$$
while on $\Kfilled (F)$ both functions vanish. 
  
The Julia set $\Jul(F)= \C\times \Jul(f) $ is naturally laminated by
the horizontal  complex lines $L_\theta = \C\times \{  \theta\}$,
$\theta\in \Jul (f)$. The Green current 
$$
 \Om  =  \frac i2 \di\dibar G = \De G_f \, d\theta \wedge d \bar\theta = \om \, d\theta \wedge d \bar\theta 
$$
is a horizontal  laminar current whose transverse measure is equal to
the harmonic measure $\om $ for $f$. Thus, for a non-horizontal
holomorphic curve $S\subset \C^2$, the restriction $\Om \cond S$
is identified  with  the measure $\om_S \colon = (p_2\cond S) ^* (\om)$, where 
$ p_2\colon \C^2\ra \C$ is the projection to the $\theta$-axis.

Given two holomorphic curves,  $C$ and $S$,
which do not have common irreducible components,
we let $[C\cap S ]  $ be the counting measure on $C\cap S$, it is equal to the intersection of current $[C ] \wedge [S]$.

Recall that  the points $0, \infty$ are fixed points of the squaring
map $z\mapsto z^2$. We thus say that the lines $\mathbb{P}^1 \times \{ 0\}$,
$\mathbb{P}^1 \times \{ \infty\}$
are the exceptional lines for the map $\id \times f_0$ where $f_0$ is the squaring map.


\begin{lem}\label{equidistr for products} Suppose that $F = \id \times f_0$ where $f_0$ is the squaring map $z \mapsto z^2$. 
Let  $C$ and $S$  be  two  irreducible algebraic curves 
 such that $C$ is neither a  vertical line nor  a horizontal exceptional line 
   while $S$ is not horizontal and such that the points of $C\cap S$ are not on the exceptional lines.
Then 
\begin{equation}\label{equidistr of intersections}  
     \frac 1 {2^n} \, [ (F^n)^* C \cap S ] \to (\deg C) \cdot
     (\deg S)\cdot \omega_S.  
\end{equation}
\end{lem}

\begin{rem} Observe that the equidistribution of the preimages of $C$ by $F$ does not directly imply the convergence of their intersection with $S$ to the above measure. The main issue is that the product of currents is not continuous with respect to the weak topology on currents. However, here we exploit the basic dynamical properties of the squaring map.
\end{rem}

\begin{proof}

Denote by $p_1,p_2$ the projection of $\C \times \C \equiv \C_1\times \C_2$
onto the first and second factor,
$\C_1\equiv \C\times \{0\} $ and $\C_2\equiv \{0\}\times \C$,  respectively.
For $\eta\in \C_1$, we let $L_\eta:= \p_1^{-1}  (\eta)$ be the fiber line over $\eta$,
and let $\T_\eta\subset L_\eta$ be the unit circle inside. 

Let $B\subset \C_1$ be the set of projections of the branch points of $\pi_1: C\to \C_1$.
Let
$$
 C^* := C\sm ( \p_1^{-1} (B) \cup \p_2^{-1} (0)).
$$
Note that the points of intersection of the horizontal line $\C_1 \times \{0 \}$ with the curve $C$ are fixed points of $F$,
we choose a base point $\eta_\base \in \C_1 \sm (B\cup (C\cap \C_1) )$ 
so that the corresponding vertical line $L_{\eta_\base} $
avoids $C\cap \C_1$  and the branch points of $\p_1 |\, C $.
 
Let $C \cap L_{\eta_\base} = \{ \eta_\base \} \times Q_\base $; it consists of 
$\delta:= \deg C $ points of transverse intersections. 
Then let
$$ Q^n_\base  := F^{-n} (Q)  =   F^{-n} (C) \cap  L_{\eta_\base} = \{\eta_\base \} \times f_0^{-n} (Q_\base) ;$$
it consists of  $\delta d^n $  transverse intersection points.
The   uniform  
measures $\mu^n_\base$  on these sets converge to the
Lebesgue measure  $\om_\base$ on $\T_\base\equiv \T_{\eta_\base}$.   

Let $T:= S\cap \p_2^{-1} (\T) $,
and let $T^*$ be obtained from $T$ by puncturing out
 branch point of   $p_1|\, S$  and $p_2|\, S $,
and points of  $T\cap p_1^{-1} ( B) $. 
Take a point $s\in T^*$,
and  select a simply connected neighborhood 
$U\supset \{\eta_0, p_1(s) \} $ in the horizontal axis $\C_1$ whose closure does not contain points of  $B$ and $C$.
Then $C$ is decomposed over $U$ in $\delta$ univalent branches $C_i\subset \C^2$
(i.e., graphs of holomorphic functions $\psi_i: U\ra \C$).
Taking preimages of these branches by $F^n$, we obtain $\de\, 2^n$ univalent  branches $C_{ij}^n\subset \C^2$
over $U$ parametrized by holomorphic functions $\psi_{i,j}^n : U \to \C$ such that:

\smallskip\noindent (i)
They are pairwise disjoint, so they induce a holomorphic motion $h_\eta$
of the set $X_\base:=\bigcup Q^n_\base$ over $U$
(see e.g., \cite[\S 17]{L:book}).

\smallskip\noindent (ii)
Their slopes go to $0$ exponentially fast (since the
fibered map $F$ is horizontally expanding away from the exceptional lines).

\smallskip By the $\lambda$-Lemma (see e.g., \cite[\S 17.2]{L:book}),
$h_\eta$ extends to a holomorphic motion of the closure $\overline{ X_\base}= X_\base \cup \T_\base$
(for which we will keep the same notation).
By (ii), the limiting functions for $\psi^n_{ij}$ are constants,
so $h_\eta |\, \T_{\base}  = \id$.

Take now a relative neighborhood $W\subset S$ of $s$ that projects univalently
to $\C_2$ by $p_2$, and let $W_\base:=  (p_2|\, L_{\eta_\base})^{-1} (p_2(W))$.   
Then our holomorphic motion induces a homeomorphic holonomy map
$\gamma: W_0\to W$.

Let $\om := \gamma_* (\om_\base)$, $\mu^n:= \gamma_* \mu^n_\base$. 
Take a continuous test function $u$ on $S$ supported on $W$,
and let $u_\base$ be its pullback to $W_0$.
Since the measures $\mu^n_\base$ converge to $\om_\base$,
$$
   \int u_\base\, d\mu^n_\base \to \int u_\base \, d\om_\base. 
$$
Pushing this forward by  $\gamma$ to $W$, we obtain:
$$
     \int u\, d\mu^n \to \int u\, d\om.
$$
    It folllows that any limiting measure $\nu$ on $S$  for the sequence $(\mu^n)$,
being restricted to $T^*$, coincides with $\om$. In particular, $\nu|\, T^*$
is a  probabilty measure, implying that $\nu(T\setminus T^*) =0$.
Hence $\nu= \om$, and the conclusion follows.

\end{proof}

\msk

\subsubsection{Twist map on the elliptic cylinder related to the Lamplighter group} \label{section_twist}

Consider a product map $F \colon \C^2 \to \C^2$ given by:
\begin{equation*}
(\eta , z ) \mapsto (\eta , M_\eta (z)), 
\end{equation*}
where $M_\eta \in \GL_2(\C)$ defines a M\"obius transformation with polynomial coefficients in $\R[\eta]$ such that the trace $\tr (M_\eta)$ is a non-constant polynomial in $\eta$ of some degree $d$.
Denote by $E \subset \C$ the locus of parameter $\eta$ such that the transformation $M_\eta$ is elliptic.
Observe that $E = \{ \eta \ | \ \tr(M_\eta) \in [-2,2]\}$ is a finite union of at most $d$ intervals.  
For each $\eta \in E$, $M_\eta$ is conjugate to a rotation by $\rho(\eta)$ and the corresponding conjugation maps the real line to the unit circle. As a result,  the set of non-wandering points for $F$ is the product $E\times \T$. 
Consider the parabolic locus $\mathcal{P}$ for the family $(M_\eta)$. 
To describe the spectral current whose support is on this set, we need to consider the conjugation $\varphi\colon (\C\sm \mathcal{P}) \times \C \to (\C\sm \mathcal{P}) \times \C^*$ such that the restriction to the non-parabolic locus is of the form
\begin{equation*}
\varphi \circ F \circ \varphi^{-1} \colon(\eta, u) \mapsto (\eta, e^{i\rho(\eta)} u ).
\end{equation*}
Letting  $\tilde{F} = \varphi \circ F \circ \varphi^{-1}$ be the conjugate of $F$ by $\varphi$,
 we obtain the following.

\begin{prop} \label{prop_twist} Take $L$ and $C$ two real lines which are neither on a vertical  nor horizontal in $\R^2$, consider their complexifications $C_\C, L_\C$ in $\C^2$, let $\tilde{C}_{\C}$, $\tilde{L}_\mathbb{C}$ be their image by $\varphi$, and let $l \colon \C \to \C$ be a rational function whose graph in $\C^2$ is equal to   $\tilde{L}_\C$.
 Then the following properties hold.
\begin{enumerate}
\item[(i)] The sequence of currents $$ \dfrac{1}{n} (\tilde{F}^{n})^*[\tilde{C}_\C] $$
converges to a current supported on $E\times \C^*$, laminated by vertical punctured complex lines with transverse measure $\rho^* d\theta$ where $d\theta$ is the Lebesgue measure on the circle.
\item[(ii)] The sequence of counting measures  $$\dfrac{1}{n}\tilde{F}^{-n}(\tilde{C}_\C) \cap \tilde{L}_{\C} $$
converges to the measure  $l_*\rho^* d\theta$ on $\tilde{L}_\C$.
\end{enumerate}
 
\end{prop}

\begin{rem} In the case of the Lamplighter group, the associated map $F$ is:
\begin{equation*}
(\eta, z) \mapsto \left (\eta, \dfrac{\eta z -4}{z}\right ),
\end{equation*}
and the line $L_\C$ we consider is of equation $\eta = z$.
\end{rem}


\begin{proof} Let us prove assertion (i). Observe that on the loxodromic locus, the restriction of the above current converges exponentially fast to zero.                                                                                                                                                                                                                                                                                                                                                                                                                                                                                                                                                                                                                                                                                                                                                                                                                                                                                                                                                                                                                                                                                                                                                                                                                                                                                                                                                                                                                                                                                                                                                                                                                                                                                                                                                                                                                                                                                                                                                                                                                                                                                                                                                                                                                                                                                                                                                                                                                                                                                                                                                                                                                                                                                                                                                                                                                                                                                                                                                                                                                                                                                                                                                                                                                                                                                                                                                                                                                                                                                                                                          
Indeed, on the loxodromic locus, we can suppose that $\Im \rho(\eta) > 0 $ and the forms $(\tilde{F}^n)^* du, (\tilde{F}^n)^* d\bar u$ are given by:
\begin{equation*}
(\tilde{F}^n)^* du =e^{in \Re\rho(\eta)-  n \Im \rho(\eta)} \left ( du + i n u  d\rho \right ) ,
\end{equation*} 
\begin{equation*}
(\tilde{F}^n)^* d\bar u = e^{-in \Re\rho(\eta)  -  n \Im \rho(\eta)} \left ( d\bar u - i n d\bar \rho \right ).  
\end{equation*}
Since these forms converge exponentially fast to zero in the loxodromic locus and since $C$ is not a horizontal line, we obtain that the current $(\tilde{F}^n)^* [\tilde{C}_\C]$ converge exponentially fast to zero on that locus.

Let us now consider the current on the elliptic locus. When $\eta \in E$, $\rho(\eta)$ is real and the restriction of $\rho$ to $E$ is a real analytic function. Suppose that the  curve $\varphi(C)$ is parametrized by $\eta \mapsto g(\eta)$. 
By restricting $\rho$ to a smaller subset, we can suppose that $\rho$ is injective on $E$ and let us consider the map  $\rho^{-1}$. 
The pullback of the line $C$ is then parametrized by:
\begin{equation*}
\eta \mapsto g(\eta) - n \rho(\eta) \in \R/2\pi\Z. 
\end{equation*} 
Reparametrizing by $ \omega = \rho(\eta)$, we obtain:
\begin{equation*}
\omega \in \T \mapsto g(\rho^{-1}(\omega) ) - n \omega \in \T.
\end{equation*}
Geometrically, the above map is the graph of $\omega \mapsto - n\omega \in \T$ which is transported vertically by $g (\rho^{-1}) $. 
The graphs $\omega \mapsto - n \omega$ equidistribute towards the real laminar current 
$$ \int_\T [\{ \omega\} \times \T] \  d\theta(\omega), $$
so we deduce that the real currents $(1/n)[F^{-n}(\varphi(C))]$ converge to the current
\begin{equation*}
\int_E  [\{\eta \} \times \T] \   \rho^* d\theta 
\end{equation*}
In particular, the currents associated  $(1/n)(F^{n})^* [\tilde{C}_{\C}]$ converge to the laminar current
\begin{equation*}
\int_E  [\{\eta \} \times \C] \   \rho^* d\theta, 
\end{equation*}
as required.


\bigskip

Let us prove assertion (ii). Let us also observe that the map $\rho \colon E \to \T \simeq \R/2\pi\Z$ is surjective. By restricting to a smaller subset, we can  suppose furthermore that $\rho \colon E \to \T$ is bijective.
Let us show that $F^{-n}(C) \cap L$ contains $n$ points counted with multiplicity.
Observe that the conjugation $\varphi$ maps any subset of $E\times \R$ to $E\times \T$ where $\T$ is the unit circle in $\C$.
Let us consider the real curves $\tilde{L} =\varphi(L)$, $\tilde{C} = \varphi(C)$.
These two curves $\tilde{L}, \tilde{C}$ are the graphs in $E\times \T$ of two analytic functions $l, g \colon E\to \T$.
Using an appropriate parametrization, one can always suppose that $l$ is locally constant function equal to $0 \in \R/2\pi \Z$. 
Now the intersection $\tilde{L} \cap \tilde{F}^{-n}( \tilde{C})$ is locally given by:
\begin{equation*}
\tilde{L} \cap \tilde{F}^{-n}( \tilde{C}):= \{ (\eta,0) \in E \times \R/2\pi\Z \  | \ g(\eta) -n \rho(\eta) =0 \in \R/2\pi\Z \}.
\end{equation*}
Reparametrizing by $\omega = \rho (\eta)$, we consider the set
\begin{equation*}
\{ g(\rho^{-1}(\omega)) - n \omega =0 \}.
\end{equation*}

Let us chop the circle $\T$ into $n$ subintervals $[\omega_1 , \omega_2], \ldots , [\omega_n , \omega_{n+1}]$ so that $n\omega_i = 0 \in \R/2\pi\Z$ and such that the restriction of $\omega \mapsto n \omega$ on each of these subintervals is injective. 
Now the graph of $g \circ \rho^{-1}$ intersects the graph of $ \omega \mapsto n\omega$ exactly once in each of these subintervals. 
As a result, the intersection $\tilde{L}\cap \tilde{F}^{-n}(\tilde{C})$ contains $n$ points and we have
\begin{equation*}
\dfrac{1}{n}\tilde{L}\cap \tilde{F}^{-n}(\tilde{C}) = \dfrac{1}{n}\tilde{L}_\C \cap  \tilde{F}^{-n}(\tilde{C}_\C),
\end{equation*}
since the measures $\tilde{L}_\C \cap  \tilde{F}^{-n}(\tilde{C}_\C)$ have mass $n$.  
Moreover, going back to the $\eta$ coordinates, we obtain that the sequence of measures:
\begin{equation*}
\dfrac{1}{n}(\tilde{L}\cap \tilde{F}^{-n}(\tilde{C}))
\end{equation*}  
converges to the measure $l_* \rho^* d\theta$.
%
\end{proof}


\subsubsection{Skew product over the Cantor dynamics related to the Hanoi group}

Let us consider a map \\
$F \colon \C^2 \to \C^2$ of the form
\begin{equation*}
F = (\eta , \theta) \mapsto (p(\eta), \lambda(\eta) \theta),
\end{equation*}
where $p (\eta) = \eta^2 - \eta -3$ is a hyperbolic polynomial of degree $2$ and $$\lambda(\eta) = (\eta -1) (\eta + 2)/(\eta +3) ,$$ is a rational function on $\eta$.  

Recall from Section \ref{section_cantor} that $p$ is conjugate to the map $u \mapsto u^2 - 15/4$ with a Cantor Julia set lying on the real line. 
The Julia set of $F$ is laminated by a Cantor set of vertical complex lines $\{ \eta \} \times \C$ where $\eta \in \mathcal{J}(p)$. 

The Green current of $F$
$$ \Omega = \Delta G_p  d\eta \wedge d\bar\eta  =\omega d\eta \wedge d\bar\eta $$
 is a vertical laminar current whose transverse measure is equal to the  measure of maximal entropy  $\omega$ for the polynomial $p$.

\begin{prop} \label{prop_skew_cantor} Fix $\eta_0 \in \R$. Let  $L$ a real line which is neither vertical nor horizontal and let $L_\C$  be its complexification.
Then the following properties hold.
\begin{enumerate}
\item[(i)] The sequence of currents
\begin{equation*}
\dfrac{1}{2^n} F^{-n}(\{\eta_0 \}\times \C)
\end{equation*}
converges to the Green current of $F$.
\item[(ii)] The sequence of counting measures
\begin{equation*}
\dfrac{1}{2^n} F^{-n}(\{\eta_0 \}\times \C) \cap L_\C = \dfrac{1}{2^n} (F^n)^* [\{\eta_0 \}\times \C] \wedge L_\C
\end{equation*}
converges to the measure $\Omega \wedge L_\C$ which is the transport of the  measure of maximal entropy on  $\mathcal{J}(p)$ to the line $L$.
\end{enumerate} 
\end{prop}

\begin{proof}
Assertion (i) follows directly from the equidistribution of the preimages of $\eta_0$ towards the equilibrium measure $\omega$ on the Julia set of $p$. 
The second assertion then follows from the fact that $L_\C$ is transverse to all the fibers $\{\eta \} \times \C$. 
Indeed, let us denote by $\mu_n$ the counting measure 
\begin{equation}
\mu_n := \dfrac{1}{2^n} F^{-n} (\{\eta_0 \} \times \C) \cap L_\C.
\end{equation}
Observe that the restriction of $F$ on the horizontal axis is given by $(\eta ,0) \mapsto (p(\eta), 0)$.


The preimage $F^{-n} (\{ \eta_0\} \times \C)$ is a union of $2^n$ vertical fibers counted with multiplicity and each of the $2^n$ point in the intersection of $F^{-n} (\{\eta_0\} \times \C) $ with the horizontal axis can moved to a point on $F^{-n}(\{\eta_0 \} \times \C) \cap L_\mathbb{C}$ via the  holonomy along the vertical foliation.
Since the sequence of counting measures
\begin{equation*}
\dfrac{1}{2^n} F^{-n} ( \{\eta_0  \} \times \C) \cap (\C \times \{0 \})
\end{equation*}
converge to the measure of maximal entropy of $p$ on the horizontal axis, we deduce that $\mu_n$ converges to the transport of this measure to $L_\mathbb{C}$ along the vertical foliation.
\end{proof}

\msk

                                                                       \section{Atomic density of states}

In some cases, the density of states (defined in \S~\ref{section_spectra}) is atomic. We explain this phenomenon by a discrepancy between the dynamical degree of the renormalization map and the growth of the size of the graph that appear  in the renormalization. 
This phenomenon already appeared in the work of Sabot (see \cite[Theorem 4.14]{sabot_electrical}) who used it to study the spectrum of the Laplacian arising from fractal sets.
In our situation, the renormalization transformation is related to the spectrum in a slightly different way but the resulting statement is similar. We thus state our result.

\begin{thm} \label{thm_discrete} Consider a sequence of polynomial $P_n \in \C[x,y]$ of degree $d^n$ where $d>1$ is an integer and a rational map $F \colon \C^2 \to \C^2$ whose dynamical degree $\lambda_1(F)$ satisfies the condition $\lambda_1(F) < d$ and such that:
\begin{equation*}
P_{n}(x,y) =  Q^{ d^{n-p}} \cdot  P_{n-1} (F (x,y)),
\end{equation*} 
where $p=0,1,2$, $a_i \in \N$, $Q$ is a polynomial in $\C[x,y]$.
Then the sequence of currents:
\begin{equation*}
\dfrac{1}{d^n} [P_n=0]
\end{equation*}
converges to a limiting current supported on countably many curves and its intersection with a generic curve yields an atomic measure.
\end{thm}

\begin{proof}
Taking the logarithm in the formula defining $P_n$, we have:
\begin{equation*}
\dfrac{1}{d^n}\log |P_n| = \sum_{i=1}^k  d^{-p}\log |Q| + \dfrac{1}{d^n} \log |P_{n-1}(F(x,y))|.
\end{equation*}
Applying the above formula inductively, we obtain:
\begin{equation*}
\dfrac{1}{d^n}\log |P_n| = \sum_{j=0}^{n-1}   \dfrac{1}{d^{j+p}}\log |Q(F^j(x,y))|  + \dfrac{1}{d^n} \log |P_{0}(F^n(x,y))|.
\end{equation*} 
Consider the current:
\begin{equation*}
\sum_{j=0}^n  \dfrac{1}{d^{j+p}} [Q \circ F^j =0]   + \dfrac{1}{d^n}  [P_0 \circ F^n =0] . 
\end{equation*}
Since the dynamical degree of $F$ satisfies the condition
$\lambda_1(F)< d$, so the currents $[Q_i \circ F^j=0]/ d^j$ have mass
bounded by
$C_\eps (\lambda_1(F)+\epsilon)^j/d^j$ for any $\epsilon > 0$.  
This shows that the above currents  converge to a current supported on countably many curves, so the current $[P_n=0]/ d^n$ converges to a current satisfying the same properties.

\medskip

We now show that the intersection of the limit $ \lim_n 1/d^n[P_n=0]$ with another generic curve yields an atomic measure.
Observe that $\lim  1/d^n[P_n=0]$ is a current whose support is  contained in the support of the $(1,1)$ current:
\begin{equation*}
 \sum_{j=0}^{+\infty} \dfrac{1}{d^{j+p}} [Q\circ F^j=0]. 
\end{equation*}  
Let us denote by $\phi_n$ and $\phi$ the following plurisubharmonic functions:
\begin{equation*}
\phi_n =   \sum_{j=0}^{n} \dfrac{1}{d^{j+p}} \log|Q\circ F^j|, 
\end{equation*}
\begin{equation*}
\phi =  \sum_{j=0}^{+\infty} \dfrac{1}{d^{j+p}} \log|Q\circ F^j|.
\end{equation*}
We observe that the difference $\phi - \phi_n$ is a plurisubharmonic function and is given by the tail of the above series.
We choose $S$ a generic curve and denote by $[S]$ the current of integration on $S$. 
Since $S$ is generic, the intersection of $S $ with the support of the current induced by $\phi$ is countably many points.
By Theorem \ref{thm_bedford_taylor}, the currents $\phi [S]$, $\phi_n [S]$, $(\phi-\phi_n) [S]$, $i\partial \bar\partial \phi \wedge [S]$, $i\partial \bar\partial \phi_n \wedge [S]$, $i\partial \bar\partial (\phi-\phi_n) \wedge [S]$ are well-defined and are locally finite.
Let us choose any two compact subset $K,L$ in $X$ such that $L \subset K^o$. By Theorem \ref{thm_chern_levine}, there exists a neighborhood $V$ of $S \cap L(\phi - \phi_n)$ and a constant $C>0$ such that:
\begin{equation*}
|| (\phi - \phi_n) [S] ||_L \leqslant C || \phi -\phi_n ||_{L^{\infty}(K\setminus V)} ||[S]||_K,
\end{equation*} 
\begin{equation*}
|| i\partial \bar\partial ( (\phi - \phi_n) [S]) ||_L \leqslant C || \phi -\phi_n ||_{L^{\infty}(K\setminus V)} ||[S]||_K.
\end{equation*} 
Since $\phi - \phi_n$ converges uniformly to $0$ on $K\setminus V$, we deduce from the previous estimate that:
\begin{equation*}
\lim_{n\rightarrow +\infty} i \partial \bar\partial \phi_n \wedge [S] = i\partial \bar \partial \phi \wedge [S].
\end{equation*}
Since the right hand side of the previous equality is an atomic measure, we deduce that the intersection of $\lim_n (1/d^n [P_n=0] \wedge [S])$ is also atomic.
\end{proof}
\section{Two rational maps associated with the Grigorchuk group} 

The two maps are:
$$
     F(\la, \mu) = \left( \frac {2\la^2} {4-\mu^2}, \
    \mu+ \frac  {\mu \la^2} {4-\mu^2} \right).
$$

\begin{equation}
G \colon (\lambda, \mu) \mapsto \left ( 2\dfrac{ 4 - \mu^2}{\lambda^2} , - \mu \left (1 + \dfrac{4-\mu^2}{\lambda^2}\right )\right ).
\end{equation}

In homogeneous coordinates, these maps have the form:
\begin{equation} \label{formula_grigor_F}
F = [\lambda: \mu : w] \mapsto [2 \lambda^2 w : \mu (4 w^2 - \mu^2) + \mu \lambda^2: w(4w^2 - \mu^2)],
\end{equation}
We shall set in this section $P_0, P_1, P_2$ the three polynomials defining $F$:
\begin{equation*}
P_0 = 2 \lambda^2 w , 
\end{equation*}
\begin{equation*}
P_1 = \mu (4 w^2 - \mu^2) + \mu \lambda^2,
\end{equation*}
\begin{equation*}
P_2 = w(4w^2 - \mu^2).
\end{equation*}
\begin{equation} \label{formula_grigor_G}
G \colon [\lambda : \mu : w] \mapsto [2 (4 w^2 - \mu^2)w  :  - \mu (\lambda^2 + 4w^2 - \mu^2): \lambda^2 w ].
\end{equation}

In fact $G = H \circ F$ where $H$ is the particular involution:
\begin{equation}
 [\lambda: \mu: w ] \mapsto [4w : -2 \mu : \lambda]. 
 \end{equation} 
We list the elementary properties satisfied by $F,G$. 
\begin{enumerate}
\item[$\bullet$] $F$ and $G$ have topological degree $2$.
\item[$\bullet$] $F$ and $G$ have algebraic degree $3$.
\item[$\bullet$] Both $F$ and $G$ have five indeterminacy points in total,  the points $[0: \pm 2 :1]$ in $\mathbb{C}^2$, 
and three more at infinity, the horizontal pole $[1:0:0]$ and two diagonal
points $[\pm 1: 1 :0]$.  
\end{enumerate}
\bigskip 

\subsection{Integrability of  the two renormalization maps}
\label{subsection_integrability_grigorchuk}

We first investigate the properties of the map on $\P^2$ and describe our method to recover two invariant fibrations for $F$ through the analysis of the dynamics of its indeterminacy points and curves.

Consider $\pi \colon \C^2 \dashrightarrow \C^2$ the rational map given by:
\begin{equation*}
\pi \colon (\lambda, \mu) \mapsto \left ( \eta:= \phi(\la,\mu) ,  
    \theta :=  \psi(\lambda,\mu) \right ),
\end{equation*}
where 
\begin{equation} \label{def_phi}
\psi(\lambda,\mu)= \frac {4-\mu^2+\la^2} {4\la}  ,
\end{equation}
and
\begin{equation} \label{def_psi}
\phi(\la,\mu)=  \frac{4-\la^2+\mu^2} {4\mu}.
\end{equation}
%

\begin{thm} \label{thm_conjugate_grigorchuk} 
The following properties are satisfied.

\begin{enumerate}
\item[(i)]The rational map $F$ is semi-conjugate via $\pi$ to $\id \times t$ where $t$ is the Chebyshev map (i.e  $(\id \times t) \circ \pi = \pi \circ F$).  
\item[(ii)]There exists two $F$-invariant simply connected domain $U_1, U_2 \subset \C^2$ such that $U_1 \cup U_2 = \C^2$ and  the restriction of $F$ on each of these domain is analytically conjugate to the map $(\eta , z) \mapsto (\eta , z^2)$. 
Moreover, we can choose the analytic conjugation $\varphi$  on $U_i$ so that:
\begin{equation*}
\psi \circ \varphi^{-1} (\eta, z) = \dfrac{1}{2}\left ( z + \dfrac{1}{z}\right ),
\end{equation*}
where $\psi : \C^2 \to \C$ is the function defined above.
\end{enumerate}

\end{thm}

Assertion (ii) in the above statement can be summarized in the  diagram below.

\begin{equation*}
\xymatrix{(\eta, z) \in \varphi(U_i) \ar[r] \ar[d]^{\varphi^{-1}}  \ar@/_2pc/[dd]& (\eta, z^2) \in \varphi(U_i)  \ar@/^2pc/[dd]   \\
  U_i  \ar[r]^{F} \ar[d]^{\psi} & \ar[u]^{\varphi}  \ar[d]^{\psi} U_i  \\
\dfrac{1}{2} \left ( z + \dfrac{1}{z} \right ) \in  \C \ar[r] &   \dfrac{1}{2} \left (z^2 + \dfrac{1}{z^2}\right ) \in \C. }
\end{equation*}

\begin{rem} The second map $G$ was studied in detail by M. and Y. Vorobets and a more complicated conjugation has been determined \cite{vorobets_notes}. 
\end{rem}

The proof of the above theorem  is based on several results in this section. We first study the dynamical properties of the map $F$ in Lemma \ref{lem_curve_1}, Proposition \ref{prop_contracted_grigorshuk}, Proposition \ref{prop_grigor_lift_dynamics}. Using these, we then obtain in Proposition \ref{prop_invariant_pencil_grig} the existence of two invariant pencils for $F$. Finally we study in more detail these two  pencils by rational curves to determine in 
  Lemma \ref{lem_precise_conjugation_grig} an explicit conjugation for the map $F$.

\medskip

Let us study the orbit of contracted curves for both maps (i.e curves whose image by $F$ and $G$ is collapsed to a point). 
Observe that since $F = H \circ G$, the contracted curves for $F$ and $G$ are the same. 
As a result, one finds that the jacobian of $F$ is of the form:
\begin{equation*}
J(F) = -12 \lambda (\mu - 2 w) w (\mu + 2 w) (\lambda^2 - \mu^2 + 4 w^2).
\end{equation*}
Observe that the vertical line $\{\lambda=0\}$ is  a curve of fixed point for $F$ and is mapped by $G$ to the line at infinity. 

Denote by $C_1$ the curve $\{ \lambda^2 - \mu^2 + 4 w^2 =0\} $. We obtain:

\begin{lem} \label{lem_curve_1}The curve $C_1$ is collapsed by $F$ to $[-2:0:1]$ which is then mapped by $F$ to the fixed point $[2:0:1]$ for $F$.
\end{lem}

\begin{proof}
Recall that we have denoted by $P_0 = 2 \lambda^2 w $, $P_1 = \mu (4 w^2 - \mu^2 + \lambda^2) $ and $P_2= w(4 w^2 - \mu^2)$ the homogeneous polynomials defining $F$.
Observe that $\lambda^2 - \mu^2 + 4 w^2$ divides the polynomial $P_1$, we have also
\begin{equation*}
P_0(\lambda, \pm \sqrt{4w^2 + \lambda^2}, w) = 2 \lambda^2 w,
\end{equation*}
\begin{equation*}
P_2(\lambda, \pm \sqrt{4w^2 + \lambda^2}, w) = - w \lambda^2.
\end{equation*}
In particular, this proves that the curve $C_1$ is contracted to the point $[-2:0:1]$. 
Now $F$ maps $[-2:0:1]) $ to the point $[2:0:1]$, which is then fixed by $F$, as required.
\end{proof}


We summarize the dynamics of all the contracted curves.
\bigskip

\begin{prop} \label{prop_contracted_grigorshuk} The following properties hold.
\begin{enumerate}
\item[(i)] The map $F$ collapses the curves $\{\mu = \pm 2 w\}$ to the indeterminacy points $[\pm 1: 1 :0]$ respectively. 
\item[(ii)] The map $G$ collapses the curves $\{\mu = \pm 2 w\}$ to the indeterminacy points $[0: \pm 2: 1]$ respectively.
\item[(iii)] The orbit of $C_1$ for both $F$ and $G$ is  finite and does not contain any indeterminacy points.
\item[(iv)] The line $\{\lambda=0\}$ is a curve a fixed points for $F$ and is mapped by $G$ the the line at infinity. 
\item[(v)] The line at infinity
(with the indeterminacy points removed) is  collapsed by $F$ and $G$ to the
vertical pole $q_v = [0:1:0]$ which is a fixed point for both maps. \end{enumerate}
\end{prop}

\begin{proof} Assertions (i),(ii) , (iv) and (v) follow from the expression \eqref{formula_grigor_F}, \eqref{formula_grigor_G} of $F$ and $G$. 
Assertion (iii) follows from the previous lemma together with the fact that $G = H \circ F$ and that  the points $[-2:0:1], [2:0:1]$ are both fixed by $H$.
\end{proof}

We now look at the dynamical behavior of $F$ near indeterminate points. 
Denote by $X$ the blow-up of $\P^2$ at the four points $[\pm 1: 1: 0]$ and $[0: \pm 2: 1]$ and by $\pi\colon X \to \P^2$ the associated (regular) map (see Appendix \ref{appendix_blow_up}).
Denote by $E_1, E_2, E_3, E_4$ the exceptional divisors over the points $[-1:1:0], [1:1:0]$, $[0:-2:1]$ and $[0:2:1]$ respectively.
We consider $\tilde{F},\tilde{G}$ the lifts of $F,G$ to $X$.

{\begin{figure}[!h] 
\centering
 \def\svgwidth{8cm}
   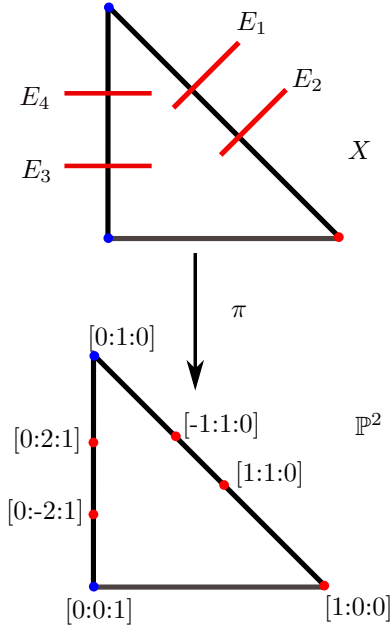
   \caption{ \label{figure_grigorchuk}Blow-up of $\mathbb{P}^2$ at the four points $[\pm 1 : 1 :0]$, $[0: \pm 2 : 1 ]$.}
\end{figure}
}

In the proposition below, we refer to Appendix \ref{appendix_blow_up} for the notion of strict transform.
\begin{prop} \label{prop_grigor_lift_dynamics} The following assertions hold.
\begin{enumerate}
\item[(i)] The involution $H$  induces an automorphism of $X$, it exchanges  $E_1$ with $E_4$ and $E_2$ with $E_3$.
\item[(ii)] $\tilde{F}$ has one indeterminacy point on $E_3$  and $ E_4$, the two exceptional divisors $E_3$ and $E_4$ are fixed by $\tilde{F}$ and the restriction to these divisors has topological degree $2$.
\item[(iii)] The image of the indeterminacy points of $\tilde{F}$ on $E_3$ and $E_4$
  are the strict transform of the lines $\{ \lambda + \mu + 2 w=0 \} $ and $\{\lambda - \mu -
  2 w=0 \}$, respectively.  
 \item[(iv)] $\tilde{F}$ is regular on $E_1,E_2$ and maps these two divisors to the strict transform of the line $\{ \lambda = -2 w \}$ (each with multiplicity one).
 \item[(v)] The image of the indeterminacy point $[1:0:0]$ by
   $\tilde{F}$ is the strict transform of the line at infinity. 
 \item[(vi)] Both $\tilde{F}$ and $\tilde{G}$ are algebraically stable on $X$.
\end{enumerate}
\end{prop}


\begin{proof}[Proof of Proposition \ref{prop_grigor_lift_dynamics}]
Recall that we have denoted by $P_0 = 2 \lambda^2 w $, $P_1 = \mu (4 w^2 - \mu^2 + \lambda^2) $ and $P_2= w(4 w^2 - \mu^2)$ the homogeneous polynomials defining $F$.

Observe that (i), (ii)  and (iv) imply (vi). 
Observe that assertion (i), (iv) and (v) are  direct computations. We leave assertion (i) and (iv)  to the reader and prove  assertion (v).

Let us blow-up the point $[1:0:0]$, we choose 
 some local coordinates  $(e = \mu/\la, l = w/\mu )$ such that the  the exceptional divisor   over $[1:0:0]$ has local equation $e=0$. In these blow-up coordinates, the map $F$ composed with the blow-down is given by:
 \begin{equation*}
 (e, l) \mapsto [1: e : l e]\in \bP^2 \mapsto [P_0(1, e , le) : P_1(1,e,le) : P_2(1,e, le)] \in \bP^2, 
 \end{equation*}
 and we obtain:
 \begin{equation*}
 (e,l)  \mapsto  [2 l: 1 - e^2 + 4 e^2 l^2: e^2 l (-1 + 2 l) (1 + 2 l) ]. 
 \end{equation*}
In particular, the image of the exceptional divisor $e=0$ by this map is parametrized by $l \mapsto [2l: 1 :0]$ and assertion (v) holds.
\medskip

Let us prove (ii) and (iii) for the exceptional divisor $E_3$, we fix some local coordinates near $E_3$. 
Take $(e = \lambda/ w , l = (\mu/w +2)/(\lambda/w))$ so that $E_3 = \{e=0 \}$, we write $F \circ \pi$ in these coordinates:
\begin{equation*}
F \circ \pi \colon (e , l) \mapsto [P_0(e , -2 + l e, 1): P_1(e , -2 + le , 1): P_2(e, -2 + le , 1)].
\end{equation*}
We obtain:
\begin{equation*}
F \circ \pi\colon (e , l) \mapsto [2 e: -(-2 + e l) (-e - 4 l + e l^2): -l (-4 + e l)].
\end{equation*} 
In particular, the restriction to $E_3$ is of the form:
\begin{equation*}
F \circ \pi \colon (e=0 , l) \mapsto [0: -2l: l] = [0:2:1],
\end{equation*}
when $l \neq 0$. As a result $F \circ \pi$ contracts $E_3$ to the point $[0:-2:1]$. 
We can thus compute the lift $\tilde{F}$ in these coordinates as $\tilde{F}$ maps $E_3$ to $E_3$, which is obtained from the following expression.
\begin{equation*}
\tilde{F} \colon (e, l) \mapsto \left ( e' = \dfrac{P_0(e , -2 + l e,  1)}{P_2(e, -2 + le , 1)} ,
 \ 
  l' = \dfrac{2 + P_1(e , -2 + l e, 1)/ P_2(e , -2 + l e, 1)}{P_0(e , -2 + l e, 1)/ P_2(e , -2 + l e, 1)} \right ).
\end{equation*}
We thus obtain:
\begin{equation*}
\tilde{F} \colon (e, l) \mapsto \left ( e'= -\dfrac{2e}{l (-4 + e l)}, \  l'=\dfrac{ -2 + e l + 4 l^2 - e l^3}{2}\right ).
\end{equation*}
The above formula proves that $l=e=0$ is an indeterminacy point of $\tilde{F}$. Blowing-up this point (e.g writing $e = e_1, l= l_1 e_1$ in $F\circ \pi$) gives the image of the indeterminacy point by $\tilde{F}$, and the computation is direct. 
The restriction to $E_3$ also yields:
\begin{equation*} \label{eq_restriction}
\tilde{F} \colon (e=0, l ) \mapsto \left (0 ,  -1 + 2 l^2 \right ).
\end{equation*}
This proves that $E_3$ is mapped  to itself with multiplicity $2$ by $\tilde{F}$ (i.e the restriction of $\tilde{F}$ to $E_3$ has topological degree 2), we have thus  proven assertion (ii) and (iii) for the exceptional divisor $E_3$.
Similar computation holds for the determination of the image of $E_4$.
\end{proof}

 We now use the dynamical features of $F$ above to find two invariant fibrations.

Let  $D_1$ be the pencil of conic in $\bP^2(\C)$ passing through all four points \\
$[\pm 1:1:0], [0: \pm 2 :1]$, and let $D_2$ be the pencil of conics in $\bP^2(\C)$ passing through all four points $[\pm 1: 1 :0] , [\pm 2: 0:1]$. 
We will now show that both pencils are invariant under $F$.
A general algebraic-geometric view of this phenomenon will be
given in \S \ref{subsection_AG_grigorchuk}. 

\begin{prop} \label{prop_invariant_pencil_grig} The two pencils $D_1,D_2$ are invariant under $F$. 
\end{prop}

\begin{proof}
Observe that the vertical line $\{ \lambda =0 \}$ is a line of superattracting fixed points. 
So the image of any  conic passing through $[0, \pm 2, 1]$ also passes through those points (see Proposition \ref{prop_grigor_lift_dynamics} (iii)). 
Take $C$ a conic in the pencil $D_1$. Since  the horizontal lines $\{\mu = \pm 2 w\}$ do not belong to the pencil $D_1$, Bezout's theorem proves that $C$ intersects each of those lines at $2$ points. 
Since these lines are collapsed to $[\pm 1: 1: 0]$ by assertion (i) of Proposition \ref{prop_contracted_grigorshuk}, this proves that the image of $C$ by $F$ passes through the  two points $[\pm 1: 1: 0]$. 
We have shown that for any conic in the pencil $D_1$, its image by $F$ passes through all four points $[\pm 1:1:0] , [0:\pm 2: 1]$. 
Let us now argue that the image of any conic in the pencil $D_1$ by $F$ is also a conic, i.e is also a curve of degree $2$. 
Since the curve $C$ passes through all four points $[\pm 1:1:0]$ and
$[0:\pm 2: 1]$,
we can calculate (using Proposition \ref{prop_formula_push_pull} (v))
the class of $C$ in $H^{1,1}(X)$:
\begin{equation*}
C = 2 \tilde{L}_\infty +E_1 + E_2 -E_3 - E_4 \in H^{1,1}(X),
\end{equation*}
where $\tilde{L}_\infty$ is the strict transform of the line at infinity by the blow-up at the four points $[\pm 1: 1: 0], [0: \pm 2: 1]$. 
By Proposition \ref{prop_grigor_lift_dynamics}.(iv), the divisors $E_1 $ and $E_2$ are mapped by $\tilde{F}$ to a line in $X$, the exceptional divisors $E_3,E_4$ are fixed by $\tilde{F}$, their indeterminacy point are mapped to a line and the image of the indeterminacy point $[1:0:0]$ is the line at infinity. 
This implies that the image of $C$ by $F$, denoted $F_* C$ is given by:
\begin{equation*}
F_* C = 2 F_* \tilde{L}_\infty + F_* E_1 + F_* E_2 - F_* E_3 - F_*E_4 = (2 +1 + 1 -1 -1) L_\infty \in H^{1,1}(\P^2).
\end{equation*}
In conclusion, $F$ maps a conic in $D_1$ to a conic passing through the same four points, so the pencil $D_1$ is preserved by $F$.
\medskip

Let us now prove that the pencil $D_2$ is also invariant. 
The same argument proves that any conic in $D_2$ has an image of degree $2$ which passes through the points $[\pm 1, 1, 0]$. 
Take $C$ a conic in the pencil $D_2$. Since the conic $C_1:=\{\lambda^2 - \mu^2 + 4w^2 =0 \}$ does not belong to the pencil $D_2$, Bezout's theorem proves that $C$ intersects $C_1$ at four points. 
By Lemma \ref{lem_curve_1}, the curve $C_1$ is collapsed by $F$ to $[-2:0:1]$, so the image $F(C)$ passes through that point. Moreover, the point $[2:0:1]$ is a fixed point for $F$, so the image $F(C)$ also passes through that point. 
Overall, we have shown that any conic in the pencil $D_2$ is mapped by $F$ to a conic passing through all four points $[\pm 1: 1: 0], [\pm 2: 0:1]$, hence $F$ preserves the pencil $D_2$, as required. 
\end{proof}

We obtain an explicit characterization of the two pencils $D_1$ and $D_2$. 

\begin{cor} \label{cor_grigorchuk_pencil} The pencil $D_1$ and $D_2$ are parametrized respectively by two rational maps $\phi \colon \P^2 \dashrightarrow \P^1$ and  $\psi \colon \P^2 \dashrightarrow \P^1$ defined by the formulas \eqref{def_phi} and \eqref{def_psi}.
Moreover, $\phi \circ F = \phi, \psi \circ F = t \circ \psi$ where $t$ is the Chebyshev map,
\end{cor}

To go further, we need to parametrize holomorphically the fibers of the map $\phi$ in order to find an appropriate conjugate for $F$.
Recall that the point $[2:0:1]$ is fixed by $F$ and that the point $[-2:0:1]$ is mapped by $F$ to that point. 
These two points correspond to the repelling fixed point and its preimage for the Chebyshev map $2 z^2 -1$.

Consider two simply connected domains $U_1, U_2$ of $\C^2$ such that $U_1 \cup U_2 = \C^2$ and fix two determinations of the logarithm on each of these domains so that the square root $\sqrt{\eta^2 -1}$ is well-defined on $U_1, U_2$ respectively. 
 
For each $i=1,2$ and any point $p \in U_i \cap \phi^{-1}(\eta)$, we take $\varphi_\eta(p) \in \P^1$ to be the slope of the line joining  $p$ and the point $(2,0) \in \C^2$. We normalize in such a way that the tangent line to the hyperbola $\phi^{-1}(\eta)$ at $[2:0:1]$ is mapped to $[1:1]\in \P^1$, such that the point $[0:2\eta + 2\sqrt{ \eta^2 -1}:1]$ is mapped to $[0:1] \in \P^1$ and the point $[0:2\eta - 2\sqrt{ \eta^2 -1}:1]$ is mapped to the point at infinity $[1:0] \in \P^1$. 

\begin{lem} \label{lem_precise_conjugation_grig} The following properties hold.
\begin{enumerate}
\item[(i)]For each $i =1,2$ and for any $\eta \in \C$, the map $\varphi_\eta \colon U_i \cap \phi^{-1}(\eta) \to \C$ is an analytic function of the form:
\begin{equation*}
\varphi_\eta \colon (\lambda , \mu ) \in \C^2 \cap \phi^{-1}(\eta) \mapsto \dfrac{ 2 - \lambda - \eta \mu - \mu\sqrt{ \eta^2-1} }{ -2 + \lambda + \eta \mu -  \mu\sqrt{ \eta^2-1}}.
\end{equation*}
\item[(ii)]For each $i =1,2$ and for any $\eta \in \C$, the inverse  $\varphi_\eta^{-1}\colon \C \to U_i \cap \phi^{-1}(\eta)$ of $\varphi_\eta$ is given by:
\begin{align*}
z  \mapsto \left ( \lambda=-\dfrac{4 (-1 + \eta^2) z}{
 1 + z^2 + \eta \sqrt{-1 + \eta^2} (-1 + z^2) - \eta^2 (1 + z^2))} , \right . \\
\left . \mu = \dfrac{2 \sqrt{-1 +
   \eta^2} (-1 + z) (1 + z)}{1 + z^2 + \eta \sqrt{-1 + \eta^2} (-1 + z^2) - 
 \eta^2 (1 + z^2)} \right ).
\end{align*}
\item[(iii)] For generic $z \in \C$ and for all $\eta \in \C$, one has:
\begin{equation*}
\varphi_\eta \circ F \circ\varphi_{\eta}^{-1} (z) = z^2.
\end{equation*}
\item[(iv)] For generic $z \in \C$ and for all $\eta \in \C$, one has:
\begin{equation*}
\psi \circ \varphi_{\eta}^{-1} (z) = \dfrac{1}{2} \left ( z + \dfrac{1}{z} \right ).
\end{equation*}
\end{enumerate}

\end{lem}

\begin{proof} Let us describe how one can obtain assertion (i).
Let us denote by $l$ the slope $l:=(\lambda-2)/\mu$. 
At the two point $[0: 2 \eta + 2 \sqrt{\eta^2-1}:1], [0: 2 \eta - 2 \sqrt{\eta^2-1}:1] $, the slopes $l_+, l_-$ are given by:
\begin{equation*}
l_+ := - \eta + \sqrt{\eta^2 -1}\ ; \ l_-:= -\eta - \sqrt{\eta^2 -1}.
\end{equation*}
One checks that the slope $l$ of the hyperbola at the point $[2:0:1]$ is $-\eta$. 
Since the Mobius transformation $z \mapsto (z - l_-)/(-z + l_+)$ which maps the triplet $(l_+, l_-, -\eta)$ to the triplet $(0,\infty, 1)$ on $\mathbb{P}^1$, we obtain 
 $\varphi_\eta$ by applying this Mobius transformation to $l=(\lambda-2) / \mu$:
\begin{equation*}
\varphi_\eta (\lambda,\mu) := \dfrac{\left ( \dfrac{\lambda-2}{\mu} \right ) + \eta + \sqrt{\eta^2 -1}}{- \left ( \dfrac{\lambda-2}{\mu} \right )-\eta + \sqrt{\eta^2 -1}  }. 
\end{equation*}
For assertion (ii), one determines the inverse is obtained by first solving the system of equation
\begin{equation*}
\left \lbrace \begin{array}{ll}
4 - \lambda^2 + \mu^2 = 4 \mu \eta ,\\
\lambda-2 = l \mu.
\end{array} \right .
\end{equation*}
This determines $\lambda, \mu $ as a function of $l$ and then  one precompose by the Mobius transformation $z \mapsto (l_+ z + l_-)/(z+1)$ which maps the triplet $(0,\infty, 1)$ to $(l_+,l_-, -\eta)$.
The last two assertions (iii) and (iv) also follows from  direct computations. 
\end{proof}

Recall that the map $F$ leaves each fiber $\phi^{-1}(\eta)$ invariant, so that one has the following commutative diagram for each $i=1,2$:
\begin{equation*}
\xymatrix{ \C \ar[r]^{} \ar[d]^{\varphi_\eta^{-1}} \ar[r]^{\varphi_\eta \circ F \circ\varphi_{\eta}^{-1}}  & \C \\
 U_i \cap \phi^{-1}(\eta)  \ar[r]^{F} &  U_i \cap \phi^{-1}(\eta) \ar[u]^{\varphi_\eta} .}
\end{equation*}

\begin{proof}[Proof of Theorem \ref{thm_conjugate_grigorchuk}]

Consider the analytic map $\varphi$ given by
 $$\varphi(\lambda, \mu) := \left ( \eta := \phi(\lambda,\mu) , z:= \varphi_{\phi(\lambda,\mu)}(\lambda, \mu) \right ). $$
Using assertion (iii), (iv) of the previous lemma, we deduce that
$F$ and $(\eta, z ) \mapsto (\eta , z^2)$ are conjugate on each $U_i$ via $\varphi$ and that $\psi \circ\varphi_\eta^{-1}(z) = 1/2 (z+1/z)$, as required.
\end{proof}

\subsection{Structure of the map $F$}

%
%

Recall from the last section that we have found a map $\pi\colon \C^2 \dra \C^2$ such that
 $ \pi\circ F = \cheb \circ \pi$, where
$\cheb$ is the Chebyshev map (\ref{cheb}) where $\pi$ is defined by
$$
       \eta = \phi (\la, \mu) =   \frac{4-\la^2+\mu^2} {4\mu} \ , \quad 
    \theta =  \psi(\la, \mu) =  \frac {4-\mu^2+\la^2} {4\la}  \ .
$$
This is a  rational map of algebraic degree 3, of topological degree $2$  with the following features:

\ssk\nin $\bullet$
It is equivariant with respect to the reflections 
 $(\la, \mu) \mapsto (\mu, \la)$ and $(\eta, \theta) \mapsto (\theta, \eta)$. 

\ssk\nin $\bullet$
 It has the following  indeterminacy points: 
 two vertices $[0:2:1]$ and $[0:-2:1]$ on the line $\{\lambda=0\}$, $[1:0:0]$
 and two ``diagonal'' points at the line  $A_\infty$ at infinity:
$ d_\pm = [1: \pm 1: 0]$;


\ssk\nin $\bullet$
The vertical axis $A^\ver = \{ \la =0\}$ is collapsed 
 (after puncturing out the indeterminacy points)  
to the vertical pole $q_v = [0:1:0]$.
Symmetrically,  the horizontal axis $A^\hor$ is collapsed to the
horizontal pole $q_h = [1:0:0]$.  

\ssk\nin $\bullet$
The pullback of a vertical line $L^\ver_\eta$ through $(\eta,0)$
is a rational algebraic curves (a ``vertical hyperbola'' )
$$
     \hat L^\ver_\eta = \{ 4-\la^2+\mu^2 - 4\eta \, \mu  = 0\}
$$
 union the vertical axis $A^\ver $.
Symmetrically, the pullback of a horizontal line $L^\hor_\theta$ through $(0,\theta)$
is a horizontal  hyperbola
$$
     \hat L^\hor_\theta = \{ 4-\mu^2+\la^2 - 4\theta \, \la = 0\}
$$
 union the horizontal axis $A^\hor $. 
The projection $\pi$ is a degree two branched covering of each of the hyperbolas onto its image.

\ssk\nin $\bullet$
  The vertical hyperbolas $\hat L^\ver_\eta$  form a {\em pencil} through the points
  $a_\pm= (\pm 2, 0)$, i.e., all of them pass through these points,
  and form a holomorphic foliation of $\C^2 \sm \{0,  a_\pm, b_\pm\}$.
  [From the projective point of view, they form a pencil through four
  points (one should add the diagonal  points $d_\pm = [1:\pm 1:  0]$ at infinity)
 forming a foliation of the same space, $\C^2 \sm \{ 0, a_\pm, b_\pm\}$.] 
    The description of the pencil of horizontal hyperbolas is
    symmetric
    (with respect to the reflection $(\eta, \theta)\mapsto(\theta, \eta)$).  



\ssk\nin $\bullet$
Each real vertical hyperbola $\hat L_\eta^{\ver, \R}$ with $|\eta|< 1$ projects under $\pi$ to the
interval $\{ \eta\} \times \I $
 whose endpoints correspond to the points  $a_\pm$. 
For $|\eta|> 1$, the  hyperbola $\hat L_\eta^{\ver, \R}$  projects  to the
complement of this interval, $\{\eta\} \times (\hat \R\sm \inter \I) $.
The picture for  the real horizontal hyperbolas is symmetric.

It follows that $\pi(\RP^2)$ is the union of the square 
$\{ |\eta|\leq 1, \ |\theta| \leq 1\} $ and  four quadrants attached to its vertices. 
\begin{figure}[h!]
\includegraphics[scale=0.5]{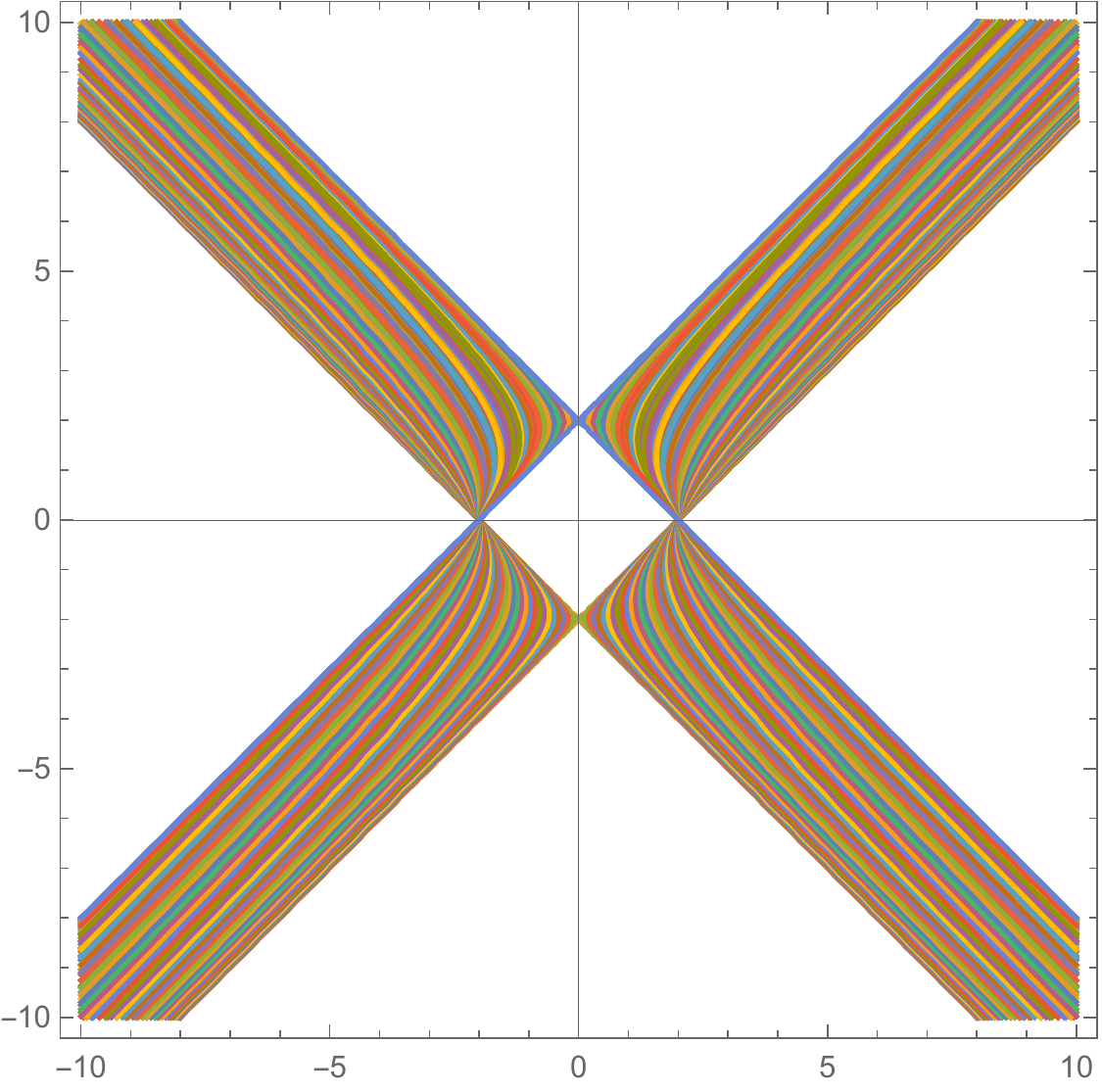}
\caption{Vertical hyperbolas $\hat{L}_\eta^{\ver}$ for $-1 \leqslant \eta \leqslant 1$.}
\end{figure}

\begin{figure}[h!]
\includegraphics[scale=0.5]{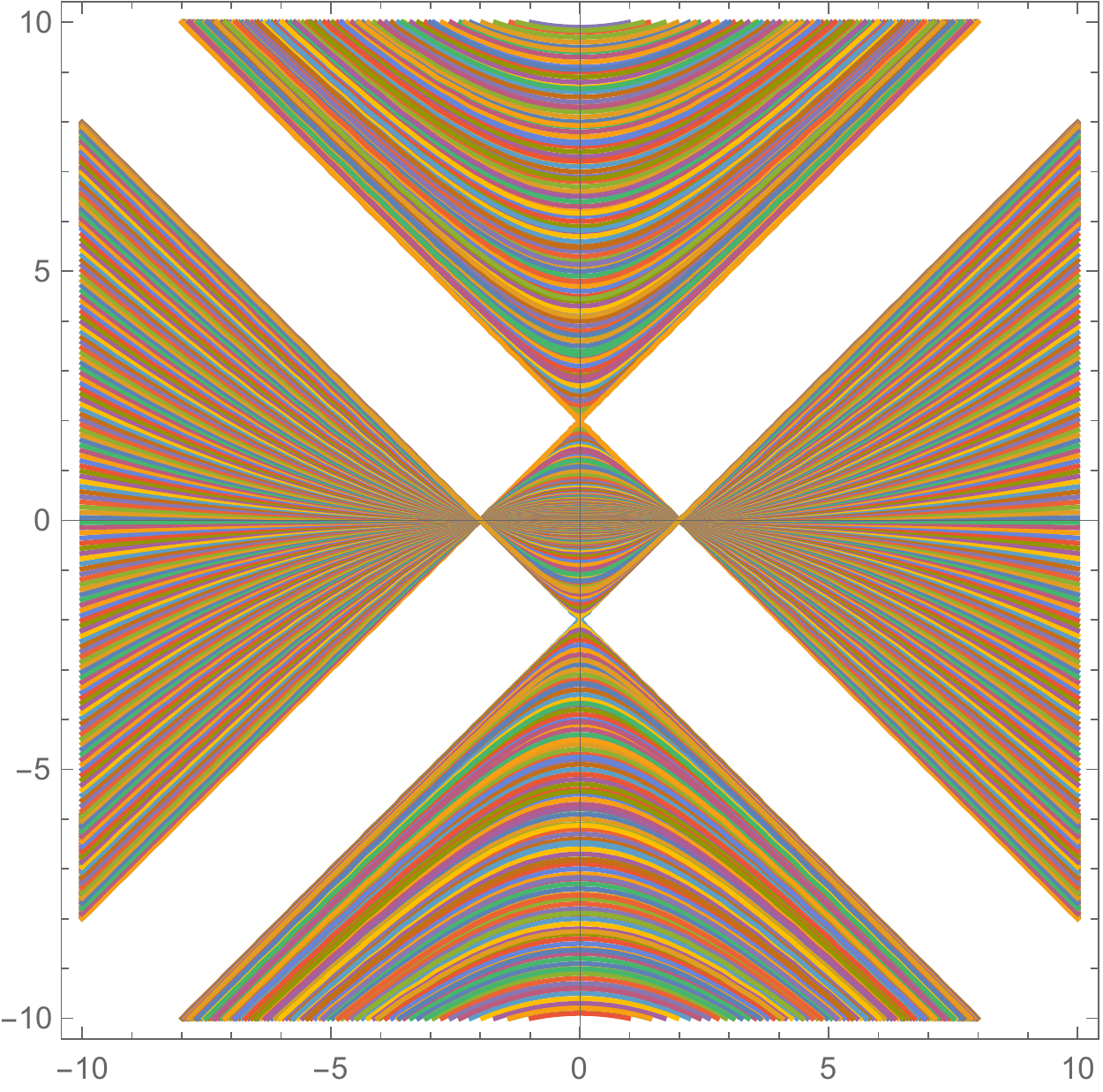}
\caption{Vertical hyperbolas $\hat{L}_\eta^{\ver}$ for $ |\eta|\geqslant  1$.}
\end{figure}


Now the dynamics of $F$ can be readily understood as the lift by $\pi$
of the Chebyshev dynamics:
%

%
%

\ssk\nin $\bullet$
The vertical foliation $\hat \LL^\ver := \bigcup_{\eta \in \C}  \hat L^\ver_\eta $ is {\em leafwise invariant},
and $F$ restricts to a degree two rational endomorphism on
each leaf $\hat L_\eta^\ver$.  This endomorphism has two superattracting fixed points,
the  intersections of $\hat L_\eta^\ver$ with  the vertical  axis $A^\ver $. Hence it is
conformally conjugate to the squaring map $f_0 \colon z\mapsto z^2$ of
$\hC$.  In this coordinate, the projection $\pi\colon \hat L_\eta\ra L_\eta$ becomes the
canonical semi-conjugacy between $f_0$ and~$\cheb$.

\ssk\nin $\bullet$
The horizontal foliation $\hat \LL^\hor := \bigcup_{\theta\in \C} \hat{ L_\theta}^\hor$ is $F$-invariant,
with the leaves transformed by the Chebyshev map:
$$
F(\hat L_\theta) ) = \hat L_{\cheb (\theta) }. 
$$ 

\ssk\nin $\bullet$
Let $W_{\eta\pm}^s$ be the superattracting basins of $F\cond \hat
L_\eta^\ver$ (with ``$+$'' corresponding to, say, the fixed point with $|\mu|>2$).
 Since the orbits in the disks $W_{\eta\pm}^s $ converge to the corresponding fixed
 points on $A^\ver $, these disks get interpreted as the {\em global superattracting
 manifolds}  of these fixed points.     
 

\ssk\nin $\bullet$
The action of $F$ on the real hyperbolas $\hat L_\eta^{\ver ,\R}$ with $|\eta|>1$ is real conjugate
to the map $f_0\colon x\ra x^2$ on $\R$.
For $|\eta|< 1$ it is real conjugate
to the map $\displaystyle {f_0\colon x\ra \frac {1}{2}  \left(  x-\frac
    {1}{x} \right)  } $ on $\R$
  (which is in turn conjugate to $f_0$ on the unit circle $\T$).%
\footnote{Incidentally, this map describes the Newton method for
  finding $\pm i$, the roots of $z^2+1$.}

\ssk\nin $\bullet$
The Julia set of  $F$ is equal to 
$$ 
 \Jul (F) = \pi^{-1} (\Jul (\Cheb) = \pi^{-1} (\C\times \I )). 
$$
This is a real-symmetric 3D variety $\MM$ that can be described   as
follows. Let $\hat \I$ be the union of  four semi-strips in $\R^2$.  
It is projected by $\psi$ to the interval $\I$, and the fibers of this
projection are real horizontal  hyperbolas. Complexifying these
hyperbolas, we obtain $\hat \I$. In this way, $\MM$ gets interpreted
as the  {\em complexification of $\hat \I$ along the horizontal foliation}.  

\ssk\nin $\bullet$
Thus, $\MM$ is foliated by (complex)  horizontal hyperbolas. 
This foliation has a global transversal, e.g., 
an interval $\TT_\la$, $\la>2$, which is the slice of one of the
half-strips of  $\hat \I$ by the
real vertical line through $(\la, 0)$.  

\ssk\nin $\bullet$
The transverse measure on $\C\times \I$  to the horizontal foliation $\LL^\hor$  
lifts to a transverse measure on $\MM$  to the horizontal foliation
$\hat \LL^\hor$. It is induced by the the 1-form 
$$
  \hat \om = \pi^* (\om) =  \frac { d\psi } {\pi \sqrt {1-\psi^2} } 
$$ 
restricted to $\MM$. 

Explicitly, the Green current is then given by the formula
\begin{equation*}
\Omega = \int_{-1}^{1} [4-\mu^2+\la^2 - 4\theta \, \la =0] \  \dfrac{d\theta}{\pi \sqrt{1 - \theta^2}}.
\end{equation*}

\subsection{The density of states via an equidistribution result for $F$}

Recall from \S \ref{section_schur_grigorchuk} that the sequence of polynomials associated to the density of states follows from the inductive formula:
\begin{equation} \label{eq_inductive_grigorchuk}
P_{n}(\lambda, \mu) = (4- \mu^2)^{2^{n-2}} P_{n-1} ( F(\lambda, \mu) ),
\end{equation}
where $P_0 = - \lambda + 2 - \mu$, $P_1 = (-\lambda + 2 - \mu )(\lambda +2 - \mu)$ and $n\geqslant 2$.

The density of states is then deduced from the zeros of the polynomials:
\begin{equation*}
\omega = \lim_{n\rightarrow +\infty} 
\dfrac{1}{2^n} \sum_{P_n(1, \mu)=0} \delta_{(\mu+1)/4}.
\end{equation*}
Note that one has to apply the transformation $\mu \mapsto (\mu+1)/4$ to get the density of states associated to the Grigorchuk group. 

\begin{thm}\label{equidistr for F} 
Let  $C$ and $S$  be  two  irreducible algebraic curves
in $\C^2$ 
 such that $C$ is not a  vertical hyperbola  
   while $S$ is not a  horizontal hyperbola and the intersection of $S$ with $C$ and the vertical line of fixed points is empty.   
Then 
\be\label{equidistr of intersections for F}  
     \frac 1 {2^n} \, [ (F^n)^* C \cap S ] \to (\deg C) \cdot
     (\deg S)\cdot \hat \om_S,  
\ee
where $\hat \om_S$ is the probability measure obtained by restricting  the $1$-form $$ \dfrac{d\psi}{\pi\sqrt{1-\psi^2}} $$ to $S$.
\end{thm} 

\begin{proof}
By Lemma \ref{equidistr for products} applied to the map $\id \times f$ where $f$ is the squaring map, we get the convergence of $(\id \times f^n)^{-1} (C) \cap S$ for any two generic curves $C, S$. 
Now since $F$ is locally analytically conjugate to $\id \times f$ by assertion (ii) of Theorem \ref{thm_conjugate_grigorchuk}, we conclude that the same property holds for $F$ using Lemma \ref{lem_slicing}.
\end{proof}

Take $A \colon \mu \mapsto (\mu+1)/4$ the affine map. 
In case of  $C= L^\hor_\theta$, $S=\TT_\la$ we obtain the desired
equidistribution result.

\begin{thm} \label{thm_spectrum_grigorchuk}We have:
\begin{enumerate}
\item The sequence of currents $$\dfrac{1}{2^n} [P_n=0]$$
converges as $n\rightarrow+\infty$ to the Green current of $F$.
\item The  density of states associated with the Grigorchuk group is a multiple of  $A_* \hat\omega_S$ where the measure $\hat\om_S$  
corresponds to the slice of the Green current of $F$ by the line $S:=\{\lambda = 1\}$. Moreover, the support of this measure is a union of two intervals.
\end{enumerate}
\end{thm}

\begin{proof}
Observe that $P_1 = 4 \mu (\phi -1)$ and that $[P_1 =0] = [\phi-1=0]$.
Since $F^* \phi = \phi$ and since $F^* \mu = 4 \lambda\mu \psi / (4 - \mu^2) $, $F^* \lambda = 2 \lambda^2/ (4-\mu^2)$, we obtain using \eqref{eq_inductive_grigorchuk}:
\begin{equation*}
P_2 = (4-\mu^2) F^* P_1 = 4 \lambda \mu \psi (\phi-1), 
\end{equation*}
hence
\begin{equation*}
\divi(P_2) = \divi (4 \lambda \mu \psi (\phi-1)). 
\end{equation*}
Since the term $(4-\mu^2)$ gets simplified in the previous calculation, we also deduce that:
\begin{equation*}
[P_2 = 0] = F^* [P_1=0]. 
\end{equation*}
By induction, using the fact that $F^* \psi = T \circ \psi$, we get:
\begin{equation*}
P_n= C_n \mu (\phi-1) \lambda^{2^{n-1}-1} \prod_{k=1}^{n-1} T^k \circ \psi,
\end{equation*}
where $C_n \in \C^*$ is a constant. 
The presence of the term $\mu\lambda^{2^{n-1} -1}$ gives that:
 $$[P_{n} =0]  = F^* [P_{n-1} = 0 ]. $$ 
 Finally we get:
 \begin{equation*}
 [P_n=0] = (F^n)^* [P_0=0]. 
 \end{equation*}
Applying the previous result, we deduce that the sequence of currents \\ $[P_n=0]$ converges to a multiple of the Green current and the sequence of measures $1/2^n[P_n(-1, \cdot) =0]$ converges to the slice of the Green current by the line $\{\lambda= -1\}$. We finally obtain the density of states by applying the appropriate affine transformation.  
\end{proof}

\section{The rational map associated with the Lamplighter group}

The map associated to the lamplighter map $F \colon \C^2 \to \C^2$ is defined as:
\begin{equation*}
F \colon (\lambda, \mu ) \mapsto \left (- \dfrac{\lambda^2 - \mu^2 -2 }{\mu - \lambda} , -\dfrac{2}{\mu - \lambda} \right ).
\end{equation*}
In homogeneous coordinates, $F$ is of the form:
\begin{equation*}
F:= [\lambda, \mu , w] \mapsto [-\lambda^2 + \mu^2 + 2 w^2 , -2 w^2, (\mu - \lambda) w],
\end{equation*}
and it has topological degree $1$ and algebraic degree $2$.

\subsection{Integrability of the map associated with the Lamplighter group}
\label{section_integrability_lamplighter}

Although the classical identity show directly that $F$ preserves a fibration, we also follow our systematic method in this case. 
The main result of this section is the following proposition.
\begin{prop} \label{prop_lamplighter_integrability} Take $\varphi \colon (\lambda, \mu) \mapsto (\lambda + \mu, \lambda - \mu)$, then the map $F$ is conjugate via $\varphi$ to the map:
\begin{equation*}
(\alpha, \beta) \mapsto \left ( \alpha,  \dfrac{\alpha \beta - 4}{\beta} \right ).
\end{equation*}
\end{prop}

We consider $X$ obtained from $\P^2$ by blowing up the two point $[\pm 1: 1: 0]$.
We consider the lift $\tilde{F}$ of $F$ to $X$, denote by $\pi\colon X \to \P^2$ the blow-down map onto $\P^2$ and by $E_1$, $E_2$ the two exceptional divisors such that $E_1$ is the exceptional divisor above $[-1:1:0]$, $E_2$ is above $[1:1:0]$ and $\tilde{L}_\infty$ is the strict transform of the line at infinity.

{\begin{figure}[h!]
\centering
 \def\svgwidth{8cm}
   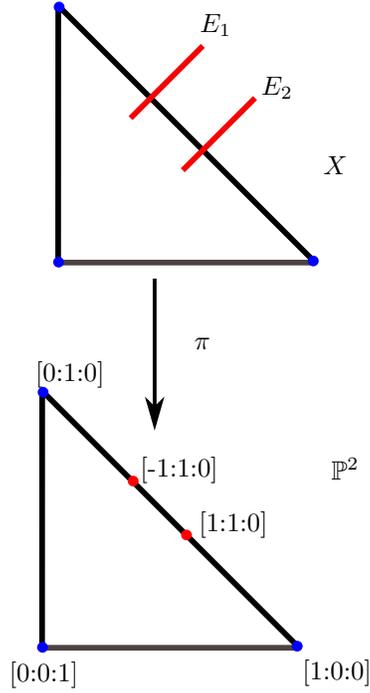
   \caption{ \label{figure_lamplighter}Blow-up of $\mathbb{P}^2$ at the two points $[\pm 1 : 1 :0]$.}
\end{figure}
}
\begin{prop} \label{prop_lamp} The following properties are satisfied. 
\begin{enumerate}
\item[(i)] $F$ is a birational map, i.e its topological degree is one.
\item[(ii)] $F$ has two indeterminacy points on $\P^2$ consisting of the two points $[\pm 1, 1, 0]$ at infinity. 
\item[(iii)] The only contracted curves for $F$ are the lines $\{ \lambda= \mu \}$ and the line at infinity. 
\item[(iv)] The strict transform of the line at infinity is contracted by $F$ to the fixed point $[1,0,0]$, the line $\{\lambda= \mu\}$ is collapsed by $F$ to the indeterminacy point $[-1,1,0]$.
\item[(v)] $\tilde{F}$ is regular near the strict transform of the line at infinity and  the image of  the line $\{ \lambda=\mu \}$ by $\tilde{F}$ is the exceptional divisor $E_1$. 
\item[(vi)] $\tilde{F}$ is regular on $E_1$ and maps $E_1$ to the line $\pi^{-1}(\{\mu=0\})$.
\item[(vii)] $\tilde{F}$ has one indeterminacy point on $E_2$ and collapses $E_2$ to the point\\
 $\pi^{-1}([1:0:0])$.
\item[(viii)] $\tilde{F}$ maps the indeterminacy point on $E_2$ to the strict transform of the line at infinity.
\item[(ix)] $\tilde{F}$ is algebraically stable on $X$.
\end{enumerate}
\end{prop}

\begin{proof}
Let us denote by $P_0 = -\lambda^2 + \mu^2 + 2 w^2 , P_1 =-2 w^2, P_2= (\mu - \lambda) w$ the three homogeneous polynomials defining $F$.

Assertions  (i), (ii), (iv), (v) are direct computations. 
Observe also that (iv), (v), (vi), (vii) imply that assertion (ix) since no curve is contracted by $\tilde{F}$ to an indeterminacy point. 
Assertion (iii) follows from the fact that the jacobian of the lift of $F$ to $\C^3$ is of the form:
\begin{equation*}
\Jac(F) = -8 (\lambda - \mu) w^2. 
\end{equation*}

Let us prove assertion (vi).

In the coordinate chart $(e = \lambda/\mu +1 , l = (w/\mu)/ (\lambda/\mu +1))$, the map $F\circ \pi$ is given by the expression:
\begin{equation*}
F \circ \pi \colon (e, l) \mapsto [P_0(e -1, 1 , le) :P_1(e -1, 1 , le): P_2(e -1, 1 , le) ].
\end{equation*}
We obtain:
\begin{equation*}
F \circ \pi \colon (e, l) \mapsto [-2 + e - 2 e l^2: -2 e l^2: -(-2 + e) l],
\end{equation*}
and in the chart $(e= w/\mu , l = (\lambda/\mu +1) / (w/\mu))$, the map $F \circ \pi$ is given by the expression:
\begin{equation*}
F \circ \pi \colon (e, l ) \mapsto [P_0(-1 + le , 1 , e): P_1(-1 + le , 1 , e): P_2(-1 + le , 1 , e)],
\end{equation*}
which simplifies as follows:
\begin{equation*}
F\circ \pi \colon (e, l ) \mapsto [-2 e - 2 l + e l^2: -2 e: 2 - e l].
\end{equation*}
In both charts, $F\circ \pi$ is well-defined and $F\circ \pi$ maps regularly $E_1$ to the line $\{\mu=0\}$. Since this line is disjoint from the indeterminacy points $[\pm 1, 1 , 0]$, we deduce that $\tilde{F}$ is regular on $E_1$ and maps $E_1$ to the preimage of the line $\{\mu=0 \}$. We have thus proved (vi).

\medskip

Let us prove assertion (vii).

Take a first chart $(e = \lambda/\mu -1 , l = (w/\mu)/ (\lambda/\mu -1))$, then $F\circ \pi$ is given by:
\begin{equation*}
F \circ \pi \colon (e, l) \mapsto [P_0(1 + e , 1 , l e):P_1(1 + e , 1 , l e):P_2(1 + e , 1 , l e)].
\end{equation*}
We obtain:
\begin{equation}
F\circ \pi\colon (e, l) \mapsto [2 + e - 2 e l^2: -2 e l^2: -e l],
\end{equation}
and $F\circ \pi$ is regular for near $e=0$ for all $l \in \C$. 
In the other chart $(e= w/\mu , l = (\lambda/\mu -1) / (w/\mu))$, we have:
\begin{equation*}
F \circ \pi \colon (e, l) \mapsto [P_0(1 +  l e , 1 , e):P_1(1 +  l e , 1 , e): P_2(1 +  l e , 1 , e)],
\end{equation*}
which simplifies as follows:
\begin{equation}
F\circ \pi \colon (e, l) \mapsto [-2 e + 2 l + e l^2: -2 e: -e l],
\end{equation}
and the latter expression has a unique indeterminacy at $(e=0, l=0)$.
One checks from the last expression that $E_2$ is mapped to $[1,0,0]$ by $F\circ \pi$. This finishes the proof of (vii).
\medskip

Let us prove assertion (viii). We blow up the indeterminacy point  $(e=0, l=0)$ on $E_2$ where $(e= w/\mu , l = (\lambda/\mu -1) / (w/\mu))$. Denote by $\pi' $ the blow-up of this point. 
Take $e = e_1, l = l_1 e_1$ where $e_1=0$ is the equation for the exceptional divisor, the map $F \circ \pi \circ \pi'$ is given in those coordinates  by:
\begin{equation*}
F \circ \pi \circ \pi ' \colon (e_1, l_1) \mapsto [P_0(1 + l_1 e_1^2, 1 , e_1): P_1(1 + l_1 e_1^2, 1 , e_1): P_2(1 + l_1 e_1^2, 1 , e_1) ].
\end{equation*}
We simplify the above formula and get:
\begin{equation*}
F \circ \pi \circ \pi' \colon (e_1,l_1) \mapsto [-2 + 2 l_1 + e_1^2 l_1^2: -2: -e_1 l_1]. 
\end{equation*}
The above expression is regular near $e_1=0$. Let us look near $l_1=\infty$, take $e = l_2 e_2, l = e_2$, so the map $F \circ \pi \circ \pi'$ is given in those coordinates  by:
\begin{equation*}
F \circ \pi \circ \pi ' \colon (e_2, l_2) \mapsto [P_0(1 + l_2 e_2^2 , 1 , e_2 l_2): P_1(1 + l_2 e_2^2 , 1 , e_2 l_2): P_2(1 + l_2 e_2^2 , 1 , e_2 l_2)].
\end{equation*}
We obtain:
\begin{equation}
F \circ \pi \circ \pi' \colon (e_2, l_2) \mapsto [2 - 2 l_2 + e_2^2 l_2: -2 l_2: -e_2 l_2],
\end{equation}
which is also regular near $e_2 =0, l_2 = 0 $. 
In particular, this proves that the map $F\circ \pi \circ \pi'$ is regular on the exceptional divisor above the indeterminacy point of $\tilde{F}$ on $E_2$. 
The above expression  also shows  
$$F\circ \pi \circ \pi' \colon (e_1 =0, l_1) \mapsto [-2 + 2l_1: -2: 0],$$
so the indeterminacy point of $\tilde{F}$ is mapped to the line at infinity. We have thus proved assertion (viii).  
\end{proof}

Consider $D$ the pencil of lines passing through the point at infinity $[-1:1:0]$. 

\begin{prop} The pencil of lines $D$ is preserved by $F$.
\end{prop}

\begin{proof}
Take a line $C$ belonging to the pencil $D$. We show that the image of $C$ by $F$ is a curve passing through the point $[-1:1:0]$. Observe that the line $\{\lambda = \mu \}$ intersects the line $C$, and since $\{\lambda = \mu\}$ is collapsed by $F$ to the point $[-1:1:0]$ by assertion (iv) of Proposition \ref{prop_lamp}, we deduce that the image of $C$ by $F$ is a curve passing through the point $[-1:1:0]$.

Let us show that the image of $C$ by $F$ is a line or equivalently that the curve $F(C \setminus I(F))$ is of degree $1$. 
Since $C$ is a line passing through the point $[-1:1:0]$, assertion (v) of Proposition \ref{prop_formula_push_pull} shows that its strict transform $\tilde{C}$ in $X$ satisfies:
\begin{equation*}
\tilde{C} = \tilde{L}_\infty + E_2 \in H^{1,1}(X).
\end{equation*} 
We have thus:
\begin{equation*}
\tilde{F}_* \tilde{C} = \tilde{F}_* \tilde{L}_\infty + \tilde{F}_* E_2. 
\end{equation*}
By assertion (iv) and (v) of Proposition \ref{prop_lamp}, the line at infinity is collapsed regularly by $\tilde{F}$ to a point, so we have
\begin{equation*}
\tilde{F}_* \tilde{L}_\infty = 0 \in H^{1,1}(X).
\end{equation*}
By assertion (vii) and (viii) of Proposition \ref{prop_lamp}, the divisor $E_2$ is collapsed and its indeterminacy point is mapped to $L$, so we have:
\begin{equation*}
\tilde{F}_* E_2 = \tilde{L}_\infty \in H^{1,1}(X).
\end{equation*}
This shows that $\tilde{F}_* \tilde{C} = \tilde{L}_\infty$, hence the line $C$ is mapped to a line by $F$.
\end{proof}

Observe that the member of the pencil $D$ are lines passing through $[-1:1:0]$, so each of those line is given by an equation of the form:
\begin{equation*}
\phi(\lambda,\mu):= \lambda + \mu = \alpha,
\end{equation*}
where $\alpha \in \C$.
One checks that $\phi \circ F = \phi$ so $\phi$ semi-conjugates $F$ to the identity. 

We now choose a transverse coordinate on each fiber of $\phi$, 
take $\psi \colon \C^2 \to \C$ the map given by:
\begin{equation*}
\psi(\lambda , \mu) = \lambda - \mu,
\end{equation*}
then the map $F$ is conjugate via $\phi \times \psi$ to the rational map:
\begin{equation*}
(\alpha, \beta) \mapsto \left ( \alpha,  \dfrac{\alpha \beta - 4}{\beta} \right ).
\end{equation*}

\subsection{The density of states for  the Lamplighter group}

Recall from Section \ref{section_schur_lamplighter} that the spectrum of the Schreier graph associated to the lamplighter group is related to a sequence of polynomials $P_n$ defined inductively as follows:
\begin{equation*}
P_n (\lambda, \mu) = (\mu - \lambda)^{2^{n-1}} P_{n-1}(F(\lambda,\mu)),
\end{equation*}
where $P_0 = 4 - \lambda - \mu$.
The spectrum associated to the lamplighter group is the limit of $1/2^n [P_n=0]$ with the line $\{\mu =0\}$. 
We thus recover the fact that the density of states is atomic, which was first proved by Grigorchuk-Zuk (\cite[Theorem 3]{grigorchuk_zuk_lamplighter}). 
Denote by $\omega_n$ the counting measure $1/2^n [P_n=0] \wedge [\mu=0]$ and denote by $\omega$ the limiting measure. 
Recall that we have defined in Proposition \ref{prop_twist}.(i) a laminar current $T_F$ on the elliptic cylinder.
\begin{thm} \label{thm_spectrum_lamplighter}The following properties are satisfied.
\begin{enumerate}
\item[(i)] The sequence of currents $1/2^n [P_n=0]$ converges as $n$ tends to $\infty$ to a current supported on countably many curves.
\item[(ii)] The sequence of measures $1/2^n [P_n =0] \wedge [\mu = 0]$ converges to an atomic measure. 
\item[(iii)] The sequence of measures $$\dfrac{2^n}{n}(\omega - \omega_n)$$
converges to the measure $T_F \wedge \{\mu = 0 \}$. 
\end{enumerate}
In particular, assertions {\rm (ii)}  shows that the  density of states of the Lamplighter group is atomic.
\end{thm}

\begin{proof}
By Proposition \ref{prop_lamp}, the map $F$ is conjugate to the map:
\begin{equation*}
(\alpha, \beta) \mapsto \left ( \alpha, \dfrac{\alpha \beta - 4}{\beta} \right ).
\end{equation*}
By assertion (ii) of Theorem \ref{thm_dynamical_degree}, we have that $\lambda_1(F) = \max (1,1) = 1 = \lambda_2(F)$ and $F$ is a birational map. 
Since $\lambda_1(F) =1 < 2$, Theorem \ref{thm_discrete} can be applied and yields assertion (i) and (ii).

Let us now prove assertion (iii). 
Observe that $P_0 = 4 - \lambda - \mu$ so its poles and zeros satisfy the relation:
\begin{equation*}
\divi(P_0 \circ \pi) = \pi^o [4 - \lambda - \mu =0] + E_1 - \tilde{L}_\infty.
\end{equation*}
Using the inductive relation, the fact that $\divi ((\mu - \lambda )\circ \pi) = \pi^o [\lambda - \mu =0] + E_2 - \tilde{L}_\infty$ and the equality $\tilde{F}^* \tilde{L}_\infty = \tilde{L}_\infty + 2E_2$, we obtain:
\begin{align*}
\divi(P_1 \circ \pi) = \pi^o [\mu - \lambda=0] + E_2 -\tilde{L}_\infty + \tilde{F}^* (\pi^o [4 - \lambda - \mu =0] + E_1 - \tilde{L}_\infty)  \\
= \pi^o [\mu - \lambda=0] - E_2 - 2\tilde{L}_\infty +\tilde{F}^* (\pi^o [4 - \lambda - \mu =0] + E_1) .
\end{align*}
We apply the above argument inductively and deduce:
\begin{align*}
\divi(P_n \circ \pi) = \sum_{k=0}^{n-1}  \left (2^{n-1-k} (\tilde{F}^k)^* \pi^o [\mu - \lambda=0] - 2^{n-1-k}(\tilde{F}^k)^* E_2\right ) - 2^n \tilde{L}_\infty  \\
+ (\tilde{F}^n)^* (\pi^o [4 - \lambda - \mu =0] + E_1)
\end{align*}
Pushing forward by $\pi$ the previous relation and looking only at the zero locus of $P_n$, we thus obtain that $$ \dfrac{1}{2^n} [P_n =0] = \sum_{k=0}^{n-1} \dfrac{1}{2^{1+k}} (F^k)^*\pi^o[(\mu-\lambda) = 0]   + \dfrac{1}{2^n}  (F^n)^* [P_0 = 0].$$ 
 By Proposition  \ref{prop_twist}, the currents  $ (2^n/n) (\tilde{F}^n)^*\pi^o[\mu - \lambda = 0]$, $(2^n/n) (\tilde{F}^n)^* [P_0 =0]$ converge to $T_F$ and their slice converge to $T_F \wedge [\mu=0]$ and assertion (iii) holds.
\end{proof}

\section{The rational map associated with the Hanoi group}

Recall that the renormalization transformation $F$ associated with the Hanoi group was computed given in Proposition \ref{prop_renorm_hanoi}. 
 The map $F$ induces a rational map on $\P^2$ given by:
 \begin{equation*}
 F := [x,y,z] \mapsto [P_0(x,y,z): P_1(x,y,z): P_2(x,y,z)],
\end{equation*}  
where $P_0, P_1, P_2$ are homogeneous polynomials of degree $4$ given by:
\begin{equation*}
\begin{array}{l}
P_0 :=x (x-z - y)(x^2 - z^2 + y z- y^2) + 2 y^2 (-x^2 + xz + y^2), \\
 P_1 := y^2z (x-z+y) ,\\
 P_2 := (x-z -y)(x^2 - z^2 + yz -y^2)z .
\end{array}
\end{equation*}
The map $F$ has algebraic degree 4, and  its topological degree is $2$. Note that the computation of the topological degree is not direct and follows from Theorem \ref{thm_conj_hanoi}.
\subsection{Integrability of the  map associated with the Hanoi group}
\label{subsection_integrability_hanoi}

The main result of this section is the following theorem.
\begin{thm} \label{thm_conj_hanoi} Take $\varpi \colon \C^2 \to \C^2$ the birational map  given by the following formula:
\begin{equation*}
(x,y) \mapsto \left (  \dfrac{x^2 - 1 - x y - 2 y^2}{y} , \   \dfrac{(1 + x - 2 y) (1 + x + y)}{2 y}\right ). 
\end{equation*}
 The map $F$ associated with the Hanoi group is conjugate via $\varpi$ to the rational map:
\begin{equation*}
(x,y) \mapsto \left ( x^2 - x -3 , \  \dfrac{(x-1)(x+2)}{x+3} y\right ). 
\end{equation*}
\end{thm}

To determine the map $\varpi$ in the above theorem, we need to explain the dynamical properties of the map $F$, first in $\P^2$ then on an appropriate blow-up of $\P^2$.

\begin{lem}
 The map $F$ has five indeterminacy points in $\P^2$, the points\\
  $[1:0:1]$, $[-1:0:1]$, $[-1:1:0]$, $[1: 1: 0]$ and $[2: 1: 0]$. 
\end{lem}
\begin{proof}
We observe that $z$ divides $P_1$ and $P_2$, and that $P_0(x, y ,0)$ factors as:
\begin{equation*}
P_0 (x,y,0) = (x+y)^2 (x^2 -3xy +2 y^2), 
\end{equation*}
which vanishes when $x=-y=1$ and $x=y=1$ and $x=2, y =1$. 
This proves that the indeterminacies of $F$ on the line at infinity are the points $[1:1:0]$, $[2,1,0]$, $[-1,1,0]$.
Observe that $y$ divides $P_1$ and that $(x-z)(x^2 - z^2)$ divides both $P_0(x,0,z)$ and $P_2(x,0,z)$. 
This proves that the indeterminacies of $F$ on the line $\{y=0\}$ are $[1:0:1]$ and $[-1:0:1]$.
Let us prove by contradiction that there are no  indeterminacy points outside the two lines $\{y=0 \}$ and $\{z=0\}$. 
Take $[x,y,z] \in \P^2\setminus \left ( \{y=0 \} \cup \{z=0 \}\right ) $ a point of indeterminacy of $F$. Since $P_2(x,y,z)=0$ and $y,z \in \C^*$, this proves that the point $(x,y,z)$ lies on the curve
\begin{equation*}
x+ y -z =0.
\end{equation*}
In particular, $z = x+y$ and we compute the polynomial $P_0(x, y , x+y), P_2(x,y,x+y)$:
\begin{equation*}
 \begin{array}{ll}
 P_0(x, y , x+y) = P_2(x,y, x+y )=2y^2 (x+y)^2.
 \end{array}
 \end{equation*} 
 Since $y \neq 0$, this proves that $x+y=0$ but this contradicts the fact that $x+y = z \neq 0$.
\end{proof}

\begin{lem} The Jacobian of $F$ is given by the homogeneous polynomial:
\begin{equation*}
\Jac(F) = 4 y (x - y - z) (x + y - z)^2 z (2 x^2 - 2 x y - 4 y^2 - y z - 
   2 z^2) (x^2 - y^2 + y z - z^2).
\end{equation*}
\end{lem}

  Consider the curves $$C_1 := \{x+y -z=0\},$$
$$C_2 :=\{-x+ y + z =0 \},$$
and 
$$C_3:= \{ x^2 - y^2 +y z - z^2=0\}.$$

\begin{prop} \label{prop_contracted_curves_hanoi} The following properties are satisfied.
\begin{enumerate}
\item[(i)] The only collapsed curves are $C_1, C_2, C_3$ and the line at infinity. 
\item[(ii)] The line at infinity is collapsed to the fixed point $[1:0:0]$.
\item[(iii)]The line $C_1$ is collapsed to the indeterminacy point $ [1:0:1]$.
\item[(iv)] The line $C_2$ is collapsed to the indeterminacy point $ [-1:1:0]$.
\item[(v)]The line $C_3$ is collapsed to the indeterminacy point $ [2:1:0]$.
\end{enumerate}
\end{prop}
\begin{proof}
The above statement follow from  direct computations.
\end{proof}

Consider $X$ the blow-up  of $\P^2$ at the four points $p_1 :=[-1:1:0], p_2:=[2:1:0]$ and $p_3:=[-1:0:1], p_4 := [1:0:1]$ and denote by  $\pi \colon X \to \P^2$ the blow-down morphism. Take $\tilde{F} := \pi^{-1} \circ F \circ\pi$ the lift of $F$ to $X$, we look at the dynamics of the collapsed curves.
We denote by $E_1, E_2, E_3, E_4$ the four exceptional divisors above $p_1, p_2, p_3$ and $p_4$ respectively, by $L_\infty$ the line at infinity in $\mathbb{P}^2$ and by $\tilde{L}_\infty$  its strict transform by $\pi$.

{\begin{figure}[h!]
\centering
 \def\svgwidth{8cm}
   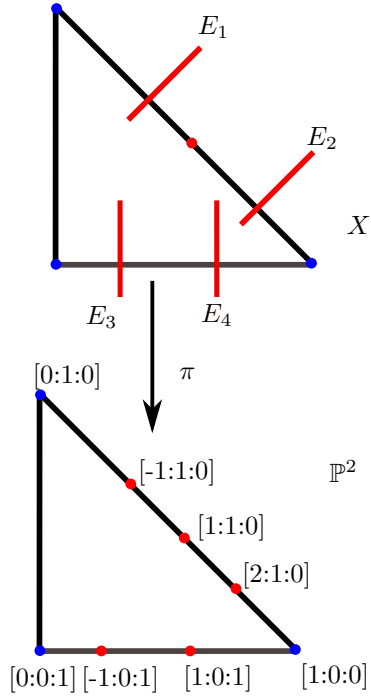
   \caption{ \label{figure_hanoi}Blow-up of $\mathbb{P}^2$ at the four points $[- 1 : 1 :0]$, $[2:1:0]$, $[\pm 1 :0 :1]$.}
\end{figure}
}

To simplify the notation, we shall write by $\tilde{C}_i$ the strict transform of the curves $C_i$ by $\pi$.
\begin{prop} 
\label{prop_dynamics_contracted_hanoi}
The following properties are satisfied.
\begin{enumerate}
\item[(i)] The curve $\tilde{C}_1$ is collapsed by $\tilde{F}$ to a point on the exceptional divisor $E_4$ whose orbit by $F_1$ is regular.
\item[(ii)] The line $\tilde{C}_2$ is mapped by $\tilde{F}$ to the exceptional divisor $E_1$.
\item[(iii)]  The curve $\tilde{C}_3$ is mapped by $\tilde{F}$ to the exceptional divisor $E_2$.
\item[(iv)] The exceptional divisor $E_1$ is mapped regularly by $F \circ \pi$ to the conic curve parametrized by:
\begin{equation*}
l \in \C \mapsto [6 - 3 l - l^2: -(-1 + l) l: 2 (2 - l) l] \in \P^2,
\end{equation*}
where $l$ parametrizes the slopes $z/(x+y)$. 
\item[(v)] The exceptional divisor $E_2$ is mapped regularly by $F \circ \pi$ to the  line \\
$\{ z = y\} \subset \P^2$. 
\item[(vi)] The map $\tilde{F}$ fixes $E_3$, it induces a map on $E_3$ with topological degree $2$ and  it has one indeterminacy point on $E_3$. The image of the indeterminacy point on $E_3$ by $\tilde{F}$ is the line parametrized by:
\begin{equation*}
l \mapsto [-22 + 2 l : -8 : 2(3- l)].
\end{equation*}
\item[(vii)] The map $\tilde{F}$ fixes $E_4$, it induces a map on $E_4$ with topological degree $2$ and  has two indeterminacy points on $E_4$. The image of those indeterminacy points are two respective lines, parametrized by:
\begin{equation*}
l \mapsto [-2 - 3 l : 2 : -3 l],
\end{equation*}
\begin{equation*}
l \mapsto [-17 + 3 l : -4 : -3 (3-l)].
\end{equation*}
\item[(viii)] The indeterminacy point on $L$ is mapped to $L$ by $F_1$.
\end{enumerate}
\end{prop}

We postpone the proof of this proposition to the end of the section.

Consider $D$ the pencil of conics  
passing through the four points $[\pm 1 :0: 1]$,\\
$[-1:1:0]$ , $[2:1:0]$.

\begin{prop} The pencil $D$ is invariant by $F$.
\end{prop}

\begin{proof} Take $C$ a curve in the pencil $D$. We first prove that $F(C)$ passes through all four points $[\pm 1 :0: 1],[-1:1:0] , [2:1:0]$.

 Since the line $\{y=0\} $ is a line of fixed point for $F$, we deduce that the image $F(C)$  passes through $[\pm 1: 0 :1]$.
Since the curves $C_2$ and $C_3$ do not pass through the point $[2:1:0]$, then neither $C_2$ nor $C_3$ belong to the pencil $D$, hence they intersect $C$ at $2$ and $4$ points respectively. Since $C_2$ and $ C_3$ are mapped by $F$ to the points $[-1:1 : 0]$ and $[2:1:0]$, we deduce that the image $F(C)$ also passes through these two points. 
We have thus proved that $F(C)$ is a curve passing through all four points $[\pm 1 : 0:1], [-1:1:0], [2:1:0]$. 

We now show that $F(C)$ is a  curve of degree $2$. 
Since $C$ is a member of the pencil passing through four points ,  assertion (v) of Proposition \ref{prop_formula_push_pull} shows that  its strict transform $\tilde{C}$ by $\pi$ satisfies:
\begin{equation*}
\tilde{C} = 2 \tilde{L}_\infty + E_1 + E_2 - E_3 - E_4 \in H^{1,1}(X).
\end{equation*}
Let us compute $(F\circ \pi)_* \tilde{C}$. 
By Proposition \ref{prop_dynamics_contracted_hanoi}.(viii) and by Proposition \ref{prop_contracted_curves_hanoi}, the line $L$ is collapsed to a point by $F \circ \pi$ and its indeterminacy point is mapped to $L$, as a result, we have:
\begin{equation*}
(F\circ \pi)_* \tilde{L}_\infty = L_\infty \in H^{1,1}(\P^2).
\end{equation*}
By Proposition \ref{prop_dynamics_contracted_hanoi}.(iv), the exceptional divisor $E_1$ is mapped regularly by $F \circ \pi$ to the conic curve parametrized by:
\begin{equation*}
l \mapsto [6 - 3 l - l^2: -(-1 + l) l: 2 (2 - l) l].
\end{equation*}
Since the above curves is of degree $2$, we have:
\begin{equation*}
(F\circ \pi)_* E_1 = 2 L_\infty \in  H^{1,1}(\P^2).
\end{equation*}
Using assertion (v), (vi), (viii) of Proposition \ref{prop_dynamics_contracted_hanoi}, we obtain:
\begin{equation*}
(F\circ \pi)_* E_2 = L_\infty \in  H^{1,1}(\P^2),
\end{equation*}
since $E_2$ is mapped to a line,
\begin{equation*}
(F\circ \pi)_* E_3 =  L_\infty \in  H^{1,1}(\P^2),
\end{equation*}
since $E_3$ is collapsed to a point but its indeterminacy point is mapped to a line,
\begin{equation*}
(F\circ \pi)_* E_4 =  2 L_\infty \in  H^{1,1}(\P^2),
\end{equation*}
since $E_4$ is collapsed to a point but it has two indeterminacy points which are mapped to lines.
Overall, we have:
\begin{equation*}
(F\circ \pi)_* \tilde{C} = (2 + 2 + 1 -1 -2) L_\infty = 2 L_\infty \in H^{1,1}(\P^2).
\end{equation*}
This proves that $F$ maps $C$ to a degree $2$ curve, as required.
\end{proof}

Denote by $\psi \colon \P^2 \dashrightarrow \P^1$ the map associated with the pencil $D$.

\begin{cor} The map $\psi$ is of the form:
\begin{equation*}
\psi \colon [x: y : z] \dashrightarrow [x^2 - z^2 - x y - 2 y^2 :y z].
\end{equation*}
Moreover, one has $F \circ \psi = \psi \circ g $ where $g \colon \P^1 \to \P^1$ is the rational map given by the polynomial $$g =  x^2 -x -3.$$
\end{cor}

We now pursue our study by finding an appropriate parametrization of the fibers of $\psi$.  
To do so, let us first observe that any two distinct curves  belonging to  the pencil $D$ or equivalently any two distinct fiber of $\psi \circ \pi$ intersect the divisors $E_3 , E_4$ at distinct points.

For each $\alpha \in \P^1$, let us denote by $s_0(\alpha)$ the intersection  of the fiber  $(\psi \circ \pi)^{-1}(\alpha) $ with $E_3$ and by $s_\infty(\alpha)$ the intersection of the fiber $(\psi \circ \pi)^{-1}(\alpha) $ with $E_4$. 
By construction, the two functions 
\begin{equation*}
s_0 , s_\infty \colon \P^1  \dashrightarrow X,
\end{equation*}
define two marked points on the family of conics $\psi^{-1}(\alpha)$. 
Recall that the conics $\psi^{-1}(\alpha)$ of $\P^2$ are rational, we shall consider an explicit rational map from $\psi^{-1}(\alpha)$ to $\P^1$. 
In this case, for any point $p  $ on a conic $\psi^{-1}(\alpha)$, we associate  the line $L_p$ passing through $[1:0 : 1]$ and $p$. Since the set of lines passing through $[1:0:1]$ is parametrized by $\P^1$, we obtain a map from $\psi^{-1}(\alpha)$ to $\P^1$.

More explicitly, for each $\alpha = [\alpha_0 : \alpha_1] \in \P^1$, we define the function \\
$\varphi_\alpha \colon \psi^{-1}(\alpha) \dashrightarrow \P^1$ 
induced by the following transformation:
\begin{equation*}
\varphi_\alpha \colon (x ,  y )  \in \C^2 \cap  \psi^{-1}(\alpha) \mapsto  \dfrac{x+1 }{y} + \dfrac{\alpha_0/\alpha_1 - 1 }{2} \in \C.
\end{equation*} 
In homogeneous coordinates, $\varphi_\alpha$ is of the form:
\begin{equation*}
[x:y :z] \in \psi^{-1}(\alpha) \mapsto [2 \alpha_1 x + 2 \alpha_1 z + (\alpha_0 - \alpha_1) y: 2 \alpha_1 y] \in \P^1. 
\end{equation*}
In fact, the map $\varphi_\alpha$ is constructed via the above construction and was normalized in such a way that the marked points $s_0(\alpha) , s_\infty(\alpha)$ are identified to the point $[0:1]$ and $[1:0]$.

\begin{lem} The following properties hold. 
\begin{enumerate}
\item[(i)] For each $\alpha = [\alpha_0: \alpha_1] \in \P^1 \setminus \{[3:1 ] , [1:0] \}$, the map $\varphi_\alpha \colon \psi^{-1}(\alpha) \dashrightarrow \P^1$ is invertible and its inverse is given by the formula: 
\begin{multline*}
\varphi_\alpha^{-1}  \colon [z_0: z_1] \in \P^1 \mapsto [4 \alpha_1^2 z_0^2 + 4 \alpha_1^2 z_0 z_1 - \alpha_0^2 z_1^2 + 
 9 \alpha_1^2 z_1^2: \\
  8 \alpha_1^2 z_0 z_1:
   (2 \alpha_1 z_0 - \alpha_0 z_1 - 3 \alpha_1 z_1) (2 \alpha_1 z_0 - 
   \alpha_0 z_1 + 3 \alpha_1 z_1)] \in \psi^{-1}(\alpha).
\end{multline*} 
\item[(ii)] For each $\alpha  \in \P^1 \setminus \{[3:1 ] , [1:0] \}$, the points $s_0 (\alpha), s_\infty (\alpha)$ are mapped by $ \varphi_\alpha \circ \pi$ to the point $[0:1]$ and $[1:0]$ respectively.
\end{enumerate}

\end{lem}

\begin{proof} Assertion (i) follows from solving the system of equation:
\begin{equation*}
\left  \{ \begin{array}{ll}
(x^2 - 1 - x y - 2 y^2) /y = \dfrac{\alpha_0}{\alpha_1}, \\
 \dfrac{2 \alpha _1 (x + 1 ) + (\alpha_0 - \alpha_1) y}{2 \alpha_1 y} = z_0/z_1,
\end{array} \right.
\end{equation*}
where $[\alpha_0 : \alpha_1], [z_0 : z_1]  \in \P^1$.
 Assertion (ii) also follows from the fact that the point $[1: 0 :1]$
 is mapped by $\varphi_\alpha$ to $[1:0]$ and that the tangent line to
 the conic 
 $\psi^{-1}(\alpha)$ at the point $[-1:0:1]$ where $\alpha = [\alpha_0 : \alpha_1]$ is given by the equation:
 \begin{equation*}
 2 (\alpha_1 x+ \alpha_1 z) + (\alpha_0 - \alpha_1) y =0. 
 \end{equation*}
\end{proof}

Recall that the fiber $\psi^{-1}(\alpha)$ is mapped by $F$ to the fiber $\psi^{-1}(g(\alpha))$ where $g$ is the polynomial $x^2 - x -3$ so that  the following  diagram is commutative. 

\begin{equation*}
\xymatrix{ \P^1 \ar[r]^{\varphi_{g(\alpha)} \circ F \circ \varphi_\alpha^{-1} } & \P^1 \\
\psi^{-1}(\alpha) \ar[u]^{\varphi_\alpha} \ar[r]^{F} & \psi^{-1}(g(\alpha)) \ar[u]^{\varphi_{g(\alpha)}}}.
\end{equation*}

\begin{prop} For any $\alpha= [\alpha_0 : \alpha_1] \in \P^1$ and set $\beta = \alpha^2 - \alpha -3$. 
 Then the map $\varphi_\beta \circ F \circ \varphi_\alpha^{-1} : \P^1 \dashrightarrow \P^1 $ is given by the formula:
\begin{equation*}
\varphi_\beta \circ F \circ \varphi_\alpha^{-1}  \colon z \in \P^1 \mapsto \dfrac{(\alpha_0 - \alpha_1)(\alpha_0 + 2 \alpha_1)}{\alpha_0 +3 \alpha_1 } z.
\end{equation*}
\end{prop}

Replacing $\alpha= \psi(x,y)$ by its appropriate rational expression in $x,y$, we deduce a formula for $\varphi_{\psi(x,y)}$ and this yields a map  $\varpi \colon \C^2 \to \C^2$  given by:
\begin{equation*}
(x,y) \mapsto \left ( \psi(x,y) = \dfrac{x^2 - 1 - x y - 2 y^2}{y} , \varphi_{\psi(x,y)}= \dfrac{(1 + x - 2 y) (1 + x + y)}{2 y}\right ). 
\end{equation*}

\begin{proof}[Proof of Theorem \ref{thm_conj_hanoi}]
This follows directly from the fact that $g \circ \psi = \psi \circ F$ and the previous proposition.
\end{proof}

\begin{proof}[Proof of Proposition \ref{prop_dynamics_contracted_hanoi}]

Assertion (ii) follows from the computation of  the restriction to $C_2$ of the function $$\left ( P_0/P_1 + 1, \dfrac{P_2/P_1}{P_0/P_1 + 1} \right ).$$

Similarly, assertion (iii) follows from the computation of the restriction to $C_3$ of the function
$$\left ( P_0/P_1 -2 , \dfrac{P_2/P_1}{P_0/P_1 -2} \right ) . $$

Assertion (iv) follows from the computation of the restriction to $e=0$ of the following expression:
\begin{equation*}
(e, l) \mapsto [P_0(-1 + e , 1 , l e):  [P_1(-1 + e , 1 , l e) : [P_2(-1 + e , 1 , l e)].
\end{equation*}

Similarly, assertion (v) follows from the computation of the restriction to $e=0$ of the following expression:
\begin{equation*}
(e,l) \mapsto  [P_0(2+ e, 1 , le):P_1(2+ e, 1 , le): P_2(2+ e, 1 , le)].
\end{equation*}
For assertion (viii), to determine the image of the point $[1:1:0]$, we restrict to $e=0$ the expression
\begin{equation*}
(e,l) \mapsto [P_0(1 + e , 1 , le ): P_1(1 + e , 1 , le ) : P_2(1 + e , 1 , le )].
\end{equation*} 

We now prove successively in detail assertion (vi), (vii), (i).

\bigskip

Let us prove assertion (vi).

In the local coordinates $e= x/z+1 , l = y/(x+z)$, the map $F \circ \pi$ is given by the formula:
\begin{equation*}
F\circ \pi \colon (e, l ) \mapsto [P_0(-1 + e , l e , 1):P_1(-1 + e , l e , 1): P_2(-1 + e , l e , 1)],
\end{equation*}
which simplifies as follows:
\begin{multline*}
F \circ \pi \colon (e, l) \mapsto [-4 + 8 e - 5 e^2 + e^3 + 2 l - 5 e l + 4 e^2 l - e^3 l - 5 e l^2 + 
 8 e^2 l^2 - 3 e^3 l^2 - e^2 l^3 + e^3 l^3 + 
 2 e^3 l^4\\
 : e l^2 (-2 + e + e l): (2 - e + e l) (2 - e - l + e l^2)].
\end{multline*}
In particular, the exceptional divisor $E_3 $ given by the local equation $e=0$ is mapped to $p_3$ by $F \circ \pi$, indeed:
\begin{equation*}
F \circ \pi \colon (e=0, l) \mapsto [-4 + 2 l: 0: 2 (2 - l)]= [-1:0:1].
\end{equation*}
This shows that $\tilde{F}$ maps $E_3$ to either $E_3$ or a point on $E_3$. 
In the same local coordinates, $\tilde{F}$ is obtained by simplifying the expression:
\begin{equation*}
\tilde{F} \colon (e,l) \mapsto \left (e' :=  \dfrac{P_0(-1 + e , l e , 1) }{ P_2(-1 + e , l e , 1)} + 1 ,\  l':=\dfrac{P_1(-1 + e , l e , 1)}{P_0(-1 + e , l e , 1) + P_2(-1 + e , l e , 1)} \right ),
\end{equation*} 
and we get:
\begin{multline*}
\tilde{F} \colon (e, l ) \mapsto  ( e':=  \dfrac{e (4 - 4 e + e^2 - 2 l + 3 e l - e^2 l - 4 l^2 + 7 e l^2 - 
   3 e^2 l^2 + e^2 l^3 + 2 e^2 l^4)}{(2 - e + e l) (2 - e - l + 
   e l^2)},\\
    l':= \dfrac{l^2 (-2 + e + e l)}{4 - 4 e + e^2 - 2 l + 3 e l - 
 e^2 l - 4 l^2 + 7 e l^2 - 3 e^2 l^2 + e^2 l^3 + 2 e^2 l^4} \big )
\end{multline*}
One sees that the point $(e=0, l=2)$ is the only indeterminate point for $\tilde{F}$ on $E_3$.  
The restriction to $E_3$ yields:
\begin{equation*}
\tilde{F} \colon (e=0, l) \mapsto \left (e'= 0, \  l' = -\dfrac{2 l^2}{4 - 2 l - 4 l^2} \right ),
\end{equation*}
so the restriction of $\tilde{F}$ to $E_3$ is a self-map of topological degree $2$ with one indeterminacy point $(e=0, l=2)$.
Let us now compute the image of the indeterminate point $(e=0, l=2)$. 
Set $e_1= e, l = 2 + l_1 e_1 $, we express $F \circ \pi (e_1 , 2 + l_1 e_1)$ :
\begin{multline*}
(e_1 ,l_1) \mapsto [-22 + 27 e_1 + 27 e_1^2 + 2 l_1 - 25 e_1 l_1 + 24 e_1^2 l_1 + 63 e_1^3 l_1 - \\
 5 e_1^2 l_1^2 + 2 e_1^3 l_1^2 + 51 e_1^4 l_1^2 - e_1^4 l_1^3 + 
 17 e_1^5 l_1^3 + 
 2 e_1^6 l_1^4:\\
  (2 + e_1 l_1)^2 (-2 + 3 e_1 + e_1^2 l_1): (2 + e_1 + 
   e_1^2 l_1) (3 - l_1 + 4 e_1 l_1 + e_1^2 l_1^2)].
\end{multline*}
Evaluating at $e_1= 0$ yields:
\begin{equation*}
F \circ \pi (e_1 , 2 + l_1 e_1) \colon \  (e_1 = 0 , l_1) \mapsto [-22 + 2 l_1: -8: 2 (3 - l_1)].
\end{equation*}
In particular, the exceptional divisor $e_1 = 0$ corresponding to the indeterminacy point on $E_3$ is mapped to the line parametrized by:
\begin{equation*}
l_1 \mapsto [-22 + 2 l_1: -8: 2 (3 - l_1)].
\end{equation*}
We have thus proven assertion (vi).

\bigskip

Let us prove assertion (vii).
In the local coordinates $e = x/z- 1, l = y/(x-z)$, the map $\tilde{F}$ is of the form:
\begin{align*}
\tilde{F} \colon (e , l) \mapsto \left  (e' = \dfrac{ e (1 + l) (2 + e - 3 l - 2 e l - e l^2 + 2 e l^3)}{(-1 + l) (-2 - 
   e - l + e l^2)}, \right .\\
    \left . l'=\dfrac{l^2}{2 + e - 3 l - 2 e l - e l^2 + 2 e l^3} \right  ) .
\end{align*}
The above formula proves that $\tilde{F}$ has two indeterminacy points on $E_4$, the points $(e=0, l =1)$ and $(e=0, l=-2)$.    
  Moreover, the exceptional divisor $E_4$ is fixed and the induced map on $E_4$ is of the form:
  \begin{equation*}
 \tilde{F} \colon (e=0 , l) \mapsto \left (e'=0 , \dfrac{l^2}{2 - 3l} \right ),
\end{equation*}   
which is of topological degree $2$. 
We now compute the image of the two indeterminacy points $(e=0, l =1)$ and $(e=0, l=-2)$. 
Set $e= e' , l = 1 + e' l'$ so that $e'=0$ denotes the local equation for the exceptional divisor of the blow-up of the point $(e=0, l=1)$, we compute $F \circ\pi(e' ,1+ l' e')$ in these new coordinates:
\begin{multline*}
(e' , l') \mapsto  [-2 - 3 l' - 7 e' l' + 4 e'^2 l' - e' l'^2 - e'^2 l'^2 + 12 e'^3 l'^2 +
  e'^3 l'^3 + 9 e'^4 l'^3 + 
 2 e'^5 l'^4: \\
  (1 + e' l')^2 (2 + e' l'): l' (-3 - e' l' + 2 e'^2 l' + 
   e'^3 l'^2)] \in \P^2. 
\end{multline*}
In particular, the exceptional divisor $e'=0$ is mapped to the line parametrized by:
\begin{equation*}
(e' =0, l') \mapsto [-2 - 3 l': 2: -3 l'].
\end{equation*}
We now determine the image by $\tilde{F}$ of the other indeterminacy point $(e=0, l=-2)$.
Set $e= e' , l = -2 + e' l'$, we compute $F \circ \pi (e' , -2 + e' l')$:
\begin{multline*}
(e' , l') \mapsto [-17 + 15 e' + 3 l' + 26 e' l' - 41 e'^2 l' - e' l'^2 - 10 e'^2 l'^2 + 
 39 e'^3 l'^2 + e'^3 l'^3\\
  - 15 e'^4 l'^3 + 
 2 e'^5 l'^4:
  (-2 + e' l')^2 (-1 + e' l'): (-3 + e' l') (3 - l' - 
   4 e' l' + e'^2 l'^2)].
\end{multline*} 
In particular, the exceptional divisor $e'=0$ above $(e=0, l=-2)$ is mapped to the line parametrized by:
\begin{equation*}
(e'=0, l') \mapsto [-17 + 3 l': -4: -3 (3 - l')].
\end{equation*}
We have thus proved assertion (vii).
\bigskip

Let us prove assertion (i). 

Take $e=  x/z-1 , l= y/ (x-z)$ some local coordinates near the exceptional divisor $E_4$, then the map $\pi^{-1} \circ F \colon \P^2 \dashrightarrow X$ is of the form:
\begin{equation*}
\pi^{-1} \circ F\colon [x: y : 1] \in \P^2 \mapsto  (e(x,y), l (x,y)) \in X,
\end{equation*}
where $e(x,y), l(x,y)$ are two rational functions given by:
\begin{equation*}
e(x,y) =   \dfrac{(-1 + x + y) (1 - x - x^2 + x^3 + y + x y - 2 x^2 y + y^2 - x y^2 + 
   2 y^3)}{(-1 + x - y) (-1 + x^2 + y - 
   y^2)},
   \end{equation*}
   \begin{equation*}   
    l(x,y)= \dfrac{ y^2}{1 - x - x^2 + x^3 + y + x y - 2 x^2 y + y^2 - x y^2 +  2 y^3} .
   \end{equation*}
   Since the term $x+ y -1$ divides $e(x,y)$, we conclude that the curve ${C}_1$ is collapsed by $\pi^{-1} \circ F$ to a point on  $E_4$ or the whole curve $E_4$. 
We compute explicitly \\
$l(x, 1- x) = 1/5$, hence ${C}_1$ is collapsed by $\pi^{-1}\circ F$ to the point $(e= 0 , l =1/5)$ on $X$. 
We now determine the orbit of the point $(e=0, l =1/5)$ by $\tilde{F}$ on $E_4$. 
In the proof of assertion (viii), we have showed that the induced map on $E_4$ is of the form:
\begin{equation*}
(e=0 , l) \mapsto \left (e=0, \dfrac{l^2}{2 - 3 l} \right ). 
\end{equation*}
As $l=1/5$ is in the basin of attraction of the fixed point $(e=0, l=0)$, we deduce that the orbit of $(e=0, l=1/5)$ by $\tilde{F}$ avoids the indeterminacy $(e=0, l=1), (e=0, l=-2)$ and converges to the point $(e=0, l=0)$. 
This proves that the orbit of the line $\tilde{C}_1$ by $\tilde{F}$ is regular and assertion (i) holds.
\end{proof}

\subsection{The density of states for  the Hanoi group}

Recall from Section \ref{section_schur_hanoi} that the spectrum of the Hanoi group is associated to the following families of polynomial $P_n$ defined inductively by the relation:
\begin{equation*}
P_n(x,y) = \left (x^2 - (1+y)^2 \right )^{3^{n-2}} \left (x^2 - 1 + y -y^2\right )^{2 \cdot 3^{n-2}} P_{n-1}(F(x,y)),
\end{equation*}
where $P_1 $ is the polynomial:
\begin{equation*}
P_1 = - (x-1 -2 y) (x-1 +y)^2.
\end{equation*}
The density of states is in this case the limit, denoted $\omega$, of the sequence of measures $\omega_n:=1/3^n [P_n =0] \wedge [y=1]$.
Denote by $T_F$ the Green current associated to the Hanoi map $F$.

\begin{thm} \label{thm_spectrum_hanoi} 
The following properties are satisfied.
\begin{enumerate}
\item[(i)] The sequence of currents $1/3^n [P_n =0]$ converges as $n$ tends to $+\infty$ to a current which is supported on countably many curves. 
\item[(ii)] The sequence of measures $1/3^n [P_n = 0] \wedge [y=1]$ converges as $n$ tends to $+\infty$ to an atomic measure.
\item[(iii)] The sequence of measures
$$ \dfrac{3^n}{2^n} (\omega - \omega_n) $$
converges to the measure $T_F \wedge [y=1]$.
\end{enumerate}
In particular, the density of states associated with the Hanoi group is an atomic measure.
\end{thm}

\begin{proof}
Applying assertion (ii) of Theorem \ref{thm_dynamical_degree}, we
obtain 
$$
  \lambda_1(F) = \max(2,1)= 2
$$
since $F$ is conjugate to the map $(x,y) \mapsto \left (x^2 - x - 3, (x-1)(x+2)y/(x+3) \right )  $ by Theorem \ref{thm_conj_hanoi}. 
Since $\lambda_1(F) < 3$, we have by Theorem \ref{thm_discrete} that the limit of currents:
\begin{equation*}
\lim_{n\rightarrow +\infty} \dfrac{1}{3^n}[P_n=0] 
\end{equation*}
is supported on countably many curves and the intersection
\begin{equation*}
\lim_{n\rightarrow +\infty} \dfrac{1}{3^n}[P_n=0] \wedge [y=1] 
\end{equation*}
is an atomic  measure.
\bigskip

Let us prove assertion (iii). 
Take $X$ the blow-up of $\P^2$ at the four points $[-1:1:0], [2:1:0]$ and $[\pm 1 : 0: 1]$ and we denote by $\pi \colon X\to \P^2$ the associated blow-down map. We also denote by $E_1, E_2, E_3, E_4$ the exceptional divisors above $[-1:1:0], [2:1:0]$, $[- 1 : 0: 1]$ and $[ 1 : 0: 1]$ respectively and by $\tilde{L}_\infty$ the strict transform of the line at infinity by $\pi$. 

Let us first observe that the following equalities of cycles hold:
\begin{equation*}
\divi((x-1 - y) \circ \pi) = \pi^o [x-1-  y =0] + E_4  -\tilde{L}_\infty,
\end{equation*}
\begin{equation*}
\divi((x+1 + y) \circ \pi) = \pi^o [x+ 1+  y =0] +E_3 +E_1 - \tilde{L}_\infty,
\end{equation*}
\begin{equation*}
\divi((x-1 + y-y^2) \circ \pi) = \pi^o [x-  1+  y - y^2 =0]+ E_4  -2 \tilde{L}_\infty,
\end{equation*}
where $\pi^o$ denotes the strict transform by $\pi$. 
Since $P_1 = - (x-1 -2 y) (x-1 +y)^2$, we obtain the equality of cycles 
\begin{align*}
\divi (P_1 \circ \pi) = \pi^o[x-1 -2 y =0] + E_2 + E_4 + 2 (E_1 + \pi^o [x-1 + y =0] + E_4)  - 3 \tilde{L}_\infty\\
 =  2 E_1 + E_2 + 3E_4 +   \pi^o[x-1 -2 y =0] + 2 \pi^o [x-1 + y =0] - 3 \tilde{L}_\infty.
\end{align*}
To simplify the computations, we set:
\begin{align*}
L_1 := \pi^o [x-1-y =0],\ & L_2:= \pi^o [x+1 +y =0],& L_3:= \pi^o[x-1 -2 y =0] ,\\
L_4 :=  \pi^o [x-1 + y =0] ,\  &C:= \pi^o [x-  1+  y - y^2 =0].
\end{align*}
Using the fact that $\tilde{F}^* \tilde{L}_\infty =  \tilde{L}_\infty + L_1 + C$ and the inductive relation on the polynomials $P_n$, we get the equality of cycles:
\begin{align*}
\divi(P_2 \circ \pi) =& L_1 + L_2 + 2 C + E_1 +E_3 +  3 E_4 - 6\tilde{L}_\infty  \\ &+ \tilde{F}^* (2 E_1 + E_2 + 3 E_4 + L_3 + 2 L_4 - 3 \tilde{L}_\infty)\\
 =& -2  L_1 + L_2 - C + E_1 + E_3 + 2 E_4  - 9\tilde{L}_\infty \\
 &+ \tilde{F}^* (2 E_1 + E_2 + 3 E_4 + L_3 + 2 L_4).
\end{align*}
Using the above argument inductively, we deduce:
\begin{align*}
 \divi (P_n \circ \pi) = \sum_{k=0}^{n-1} 3^{n-2-k} (\tilde{F}^k)^* \left ( -2  L_1 + L_2 - C + E_1 + E_3 + 2 E_4  \right ) \\
  +   (\tilde{F}^{n-1})^* (2 E_1 + E_2 + 3 E_4 + L_3 + 2 L_4) - 3^{n} \tilde{L}_\infty.
\end{align*}
Taking the zero of $P_n$ and pushing forward by $\pi$, we obtain the equality of currents:
\begin{equation*}
\dfrac{1}{3^n} [P_n = 0] = \sum_{k=0 }^{n-1} \dfrac{1}{3^{k+2}} (F^k)^*\pi_* L_2 +  \dfrac{1}{3^n}(F^{n-1})^* (\pi_*L_3 + 2\pi_* L_4). 
\end{equation*}
Observe that the lines $L_2, L_3 , L_4$ are degenerates fibers of the associated fibers of the fibration map associated to $F$ (it is a fibration by conics passing through the four points $[-1 :1:0], [2:1:0]$ and $[\pm 1:0:1]$). By Proposition \ref{prop_skew_cantor}, the currents $3^n/2^n(F^n)^*\pi_* L_2$, $ 3^n/2^n(F^n)^* \pi_*L_3$ , $3^n/2^n(F^n)^* \pi_*L_4 $ converge to the current $T_F$ and their slices by $[y=1]$ converge to $T_F \wedge [y=1]$.
\end{proof}


\comm{******
\section{The spectrum of the Schreier graph associated to the Basilica group}

The map associated to the Basilica group, denoted $F$ is given by the formula:
\begin{equation*}
F (\la, \mu) := \left  ( - 2 - \dfrac{\la (2-\la)}{\mu^2} , \dfrac{\la -2 }{\mu^2} \right ).
\end{equation*}

In contrary to the previous map, this map is not algebraically semi-conjugate to a map on a curve. Indeed, if it were, its first dynamical degree would be an integer, but by Theorem \ref{thm_bedford}, the dynamical degree $\lambda_1(F) \simeq 1.7$ is not an integer.
Recall also from Section \ref{section_schur_basilica} that the spectral measure is determined by the recursion formula:
\begin{equation*}
P_{n+1}(\lambda,\mu) = \mu^{2^n} P_n(F(\la, \mu)),
\end{equation*}
where $P_1 = 2 \mu + 2 - \la $,
and the spectral measure is the limit as $n \rightarrow +\infty$ of the counting measure $\dfrac{1}{2^n}[P_n(\cdot, 1) = 0]$. 

\begin{thm} The spectral measure associated with the Basilica group is atomic.
\end{thm}

\begin{proof} By Theorem \ref{thm_bedford}, the dynamical degree of $F$ satisfies $\lambda_1(F)< 2$. 
Using Theorem \ref{thm_discrete}, we conclude that the limiting measure $1/2^n [P_n(\cdot , 1)]$ is atomic.
\end{proof}
**********************}


\begin{proof}[Proof of Theorem B]

Assertion (i) is the content of Theorem
\ref{thm_conjugate_grigorchuk}, assertion (ii) results from
Proposition \ref{prop_lamplighter_integrability},
and assertion (iii)  from Theorem \ref{thm_conj_hanoi}.
\end{proof}

\begin{proof}[Proof of Theorem A]
Assertion (i) results from Theorem \ref{thm_spectrum_grigorchuk}, assertion (ii) from Theorem \ref{thm_spectrum_lamplighter} and assertion (iii) from Theorem \ref{thm_spectrum_hanoi}.
\end{proof}


\section{Integrability in algebraic geometry}

\subsection{A general criterion}
In this section, we state some criterion to find invariant fibrations for rational surface maps using algebraic methods. Our goal is to present a practical criterion for integrability using the notions described in the appendix (see \S.\ref{section_appendix}).

We shall start with the following observation. 

\begin{prop} \label{prop_integrable_nef} Consider a dominant rational map $f \colon X \dashrightarrow X $ on a projective surface $X$. If $f$ is semi-conjugate to a degree $d$ map on a projective curve then
 there exists a birational modification $\tilde{X} $ of $X$ and a nef line bundle $L $ on $\tilde{X}$  satisfying the following conditions:
\begin{enumerate}
\item[(i)] One has $(c_1(L) \cdot c_1(L)) = 0$.
\item[(ii)] One has $f^* L = L^{\otimes d}$.
\item[(iii)] The line bundle $L$ has at least two independent sections. 
\end{enumerate} 
\end{prop}

\begin{proof} Let us suppose that $f$ preserves a fibration over a curve $C$.
Denote by $\varphi \colon X \dashrightarrow C$ the associated map onto $C$ and denote by $g\colon C \to C$ the map semi-conjugating $f$ via $\varphi$. 
Take $\tilde{X}$ the  graph of $\varphi$ in $X\times C$ and denote by $\pi_1, \pi_2$ the projection from $\tilde{X}$ to $X$ and $C$ respectively. 
We consider $\tilde{f}$ the restriction of the map $f\times g$ in $X\times C$ to $\tilde{X}$. 
Take $\Gamma$ the graph of $\tilde{f}$ on $\tilde{X}\times \tilde{X}$ and we denote by $u,v$ the projections of $\Gamma$ onto the first and second factor respectively.
We obtain the following commutative diagram:
\begin{equation}
\xymatrix{
  & & \Gamma \ar@/^-2.0pc/[lldd]_{u} \ar@/^2.0pc/[rrdd]^{v} & & \\
 &X \ar@{-->}[dd]^{\varphi} \ar@{-->}[rr]^{f} & &X  \ar@{-->}[dd]^{\varphi} \\
\tilde{X} \ar[ru]^{\pi_1} \ar[rd]^{\pi_2} &  & & &\tilde{X} \ar[lu]^{\pi_1} \ar[ld]^{\pi_2} \\
 & C  \ar@{->}[rr]^{g}& & C
}
\end{equation}
Take $p$ a point in $C$, we consider the divisor $F = \pi_2^* [p]$ in $\tilde{X}$ corresponding to the fiber of $\pi_2$ over $p$. 
Since any two general fibers of $\pi_2$ are disjoint, their intersection is zero. As a result, the self-intersection of $F$ satisfies the relation:
\begin{equation*}
(F \cdot F) =0. 
\end{equation*}
We now compute the pullback of $F$ by $\tilde{f}$. By definition, we have: 
\begin{equation*}
\tilde{f}^* F = u_* v^* F = u_*v^*\pi_2^* [p]. 
\end{equation*}
Using the fact that $\pi_2 \circ v =g \circ  \pi_2\circ u$, we get:
\begin{equation*}
\tilde{f}^* F = u_* v^* F = u_* u^* \pi_2^*g^* [p]. 
\end{equation*}
Since $g$ is of degree $d$, the preimage of the general point $[p]$ consists of $d$ points counted with multiplicities. As a result, we have $g^* [p] = d [p] $ in $H^{1,1}(C)$, hence the projection formula (Proposition \ref{prop_formula_push_pull} (iii)) gives:
\begin{equation*}
\tilde{f}^* F = d u_* u^* \pi_2^* [p] = d F,  
\end{equation*}
where we have used the fact that $u$ is birational, so its topological degree is one. This proves that the class $F$ is multiplied by $d$ by the action of $\tilde{f}^*$.

Since $(F \cdot F)=0$ and $\tilde{f}^* F = d F$, the two relations  also hold for $nF$ where $n$ is an integer. 
The line bundle $L = \mathcal{O}(nF)$ associated to the divisor $nF$ satisfies (i) and (ii). 
Since the divisor $[p]$ is ample on $C$, there exists a multiple $n$ such that $n[p]$ is very ample so the line bundle $\mathcal{O}_C(n[p])$ has at least two sections by Asymptotic Riemann-Roch's theorem (\cite[Example 1.2.19]{lazarsfeld_positivity_1}). Since $L =\mathcal{O}(nF) $ is the pullback of this line bundle by $\pi_2$, we deduce that property (iii) holds for a large enough multiple $n$, as required.
\end{proof}

For the converse statement, we shall need to construct fibrations using a result due to Iitaka (see \cite[Theorem 2.1.33]{lazarsfeld_positivity_1}).

\begin{thm} \label{thm_iitaka}Let $L$ be a nef line bundle on a surface $X$ and $k>0$ be an integer such that $$ 1/C \leqslant \dfrac{h^0(nL)}{n^k} \leqslant C, $$
for $C>0$. 
Then for sufficiently large $n$, the mapping induced by $L$,  $\phi : X \dashrightarrow \P(H^0(X,nL)^*)$ defines a map onto a variety of dimension $k$. 
\end{thm}

The map $\phi$ in the above theorem is called the Iitaka map associated to $L$.
In our situation, we will apply the result to a nef line bundle which will satisfy the equation $(c_1(L)^2)=0$. 
 When this happens, the Iitaka map $\phi$ associated to $L$ cannot map to a surface, but will either map to a curve or a to point. 

\medskip

\begin{thm} \label{thm_line_bundle_fibration} Let $f \colon \P^2 \dashrightarrow \P^2$ be a dominant rational map. 
Suppose that there exists a smooth surface $X$ obtained from $\P^2$ by finitely many blow-ups, a line bundle $L$ on $X$ and an integer $d \geqslant 1$ satisfying the following conditions:
\begin{enumerate}
\item[(i)] One has $(c_1(L) \cdot c_1(L))=0$.
\item[(ii)] The line bundle $L$ is nef.
\item[(iii)] One has $(K_X \cdot c_1(L)) < 0$ where $K_X$ is the (first) Chern class of the canonical bundle of $X$.
\item[(iv)] One has $f^* L = L^{\otimes d}$.
\end{enumerate} 
Then $f$ is semi-conjugate to a degree $d$ map on a curve.
\end{thm}

\begin{proof} 
We shall prove successively the following statements.

\textbf{Claim 1}: The sequence $h^0(X,nL)/n^2$ converges to zero as $n $ tends to infinity.
\smallskip

By Asymptotic-Riemann-Roch theorem (see \cite[Example 1.2.36 (ii)]{lazarsfeld_positivity_1}):
\begin{equation*}
h^0 (nL) = \dfrac{1}{2} n^2 (c_1(L)^2) + O(n).
\end{equation*}
Since $(c_1(L)^2)=0$, we obtain that $h^0(nL) = O(n)$.
As a result, this proves that $h^0(nL)/ n^2$ converges to zero and the claim is proved.
\medskip

\textbf{Claim 2}: One has $h^0(X,nL) \geqslant \chi(nL)$ for all integer $n$.  

By Serre duality (\cite[Corollary 7.7]{hartshorne}), we have $h^2(X, L)= h^0(X, L^{-1} \otimes \mathcal{O}(K_X))=0$ since $K_X$ is not effective.
We obtain that:
\begin{equation*}
 \chi(nL) = h^0(X, nL) - h^1(X, nL) + h^2(X,nL) = h^0(X, nL) - h^1(X,nL) \leqslant h^0(X,nL),
 \end{equation*} 
and the claim is proved.
\medskip

%
%
%

\textbf{Claim 3}: $\chi(nL) \geqslant \alpha n$ for some $\alpha > 0$.

By Riemann-Roch theorem (\cite[I Theorem 5.5 (6)]{barth_peters_van}) and using the fact that $(c_1(L) \cdot c_1(L))=0$, we have:
\begin{equation*}
\chi(nL) = \chi(\mathcal{O}_X) + \dfrac{1}{2} (n^2(c_1(L) \cdot c_1(L)) - \dfrac{n}{2}(c_1(L) \cdot K_X)) =  \chi(\mathcal{O}_X) - \dfrac{n}{2} (c_1(L) \cdot K_X).
\end{equation*}
This proves the result as $(c_1(L) \cdot K_X) < 0$.
\medskip

We now show that the Proposition holds.
We consider for $n$ large the map:
$$\phi_{nL} \colon X \dashrightarrow \P(H^0(X,nL)^*).$$
Claim 1 proves that the dimension of the image of $\phi_{nL}$ is not $2$. Moreover, Claim 2 and 3 imply that $h^0(X, nL) \geqslant \alpha n$, thus $L$ satisfies the condition of Theorem \ref{thm_iitaka} for $k=1$ and $\phi_{nL}$ maps $X$ to a curve $C$. Finally, the fact that $f^* L = L^{\otimes d}$ proves that $\phi_{nL}$ semi-conjugates $f$ to a degree $d$ map on $C$, as required.
\end{proof}

\subsection{The integrability criterion applied to the renormalization of the Grigorchuk group}
\label{subsection_AG_grigorchuk}

Recall from Section \ref{subsection_integrability_grigorchuk} that we have defined $X$ to be  the blow-up of $\P^2$ at the four points $[0: \pm 2 : 1]$ and $[\pm 1 : 1 : 0]$. 
We denote by $E_1, E_2, E_3 , E_4$ the exceptional divisors above the points $[-1:1:0] ,[1:1:0], [0:-2:1] , [0:2:1] $ respectively and take $\tilde{L}_\infty$ to be the strict transform of the line at infinity in $\P^2$ by the blow-up (see figure \ref{figure_grigorchuk}).

We fix a basis of $H^{1,1}(X)$, namely :
\begin{equation}
H^{1,1}(X) = \C \tilde{L}_\infty \oplus \C E_1 \oplus \C E_2 \oplus \C E_3 \oplus \C E_4.
\end{equation}

Recall that we have proved in assertion (v) of Proposition \ref{prop_grigor_lift_dynamics} that the map $\tilde{F}$ is algebraically stable on $X$. 

\begin{cor} The pushforward and pullback matrices associated to $\tilde{F}$ and $\tilde{G}$ are of the form:
\begin{equation*}
\tilde{F}_* = \left ( \begin{array}{lllll}
1  & 1 &1 & 1  &1 \\
0 & 1 & 1 & 0& 1\\
0 & 1 & 1 & 1 & 0\\
0 & 0 & 0 & 1 &0 \\
0 & 0 & 0 & 0 & 1
\end{array} \right ), \quad \tilde{F}^* = \left ( \begin{array}{lllll}
 1 & 1 & 1 & 0& 0 \\
 0 & 1 & 1 & 0 & 0 \\
 0 & 1 & 1 & 0 & 0 \\
 0 & -1& 0& 1 & 0 \\
 0& 0 & -1 & 0 & 1
\end{array}   \right ),
\end{equation*}
\begin{equation*}
\tilde{G}_* = \left ( \begin{array}{lllll}
1  & 1 &1 & 1  &1 \\
1 & 1 & 1 & 1& 2\\
1 & 1 & 1 & 2 & 1\\
-1 & 0 & 0 & 0 &-1 \\
-1 & 0 & 0 & -1 & 0
\end{array} \right ), \quad  \tilde{G}^* = \left ( \begin{array}{lllll}
3&0&0&1&1 \\
2& 0 & 0 & 1 & 1 \\
2& 0 & 0 & 1 & 1 \\
-2 & 0 & 1 & 0 & -1 \\
-2 & 1 & 0 & -1 & 0
\end{array} \right ).
\end{equation*}
\end{cor}

\begin{proof} We first compute $\tilde{F}_* \tilde{L}_\infty$, since $\tilde{L}_\infty$ is collapsed to a point and the indeterminacy point $[1,0,0]$ is mapped to $L$ by assertion (v) of Proposition \ref{prop_grigor_lift_dynamics}, we deduce that:
\begin{equation}
\tilde{F}_* \tilde{L}_\infty = \tilde{L}_\infty.
\end{equation}
Moreover, since $E_1$ and $E_2$ are both mapped regularly to the line $\{\lambda= -2w\}$ by assertion (iv) and since assertion (v) of Proposition \ref{prop_formula_push_pull} yields the equality
$\pi^o\{ \lambda= -2w \} = \tilde{L}_\infty + E_1 + E_2$ in $H^{1,1}(X)$, we obtain that:
\begin{equation*}
\tilde{F}_* E_1 = \tilde{F}_* E_2 = \tilde{L}_\infty + E_1 + E_2 \in H^{1,1}(X).
\end{equation*}
We now compute the image of $E_3$, and the computation for $E_4$ is similar. 
Using assertion (v) of Proposition \ref{prop_formula_push_pull}, we obtain   the equality $ \pi^o\{ \lambda + \mu + 2 w \} = L +E_2 -E_3 $ in $H^{1,1}(X)$, so assertion (iii) and (iv) of Proposition \ref{prop_grigor_lift_dynamics} implies that:
\begin{equation*}
 \tilde{F}_* E_3 = \tilde{L}_\infty + E_2 - E_3 + 2 E_3= \tilde{L}_\infty + E_2 + E_3 \in H^{1,1}(X),
 \end{equation*} 
 and we thus obtain the pushforward action by $\tilde{F}$. 
 The pullback action can be deduced from the matrix of $\tilde{F}_*$ by conjugating the transpose of $\tilde{F}_*$ by the intersection matrix.
 We thus obtain:
 \begin{equation*}
 \tilde{F}^* = I^{-1} \tilde{F}_*^t I =   \left ( \begin{array}{lllll}
 1 & 1 & 1 & 0& 0 \\
 0 & 1 & 1 & 0 & 0 \\
 0 & 1 & 1 & 0 & 0 \\
 0 & -1& 0& 1 & 0 \\
 0& 0 & -1 & 0 & 1
\end{array}   \right ) , 
 \end{equation*}
 where $I$ is the intersection matrix given by:
 \begin{equation*}
I := \left ( \begin{array}{lllll}
 -1 & 1 & 1 & 0& 0 \\
 1 & -1 & 0 & 0 & 0 \\
 1 & 0 &-1  & 0 & 0 \\
 0 & 0& 0& -1 & 0 \\
 0& 0 & 0 & 0 & -1
\end{array}   \right ).
 \end{equation*}
We finally deduce $G$ from the fact that $G = H \circ F$ and from the fact that $H$ is an automorphism on $X$ whose matrix is given by:
\begin{equation*}
H_* =\left  ( \begin{array}{lllll}
1  &0 & 0 & 0& 0\\
1 & 0 & 0 & 0& 1\\
1 & 0 & 0 & 1&0\\
-1 & 0 & 1 & 0&0\\
-1 & 1 & 0 & 0 & 0
\end{array} \right ).
\end{equation*}
We finally compute the matrices of $\tilde{G}_*$ and $\tilde{G}^*$.
\end{proof}
 
From the previous result, we see that the line bundle induced by the divisor $D = 2\tilde{L}_\infty + E_1 + E_2 - E_3 - E_4$ is invariant by both $\tilde{F}^*$ and $\tilde{G}^*$ since $D$ is an eigenvector for both matrices associated to the eigenvalue $2$. 
Recall that by assertion (iv) of Proposition \ref{prop_formula_push_pull}, the canonical class of $X$ is given by $-3 \tilde{L}_\infty -2 E_1 - 2 E_2 + E_3 + E_4$.

Observe that the class $D$ satisfies the following properties:

\begin{enumerate}
\item $D^2 =0$.
\item $(K_X \cdot D) < 0$.
\item $D$ is nef. 
\item $\tilde{F}^*D = \tilde{G}^*D = 2 D$
\end{enumerate}
Consider the Iitaka map induced by the linear system $|D|$ denoted
$\phi_D \colon X \dashrightarrow C,$
where $C$ is an algebraic Riemann surface. 
The fibers of $\phi_D$ are conics which pass through all four points $[0:\pm 2: 1]$ and $[\pm 1: 1: 0]$. 
Using Theorem \ref{thm_line_bundle_fibration}, we deduce the following result.

\begin{prop}
 Both maps $F$ and $G$ preserve the fibration induced by $\phi_D$ and $\lambda_1(F) = \lambda_2(G)=2$ and $F$ is semi-conjugate to $\alpha \id\times T $ where $\alpha \neq 0 $ and where $T$ is the Chebyshev map. 
\end{prop}

\begin{proof}
Since $D$ is stable by pullback by $\tilde{F}$ and $\tilde{G}$, the map $\phi_D$ defines a semi-conjugation. Let us prove that $F$ is semi-conjugate to $\alpha \id \times T$. 
Recall that the divisor $E_3$ is mapped by $\tilde{F}$ to itself and that the restriction of $\tilde{F}$ to $E_3$ (see \eqref{eq_restriction}) is of the form:
\begin{equation*}
(e=0 , l) \in E_3 \mapsto (e=0, 2l^2 -1).
\end{equation*}
Since the pencil of conic induced by $D$ all pass transversely to $E_3$, we deduce that the induced map on $C$ is the Chebyshev map on $\P^1$.
Recall that the product formula yields that the relative dynamical degree of $F$ and $G$ is one.
Moreover, the fibers of the fibration are rational curves (they are conic curves on $\P^2$), and since $F$ fixes the divisors $E_3,E_4$ which are transversal to the pencil of conics, we can identify two points on each fiber of $\phi_D$, for example identify $E_3 \cap \phi_D^{-1}(c_0)$ with $0\in \P^1$ and $\phi_D^{-1}(c_0) \cap E_4$ with the point at infinity $\infty \in \P^1$. Since $F$ fixes $E_3$ and $E_4$, we deduce that under this identification, $F$ acts on the fiber as a degree one map fixing two points $0$ and $\infty$, it is thus semi-conjugate to $\alpha \id \times T$.
\end{proof}

We now recover in an explicit way the fibration $\phi_D$ map up to conjugation geometrically. 
For any point $m  $ on the exceptional divisor $E_3$, there exists a unique conic in the linear system passing through $m$, because our fibers should already pass through the four points $[0,\pm 2, 1]$ and $[\pm 1, 1, 0]$. 
For any point $[\lambda_0,\mu_0 , w_0]$, the unique conic in the linear system passing through $[\lambda_0,\mu_0 , w_0]$ has equation:
\begin{equation}
 \lambda_0 w_0 \lambda^2 - \lambda_0 w_0 \mu^2 + (-\lambda_0^2 + \mu_0^2 - 4 w_0^2) \lambda w + 4 \lambda_0 w_0 w^2 =0.
\end{equation}
The above conic passes through the point $[0,-2,1]$ with a slope:
$$\lambda/(\mu + 2) =  -\dfrac{\lambda_0^2 - \mu_0^2 + 4 w_0^2}{4 \lambda_0 w_0}. $$

The map $\phi_D$ is then realized geometrically by the map $[\lambda, \mu, w] \in \P^2 \mapsto \psi(\lambda,\mu, w) \in E_3 $ where $\psi $ is given by:
\begin{equation}
\psi(\lambda, \mu, w) = -\dfrac{\lambda^2 - \mu^2 + 4 w^2}{4 \lambda w}.
\end{equation}

\subsection{The integrability criterion applied to the renormalization of the lamplighter group}

Recall from  Section \ref{section_integrability_lamplighter} that we have defined $X$ to be the blow-up of $\P^2$ at the two points $[\pm 1 : 1 :0]$. We denote by $\tilde{F}$ the lift of $F$ to $X$ and by $E_1,E_2$ the exceptional divisors above $[-1 :1 :0]$ and $[1:1:0] $ respectively (see figure \ref{figure_lamplighter}).

We identify $H^{1,1}(X)$ with $\C \tilde{L}_\infty \oplus \C E_1 \oplus \C E_2$ where $\tilde{L}_\infty$ denotes the strict transform of the line at infinity. 
Recall that we have proven in assertion (ix) of Proposition \ref{prop_lamp} that $\tilde{F}$ is algebraically stable in $X$. 
We thus compute the pullback and pushforward action of $\tilde{F}$ on $H^{1,1}(X)$ 

\begin{prop} The matrices of $\tilde{F}^*$ and $\tilde{F}_*$ are of the form:
\begin{equation*} \tilde{F}_* :=\left ( 
\begin{array}{lll}
0 &1&1\\
0&1&0\\
0&0&1
\end{array} \right ), \quad \tilde{F}^* := \left ( 
\begin{array}{lll}
1 &1&0\\
0&1&0\\
1&0&0
\end{array} \right ).
\end{equation*}
\end{prop}

\begin{proof}
We first compute the matrix of $\tilde{F}_*$. 

Let us first observe that $\tilde{F}$ is regular on the line $\pi^{-1}(\{\mu=0\})$ and that this line is mapped to the curve $C_1$ parametrized as follows:
\begin{equation*}
 (\lambda, w) \mapsto [-\lambda^2 + 2 w^2: -2 w^2 : -w \lambda].
 \end{equation*} 
Since $[\pi^{-1}(\mu=0)] = \tilde{L}_\infty + E_1 + E_2$ and since the curve $C_1$ is a degree $2$ curve passing through the point $[-1:1:0]$, assertion (v) of Proposition \ref{prop_formula_push_pull} shows that $\pi^o C_1 = 2 \tilde{L}_\infty + E_1 + 2 E_2$ in $H^{1,1}(X)$, and we get:
\begin{equation} \label{eq_lamp_1}
\tilde{F}_* (\tilde{L}_\infty + E_1 + E_2) = 2 \tilde{L}_\infty + E_1 + 2 E_2.
\end{equation}

Next, we consider the line $\pi^{o}(\{\lambda= -\mu\})$, since $\tilde{F}$ is regular on $E_1$ by Proposition \ref{prop_lamp} (vi), the image $C$ of this line is  parametrized by:
\begin{equation*}
\lambda \in \C^* \mapsto \pi^{-1}([1:-1: -\lambda]),
\end{equation*}
Since assertion (v) of Proposition \ref{prop_formula_push_pull} gives $\pi^{o}(\{\lambda= -\mu\}) =\tilde{L}_\infty + E_2$ and $\pi^oC= \tilde{L}_\infty+ E_2$ in $H^{1,1}(X)$, we deduce:
\begin{equation} \label{eq_lamp_2}
\tilde{F}_* (\tilde{L}_\infty+ E_2) = \tilde{L}_\infty+ E_2. 
\end{equation} 

Finally, we consider the conic curve $Q:=\pi^{o} \{ \lambda^2 - 2 \lambda\mu + \mu^2 - w^2 =0 \}$ which does not intersect $E_2$ at the indeterminacy point of $\tilde{F}$. 
The image of $Q$ by $\tilde{F} $ is the conic curve $C_2$ parametrized by:
\begin{equation}
(\lambda, w) \mapsto [2 - 2 \lambda + w^2 :-2 w^2 : -(\lambda - 1) w]. 
\end{equation}
Assertion (v) of Proposition \ref{prop_formula_push_pull} shows that $\pi^o C_2 = 2 \tilde{L}_\infty + 2E_1 + 2 E_2$ and  $Q= 2 \tilde{L}_\infty + 2 E_1$ in $H^{1,1}(X)$, we obtain:
\begin{equation} \label{eq_lamp_3}
\tilde{F}_* (2\tilde{L}_\infty + 2E_1) = 2\tilde{L}_\infty + 2 E_1 + 2 E_2.
\end{equation}
Using \eqref{eq_lamp_1}, \eqref{eq_lamp_2} and \eqref{eq_lamp_3}, we deduce the matrix of $\tilde{F}_*$.
 We then deduce the pullback matrix $\tilde{F}^*$ using the fact that 
 $\tilde{F}^* = I^{-1} \cdot \tilde{F}_*^t \cdot I$ where $I$ is the intersection matrix given by:
 \begin{equation*}
 I := \left ( \begin{array}{lll}
 -1 & 1 & 1 \\
 1 & -1 & 0 \\
 1 & 0 & -1 
\end{array}  \right ).
\end{equation*}  
\end{proof}
Recall that by assertion (iv) of Proposition \ref{prop_formula_push_pull}, the canonical class of $X$ is the divisor $K_X$ satisfying
\begin{equation*}
K_X = -3 \tilde{L}_\infty - 2 E_1 -2E_2 \in H^{1,1}(X). 
\end{equation*} 
From the above proposition, we see that the matrix of $\tilde{F}^*$ admits a Jordan block for the eigenvalue $1$.
Denote by $D_1 = \tilde{L}_\infty + E_2$, $D_2 = \tilde{L}_\infty + E_1$ we have:

\begin{enumerate}
\item $\tilde{F}^* D_1 = D_1$, $\tilde{F}^* D_2 = D_1 + D_2 $. 
\item Both $D_1$ and $D_2$ are nef.
\item One has $(D_1 \cdot D_1) = (D_2 \cdot D_2) = 0$.
\item One has $(K_X \cdot D_1) = (K_X \cdot D_2)= -2 < 0$.
\end{enumerate} 

Geometrically, the two divisors $D_1$ and $D_2$ define two pencils, $D_1$ is the pencil of lines in $\P^2$ passing through $[-1,1,0]$ and $D_2$ the pencil of lines passing through the point $[1,1,0]$.
Using Theorem \ref{thm_line_bundle_fibration}, we deduce that $F$ preserves the fibration induced by $D_1$.

\begin{cor} $F$ is semi-conjugate to a linear map on $\P^1$ and the action of $F$ on the fiber is also linear. 
\end{cor}

\subsection{The integrability criterion applied to the renormalization of the Hanoi group}

Recall from Section \ref{subsection_integrability_hanoi} that we have defined $X$ to be the blow-up of $\P^2$ at the four points $[-1:1:0], [2:1:0], [- 1 : 0 : 1], [1:0:1]$ and that we have set $E_1, E_2 , E_3$ and $E_4$ to be the four exceptional divisors on $X$ (see figure \ref{figure_hanoi}).

We identify $H^{1,1}(X)$ with 
\begin{equation*}
H^{1,1}(X) \simeq \C \tilde{L}_\infty \oplus \C E_1 \oplus \C E_2 \oplus \C E_3 \oplus \C E_4,
\end{equation*}
where  $\tilde{L}_\infty$ is the strict transform of the line at infinity in $\P^2$.
We denote by $\tilde{F}$ the lift of $F$ to $X$.

\begin{prop} The map $\tilde{F}$ is algebraically stable on $X$.
\end{prop}

\begin{proof}
Using  Proposition \ref{prop_dynamics_contracted_hanoi}, we control the orbit of every contracted curve and every exceptional divisor on $X$, none are contracted to an indeterminacy point of $\tilde{F}$, so $\tilde{F}$ is algebraically stable.
\end{proof}

\begin{cor} The pushforward and pullback action  of $\tilde{F}$ are given by the matrices:
\begin{equation*}
\tilde{F}_* = \left (\begin{array}{lllll}
1& 2 & 1 &1 & 2 \\
0 & 2 & 1 & 1 & 1 \\
0 & 1 & 1 & 0 & 1 \\
0& 0 &0 & 1 & 0 \\
0 & -1 & 0 & 0 & 0 
\end{array} \right ) ,
\end{equation*}

\begin{equation*}
\tilde{F}^* = \left ( \begin{array}{lllll}
1& 1 & 2 & 0 & 1 \\
0& 1 & 1 & 0 & 0 \\
0 & 1 & 2 & 0 & 1 \\
0& 0 & -1 & 1 & 0 \\
0 & -1 & -1 & 0 & 0 
\end{array}  \right ) .
\end{equation*}
\end{cor}

\begin{proof}
We look at the image of $L, E_1, E_2,E_3, E_4$ by $\tilde{F}$. 

Since $\tilde{F}$ contracts $\tilde{L}_\infty$ to a point and maps the indeterminacy point on $\tilde{L}_\infty$ to $\tilde{L}_\infty$ by Proposition \ref{prop_dynamics_contracted_hanoi}.(viii), we have:
\begin{equation*}
\tilde{F}_* \tilde{L}_\infty = \tilde{L}_\infty.
\end{equation*}

By Proposition \ref{prop_dynamics_contracted_hanoi}.(viii), the divisor $E_1$ is mapped regularly by $F\circ \pi$ to  the curve parametrized by:
\begin{equation*}
l \mapsto [6 - 3 l - l^2: -(-1 + l) l: 2 (2 - l) l],
\end{equation*}
and since this curve is of degree $2$ and passes through the point $[1: 0:0] , [2:1:0], [1:0:1]$, we deduce using assertion (v) of Proposition \ref{prop_formula_push_pull} that:
\begin{equation*}
\tilde{F}_* E_1 = 2 \tilde{L}_\infty + 2 E_1 + 2 E_2 -  E_2 - E_4 = 2 \tilde{L}_\infty + 2 E_1 + E_2 - E_4.
\end{equation*}
By Proposition \ref{prop_dynamics_contracted_hanoi}.(v), the exceptional divisor $E_2$ is mapped regularly by $F \circ \pi$ to the  line $\{ z = y\} \subset \P^2$ and since this line does not pass through any indeterminacy point of $F$, assertion (v) of Proposition \ref{prop_formula_push_pull} gives:
\begin{equation*}
\tilde{F}_* E_2 = \tilde{L}_\infty + E_1 + E_2.
\end{equation*}

By Proposition \ref{prop_dynamics_contracted_hanoi}.(vi), the exceptional  divisor $E_3$ is  fixed by $\tilde{F}$, the restriction of $\tilde{F}$ to $E_3$ has topological degree $2$ and $\tilde{F}$ has one indeterminacy point on $E_3$. The image of the indeterminacy point on $E_3$ by $\tilde{F}$ is the line parametrized by:
\begin{equation*}
l \mapsto [-22 + 2 l : -8 : 2(3- l)].
\end{equation*}
Since this line passes through the points $[2:1:0], [-1:0:1]$, we deduce using assertion (v) of Proposition \ref{prop_formula_push_pull} that:
\begin{equation*}
\tilde{F}_* E_3 = 2E_3 + (\tilde{L}_\infty + E_1 + E_2 - E_2 - E_3) = \tilde{L}_\infty + E_1 +E_3.
\end{equation*}

By Proposition \ref{prop_dynamics_contracted_hanoi}.(vii), the exceptional divisor $E_4$ is fixed by $\tilde{F}$ with multiplicity $2$ and $\tilde{F}$ has two indeterminacy points on $E_4$. The image of those indeterminacy points are two respective lines, parametrized by:
\begin{equation*}
l \mapsto [-2 - 3 l : 2 : -3 l],
\end{equation*}
\begin{equation*}
l \mapsto [-17 + 3 l : -4 : -3 (3-l)].
\end{equation*}
These line pass through the points $[-1:1:0], [1:0:1]$ and through the points $[2:1:0], [1:0:1]$ respectively. 
Thus assertion (v) of Proposition \ref{prop_formula_push_pull} gives:
\begin{equation*}
\tilde{F}_* E_4 = 2 E_4 + (\tilde{L}_\infty+E_1 + E_2 - E_1 - E_4) + (\tilde{L}_\infty+E_1 + E_2 - E_2 - E_4) = 2 \tilde{L}_\infty + E_1 + E_2 .
\end{equation*}
Finally, using the expressions image by $\tilde{F}$ of $\tilde{L}_\infty, E_1, E_2, E_3, E_4$, we deduce that the matrix of $\tilde{F} $ is given by:
\begin{equation*}
\tilde{F}_* = \left (\begin{array}{lllll}
1& 2 & 1 &1 & 2 \\
0 & 2 & 1 & 1 & 1 \\
0 & 1 & 1 & 0 & 1 \\
0& 0 &0 & 1 & 0 \\
0 & -1 & 0 & 0 & 0 
\end{array} \right ) .
\end{equation*}
We deduce the matrix of the pullback action  $ \tilde{F}^*$ by conjugating the transpose of the matrix of $\tilde{F}_*$ by the intersection matrix. We get:
\begin{equation*}
\tilde{F}^* = I^{-1} \tilde{F}_*^t I =  \left ( \begin{array}{lllll}
1& 1 & 2 & 0 & 1 \\
0& 1 & 1 & 0 & 0 \\
0 & 1 & 2 & 0 & 1 \\
0& 0 & -1 & 1 & 0 \\
0 & -1 & -1 & 0 & 0 
\end{array}  \right ),
\end{equation*}
where $I$ is the matrix given by:
\begin{equation*}
I = \left ( \begin{array}{lllll}
-1 & 1 & 1 & 0 & 0 \\
1 & -1 & 0 & 0 & 0 \\
1 & 0 & -1 & 0 & 0 \\
0 & 0 & 0 & -1 & 0 \\
0 & 0 & 0 & 0 & -1
\end{array} \right ).
\end{equation*}
\end{proof}

Consider $D = 2 \tilde{L}_\infty + E_1 + E_2 - E_3 - E_4$, the divisor $D$ satisfies the following properties:
\begin{enumerate}
\item[(i)] One has $\tilde{F}^* D = 2 D$. 
\item[(ii)] One has $(D \cdot D) = 0$ and the divisor $D$ is nef.
\item[(iii)] One has $(K_X \cdot D) < 0$.
\end{enumerate}
Geometrically, any element of $D$ corresponds to a conic curve passing through the four points $[-1:1:0],[2:1:0], [\pm 1 : 0:1]$.
Using Theorem \ref{thm_line_bundle_fibration}, we deduce that $F$ is semi-conjugate to a one dimensional map.

\begin{cor} $F$ is semi-conjugate to a degree $2$ rational map on a curve.
\end{cor}


\section{Appendix: Elements of Complex Geometry}
\label{section_appendix}

In this appendix, $X$ will be either an irreducible smooth projective surface or an irreducible smooth projective curve, i.e a smooth complex surface or curve which is embedded inside a complex projective space $\P^n$ for an integer $n \geqslant 2$. 
In this particular setting, one can view $X$ as a K\"ahler manifold, endowed with the topology induced by the balls of $\C^{n+1}$, or $X$ can be described as the intersection of the vanishing locus of finitely many homogeneous polynomials in $\C^{n+1}$. In the latter viewpoint, $X$ is endowed with the Zariski topology. 
We refer to \cite[p.23]{griffiths_harris} for the definition of the Dolbeaut and deRham cohomology on $X$.
\subsection{Currents and their intersection using Bedford-Taylor's method}
\label{appendix_current}

In this section, $X$ is a smooth projective surface. 

A \textbf{current of bidegree $(1,1)$} is an element of the dual of  smooth $(1,1)$ forms on $X$, a \textbf{current of bidegree $(2,2)$} is a signed measure on $X$. There are currents of other bidegree but we will not use them in this paper.
If $T$ is a $(1,1)$ current and $\alpha$ is a smooth $(1,1)$ form on $X$, we shall denote by $\langle T, \alpha \rangle \in \R$ the natural between $T$ and $\alpha$. 
The deRham differential on smooth forms induces by duality a differential on currents which we also denote by $d$. 
We say a $(1,1)$ current $T$ is closed if $\langle T , d v \rangle$ for any smooth $1$-form $v$ on $X$.

A $(1,1)$ current $T$ is said to be \textbf{positive} if for any positive function $u$ and for any $(1,0)$ smooth form $\alpha$ on $X$, $ T \wedge u \alpha \wedge i \bar\alpha$ is a positive measure on $X$. 
Like measures, the \textbf{support} of a $(1,1)$ current $T$, denoted $\supp(T)$ is the smallest closed subset such that  the current is identically zero on $X\setminus \supp(T)$. 
 
\medskip

The main example of currents are the one induced by smooth forms and by plurisubharmonic functions. 

Let us describe the first type of currents.  Take $\alpha $ a smooth $(1,1) $ form on $X$, and define the functional:
\begin{equation*}
T_\alpha \colon\beta \mapsto \int_X \alpha \wedge \beta,
\end{equation*}
where $\beta $ is a smooth $(1,1)$ form on $X$. 
Since $T_\alpha$ is linear and continuous, the functional $T_\alpha$ defines a $(1,1)$ current.

On the other side of the spectrum, if $V$ is an (complex) analytic curve on $X$, defined the functional $[V]$  by the formula:
\begin{equation*}
[V]\colon \beta \mapsto \int_{V^{\text{sm}}} \beta,
\end{equation*}
where $\beta$ is a smooth $(1,1)$ form on $X$ and where $V^{\text{sm}}$ denotes the smooth locus of $V$.
The functional $[V]$ defines a closed positive $(1,1)$ current on $X$, and is refered as the\textbf{ current of integration} on $V$.

The currents induced by plurisubharmonic are in between these two type of $(1,1)$ currents. 
A function $\phi \colon \Omega \subset \C^2 \to \R \cup \{ - \infty \}$ on a domain $\Omega$ of $\C^2$ is \textbf{plurisubharmonic} if it is upper-semicontinuous and the restriction of $\phi$ on every complex line is subharmonic. We state the main properties satisfied by these functions. 

\begin{thm} The following assertions hold.
\begin{enumerate}
\item[(i)] Plurisubharmonic functions are stable by decreasing limit.
\item[(ii)] If $F \colon X \to Y$ is a holomorphic map between two complex surfaces and $\phi$ is a plurisubharmonic function on $Y$, then the pullback $F^* \phi =\phi \circ F$ is a plurisubharmonic function on $X$. 
\end{enumerate}
\end{thm}

Take $\phi$ a plurisubharmonic function on $X$, then the current $T$ given in local coordinates by
$$T \colon= i\partial \bar \partial \phi =i \sum_{i j}  \dfrac{\partial^2 \phi }{\partial z_i \partial \bar z_j} dz_i \wedge d\bar z_j $$
defines a closed positive $(1,1)$-current on $X$.

An important example is when $V = \{f=0 \}$ is locally the zero locus of a holomorphic function $f$, then the current of integration on $V$ satisfies the equality:
\begin{equation*}
[V] = i \partial \bar \partial \log|f|.
\end{equation*}

We now explain how to intersect two $(1,1)$ currents to obtain a $(2,2) $ current (a measure) on $X$. 
If $\alpha, \beta$ are two smooth $(1,1)$ forms on $X$, then $\alpha \wedge \beta$ defines a signed measure on $X$, so one can define the intersection $T_\alpha $ with $T_\beta$ by the measure given by $\alpha \wedge \beta$. 
If $V$ and $W$ are two distinct irreducible complex analytic curves on $X$, then one can define their intersection as  the discrete measure on $V\cap W$ counted with multiplicity. We  want to intersect two  $(1,1)$ currents induced by plurisubharmonic functions in such a way that the intersection would coincide with the other way showed previously.  The method in this case is more subtle, but was devised by Bedford-Taylor, then extended by Demailly (\cite{bedford_taylor}, \cite[Section 4]{demailly_agbook}).

Take $\phi, \psi$ two plurisubharmonic functions on $X$.
One difficulty arises immediately if there are points on $X$ where $\phi$ or $\psi$  are equal to $ -\infty$.  
We thus define the unbounded locus of $\phi$, denoted  $L(\phi)$ to be the set of points $p$ for which $\phi$ is unbounded on any neighborhood of $p$.  
To define the product $i\partial \bar\partial \phi \wedge i \partial \bar\partial \psi$, we shall apply the following  result. 

\begin{thm} (\cite[Corollary 4.10]{demailly_agbook}) \label{thm_bedford_taylor} Suppose that the intersection $L(\phi) \cap \supp \partial \bar \partial \psi$ is a countable number of points of $X$, then the current $\phi i \partial \bar \partial \psi $ is well-defined and has locally finite mass.
Moreover the measure $ \partial \bar{\partial}( -\phi  \partial \bar \partial \psi)$ is a well-defined positive measure and also has locally finite mass. 
\end{thm}

\begin{thm} (\cite[Proposition 4.6]{demailly_agbook}) \label{thm_chern_levine} Consider a plurisubharmonic function $\phi$ on $X$ and a closed positive $(1,1)$ current $T$ on $X$. 
Suppose that the intersection $\supp(T) \cap L(\phi)$ is a countable union of points. Then for any compact set $K,L$ in $X$ such that  $L \subset K^o$, there exists a neighborhood $V$ of $K \cap L(\phi)$ and a constant $C=C(K,L, \phi)>0$ such that:
\begin{equation*}
|| \phi T||_{ L} \leqslant C ||\phi||_{ L^{\infty}(K\setminus V)} ||T||_K,
\end{equation*}
\begin{equation*}
|| i\partial \bar\partial (\phi T)||_{ L} \leqslant C ||\phi||_{ L^{\infty}(K\setminus V)} ||T||_K,
\end{equation*}
where $|| \cdot ||_L$ denotes the mass of a given current on $L$.
\end{thm}

One key feature of the intersection of currents is the stability by decreasing limit of plurisubharmonic functions.

\begin{thm} Let $X = \P^1 \times \P^1$ or $X = \P^2$. Consider $\phi, \psi$ two plurisubharmonic functions such that $L(\phi) \cap L(\psi)$ is a countable union of points. 
Then the current:
\begin{equation}
- \partial \bar \partial \phi \wedge \partial \bar \partial \psi := - \partial \bar \partial  ( \phi \partial \bar \partial \psi)
\end{equation}
is well defined, and if $\phi_n, \psi_n$ are plurisubharmonic functions decreasing to $\phi$ and $\psi$ respectively, then
\begin{equation*}
\lim_n -\partial \bar \partial \phi_n \wedge \partial \bar \partial \psi_n = - \partial \bar \partial \phi \wedge \partial \bar \partial \psi.
\end{equation*}
\end{thm}

As a consequence of the above theorem together with the fact that the pullback of plurisubharmonic functions by holomorphic maps remains holomorphic, we get the following consequence.

\begin{cor} The intersection of currents is stable by pullback by holomorphic maps.
\end{cor}

\subsection{Divisors, line bundles, and first Chern class}

We shall restrict our discussion to the case where $X$ a smooth projective variety of complex dimension $1$ or $2$ (a curve or a surface). 
A \textbf{divisor} $D$ on $X$ is a finite linear combination with integer coefficients $a_i$ of irreducible hypersurfaces $D_i$ in $X$. We shall denote by:
\begin{equation*}
D = \sum a_i [D_i]. 
\end{equation*} 
If $X$ is a smooth curve, then hypersurfaces are points so a divisor $D$ is a linear combination of points. 
If $X$ is a complex surface, then a divisor is a linear combination of algebraic curves.
\medskip

Any divisor $D$ on $X$ induces a cohomology class in the deRham cohomology $H^{2}(X, \R)$. 
Indeed, if $D_i$ is smooth, then we consider the current of integration forms of degree $\dim X -2$ along the submanifold $D_i$. This defines an element in the dual of $H^{\dim X -2}(X,\R)$, which is identified to $H^2(X)$ via the Poincar\'e duality. 
When $D_i$ is not smooth, the integration along the smooth locus also yields a $d$-closed current on $X$.  
Since $D_i$ is also a complex manifold, one can also integrate forms of bidegree $(\dim X-1, \dim X-1)$ so the current of integration along $D_i$ also defines an element of the dual of $H^{\dim X -1  , \dim X -1}(X , \C)$. Hence $D_i$ induces a class in $H^{1,1}(X,\C)$ by Poincar\'e duality. 
\medskip

We now explain the connection between line bundles and divisors on curves and surfaces.
Given a divisor $D$ on a curve or a surface $X$, there exists a line bundle $L_D$ which admits a meromorphic section whose divisor of poles and zeros gives $D$. Conversely, if $L$ is a holomorphic line bundle, then the divisors of poles and zeros of one of its meromorphic section is a divisor on $X$ which is by definition the first Chern class of $L$, denoted $c_1(L)$.

\subsection{Positivity of divisors}
In this section, we present the various notions of positivity for divisors on a curve or a surface. 
These notions from algebraic geometry  refer to various cones one can consider in the cohomology of an algebraic variety.

Fix a divisor $D$ given by:
\begin{equation*}
D = \sum a_i D_i,
\end{equation*}
where $a_i$ are integers and $D_i$ are irreducible hypersurfaces on $X$.

When all the coefficients $a_i$ are non-negative, one says that $D$ is \textbf{effective} and it is one notion of positivity.
The other notions we will use are \textbf{ampleness} and \textbf{nefness} and are related to the intersection product of divisors on surfaces.

\bigskip

Suppose now that $X$ is a surface and take $D_1, D_2$ two distinct irreducible curves on $X$. 
If $p$ belongs to the intersection of these two curves, write $f =0 , g=0$ the local equation of $D_1$ and $D_2$ near $p$ and denote by $\mathcal{O}_p$ the local ring of holomorphic functions on $X$ near $p$, then we define the multiplicity of the intersection at $p$ to be:
\begin{equation*}
m_p(D_1,D_2):= \dim_\C \mathcal{O}_p/\langle f,g\rangle \geqslant 0 ,
\end{equation*}
where $\langle f,g\rangle$ denotes the ideal generated by $f$ and $g$.
The intersection of the two divisors $D_1$ and $D_2$, denoted $D_1\cdot D_2$ is by definition the linear combination of points given by:
\begin{equation*}
D_1 \cdot D_2 = \sum_{p\in D_1 \cap D_2} m_p(D_1, D_2) [p]. 
\end{equation*}
We extend this intersection to linear combination of irreducible curves by linearity and obtain a pairing on divisors. 
If we sum the contributions of all the multiplicities, we obtain a number called the \textbf{degree of the intersection}, which is denoted $(D_1 \cdot D_2)$:
\begin{equation*}
(D_1 \cdot D_2) = \sum_{p\in D_1 \cap D_2} m_p(D_1, D_2) \in \N.
\end{equation*}

We define nefness and ampleness in our setting using Kleiman's criterion \cite[Theorem 1.2.23, Theorem 1.4.9]{lazarsfeld_positivity_1}. 

A divisor $D= \sum a_i [p_i]$ on a curve $X$ is \textbf{ample} if its degree $\sum a_i$ is positive and it is \textbf{nef} if $\sum a_i$ is non-negative.

A divisor $D $ on a surface $X$ is \textbf{ample} if for any curve $C$ on $X$, the intersection $(D\cdot C)$ is positive and it is \textbf{nef} if the intersection $(D\cdot C)$ is non-negative.

Similarly, one says that a \textbf{line bundle is nef} (resp. ample) if it has a meromorphic section whose divisor of poles and zeros is nef (resp. ample).

\subsection{Rational maps} \label{section_appendix_rational_maps}


Take $X,Y$ two smooth projective surfaces. 
Fix two non-empty Zariski open subset $U,V$ of $X$ and $Y$ respectively and take a map $f_U \colon U \to V$ given by \textbf{polynomials} in some affine coordinate chart of $X,Y$. 
We say that the triple $(U,V,f_U)$ defines \textbf{a rational map $f$}, and we denote it by $f \colon X \dashrightarrow Y$. 
Observe that this definition seems quite loose since one can always restrict $f_U$ to another Zariski open subset contained in $U$. 
In any case, the Zariski (or Hausdorff) closure in $X\times Y$ of the graph of $f_U$ does not depend on the choice of the two Zariski open subsets $U,V$, we call this set the \textbf{graph} of the rational map $f$. 

When $U=X$, the rational map $f$ is said to be \textbf{regular}.

A rational map $f \colon X\dashrightarrow Y$ is called  \textbf{birational} if the map $f_U \colon U \to V$ has an inverse $f^{-1}_V \colon V \to U $ defined by polynomials in some affine coordinates. This inverse then induces a rational map $f^{-1}\colon Y \dashrightarrow X$.

The main advantage of this formulation is that for any rational transformation $f\colon X \dashrightarrow X$, one can always consider the conjugation $\varphi \circ f \circ \varphi^{-1}$ where $\varphi \colon X \dashrightarrow X'$ is a birational map between smooth projective surfaces. 
Indeed, if $f_U \colon U \to V$ is the map associated to $f$ on two Zariski open subsets $U,V$ of $X$, we choose also two Zariski open subset $U',V'$ of $X$ and a map $\varphi\colon U' \subset X \to V' \subset X'$ corresponding to the birational transformation $\varphi$. 
We consider the map:
\begin{equation*}
\varphi \circ f_U \circ \varphi^{-1} \colon \varphi(U \cap f_U^{-1} (U'))  \to V' ,
\end{equation*}
and this determines a rational map which we denote by $\varphi \circ f \circ \varphi^{-1}$.
In particular, one can lift rational maps on $\mathbb{P}^2$ to rational maps on a blow-up of $\mathbb{P}^2$.
\subsection{Pullback and pushforward of divisors} \label{appendix_push_full}

We recall general facts on pullback and pushforward of divisors. 

\textbf{Pushforward by regular maps}: Take $f \colon X \to Y$ a regular map of smooth surfaces. 
If $C$ is an irreducible curve in $X$, the pushforward of the divisor $[C]$ by $f$, denoted $f_* [C]$ is given by the formula:
\begin{equation}
 f_* [C] = \left\{ \begin{array}{ll}
m [f(C)] \text{ if }f(C) \text{ is a curve}, \\
0 \text{ if }f(C) \text{ is a point},
\end{array}  \right.
 \end{equation} 
where $m \in \N$ is the topological degree of the restriction of $f$ to $C$ onto its image.
We shall refer to $m$ as the \textit{multiplicity} of the divisor $[f(C)]$.
The pushforward is then defined by extending linearly $f_*$ to the abelian group of divisors on $X$.

\medskip

\textbf{Pullback by regular maps}: Take $f\colon X\to Y$ a regular map between two smooth projective varieties and $D$ be a irreducible hypersurface on $Y$. 
Take $\{g=0\}$ to be a local equation of $D$ near some point $p\in D$, 
write $$g \circ f = \prod_j g_j^{a_j},$$ where $a_j$ are integers and $g_j$ are local holomorphic functions on $X$ vanishing on an irreducible hypersurface $D_j$. 
Then the pullback of the divisor $D$ by $f$, denoted $f^* D$ is the divisor given by:
\begin{equation*}
f^* D := \sum a_i [D_i].
\end{equation*}
For a more abstract definition of the pullback, we shall refer to \cite{lazarsfeld_positivity_1}.
\medskip

We now explain the general convention for the pullback and pushforward for \textit{rational maps}. 

\textbf{Pullback and pushforward for rational maps}: 
Take $f \colon X \dashrightarrow Y$ a rational map on smooth surfaces, we take $\Gamma$  the graph of $f$ in $X\times Y$ (defined in the previous section) and denote by $\pi_1,\pi_2$ the two projections from $\Gamma$ onto the first and second component respectively.
We thus obtain the following commutative diagram:
\begin{equation*}
 \xymatrix{ & \Gamma \ar[rd]^{\pi_2} \ar[ld]_{\pi_1}& \\
 X \ar@{-->}[rr]^f& & Y. }
 \end{equation*}
 
 The pullback and pushforward of a divisor $D_Y $ on $Y$ and of a divisor $D_X$ on $X$ respectively, denoted $f^* D_Y$ and $f_* D_X$ are defined formally as:
 \begin{equation*}
  f^* D_Y :={\pi_1}_* \pi_2^* D_Y,
  \end{equation*} 
  \begin{equation*}
  f_* D_X := {\pi_2}_*\pi_1^* D_X.
  \end{equation*}
The pullback and pushforward action on divisors induce morphisms in the $(1,1) $ cohomology:
\begin{equation*}
f_* \colon H^{1,1}(X) \to H^{1,1}(Y),
\end{equation*}   
  \begin{equation*}
f^* \colon H^{1,1}(Y) \to H^{1,1}(X).
  \end{equation*}

One concept we will often use is the notion of \textbf{strict transform}. The strict transform of a curve $C$ in $X$ by $f$ is by definition the closure
\begin{equation*}
\overline{f (C\setminus I(f))},
\end{equation*}
where $I(f)$ denotes the indeterminacy set of $f$.
  
In practice, computing the pullback and pushforward action in cohomology can be done geometrically for  classes represented by ample divisors. 
The method  is to choose a representative in good position to simplify the computations. 
Indeed, if the divisor $D_X$ does not contain any indeterminacy point of $f$, then its pullback by $\pi_1$ is the preimage of $D_X$ by $\pi_1$ and $f_* D_X$ is the strict transform of $D_X$ by $f$:
\begin{equation*}
f_* D_X = [\overline{f(D_X)}].
\end{equation*}

When $D_X$ is an ample divisor which passes through an indeterminacy point of $f$, we can choose another divisor $D'_X$ representing the same class as $D_X$ in $H^{1,1}(X)$ which does not contain any indeterminacy point and determine the strict transform of  $D_X'$ by $f$.   
However, the difficulty arises when one wants to compute the pullback or the pushforward of a class of a divisor which is not ample, in this case, one needs to compute explicitly the pullback by $\pi_1$ and the pushforward by $\pi_2$ associated with $f$.

\subsection{Canonical divisor}
One divisor of particular interest in the paper is the \textbf{canonical divisor}. 
We give its definition below. 
Take $X$ is a smooth projective curve or a smooth  projective surface and fix $\omega$ a meromorphic $1$-form if $X$ is a curve or a meromorphic $2$ form if $X$ is a surface.  The \textbf{canonical divisor}, denoted $K_X$, is by definition the divisor of poles and zeros of the form $\omega$.
\begin{example} When $X = \P^1$, the form $\omega = dz$ in $\C \subset \P^1$, and at infinity $\omega$ has a pole of order $2$, i.e $\omega$ is of the form $- (1/z'^2)d z'$ where $z' = 1/z$.  
In particular, we have:
\begin{equation*}
K_{\P^1} = - 2 [\infty].
\end{equation*}

When $X = \P^2$, the form $\omega =dx \wedge dy$ in the affine coordinates $[x;y;1] \in \P^2$ has a pole of order $3$ on the line at infinity. So we have:
\begin{equation*}
K_{\P^2} = - 3 L,
\end{equation*} 
where $L$ is the line at infinity.
\end{example}

\subsection{Blowing up and down}
\label{appendix_blow_up}

Take $Y$ be a smooth projective surface and fix $p$ a point in $Y$. 
As a set, the \textbf{blow-up $X$ of $Y$ at the point $p$} is obtained by replacing  $p $ with the projective tangent bundle at  $p$:
\begin{equation*}
X := Y  \setminus \{p \}\cup \P(T_p Y).  
\end{equation*} 
Since $p$ is the smooth point, the tangent space at $p$ is naturally identified to $\{p\} \times \C^2$ and since $\P(\C^2) = \P^1$, the space $X$ is obtained from $Y$ by adding a $\P^1$ corresponding to the tangent directions at $p$. 
The subset $E:= \P^1 = \P(T_p Y)$ is called \textbf{the exceptional divisor} of the blow-up of $Y$ at $p$. By construction, the spaces $X\setminus E$ and $Y\setminus \{ p \}$ are biholomorphic. 
The set $X$ has a natural complex structure induced by $Y$ so $X$ is also a smooth projective surface and there is a natural holomorphic map $\pi \colon X \to Y$ which collapses $E$ to $p$ and is a biholomorphism from $X\setminus E $ to $Y \setminus \{p\}$. This map $\pi$ is called the \textbf{blow-down} map.

When one works with blow-ups, one common fact is to relate the  \textbf{strict transform} with the full preimage by $\pi$. 
Take $C$ a curve in $Y$, the \textbf{strict transform of $C$} is given by:
\begin{equation*}
\overline{\pi^{-1}(C \setminus \{p \})}.
\end{equation*}
One sees that if $p$ belongs to $C$, then the preimage of $\pi^{-1}(C)$ is the union of the strict transform of $C$ with the exceptional divisor, whereas $\pi^{-1}(C)$ is equal to the strict transform of $C$ when the curve $C$ does not pass through $p$.

Since we will be blowing-up points and compute explicitly the images of exceptional divisors by rational transformations, let us explain how one proceeds in local coordinates. 
Observe first that blowing up is a local operation as one only modifies the variety near the point $p$. 
Let us identify a neighborhood of $p$ as $\C^2$ with holomorphic coordinates $(x,y)$ so that $p$ is identified with $(x=0,y=0)$. 
The set of complex lines passing through $(0,0)$ are of the form:
\begin{equation*}
-\lambda x + \mu y =0,
\end{equation*}
where $[\lambda ; \mu] \in \P^1$.
The blow-up $X$ of $\C^2$ at $(0,0)$ is given  in coordinates by:
\begin{equation*}
\{ ((x,y) , [\lambda; \mu]) \in \C^2 \times \P^1 | -\lambda x + \mu y =0   \} \subset \C^2 \times \P^1.
\end{equation*}
This is the intrinsic definition of the blow-up, but to write the blow-down map $f$, one needs to write in the two charts on $\P^1$, corresponding to the parametrization of non-horizontal lines and non-vertical lines, we obtain:

\begin{equation*}
(x,\lambda)\in \C^2 \mapsto ((x, \lambda x) , [\lambda; 1]) \in X \mapsto (x, y= \lambda x) \in Y, 
\end{equation*}
where the exceptional divisor $f^{-1}(0,0)$ has local equation $x=0$ and the second chart is given by:
\begin{equation*}
(y, \mu) \in \C^2 \mapsto ((\mu y,y) , [1 ; \mu]) \in X \mapsto (x= \mu y , y ) \in Y,
\end{equation*}
where $y=0$ is the local equation of the exceptional divisor.

\begin{example} If $C$ is the curve $P(x,y)=y - x^2=0$, we can compute the local equations of the strict transform $\tilde{C}$ of $C$ on the blow-up of $(0,0)$. 
Observe that the tangent line to $C$ at $(0,0)$ is horizontal. We thus choose the parametrization of non-vertical curves $(x , \lambda) \mapsto (x, \lambda x), [\lambda, 1] $ on the blow-up:
The equations of $f^{-1}( C \setminus \{ (0,0)\})$ are of the form:
\begin{equation*}
P\circ f\colon (x, \lambda) \in \C^* \times \C \mapsto \lambda x - x^2 =  x( \lambda -  x). 
\end{equation*}
Since the local equations of $E$ is $x=0$ but $x\neq 0$, we deduce that the local equation of the strict transform is $\lambda - x$. 
The intersection of $\tilde{C}$ with the exceptional divisor $E$ is transverse at $(x=0,\lambda=0)$, and this comes from the fact that the tangent line to $C$ at $(0,0)$ is horizontal. 
\end{example}

Let us now discuss the cohomology of the blow-up at one point.
The cohomology of $X$ can be split into two parts:
\begin{equation*}
H^{1,1}(X) = \pi^* H^{1,1}(Y) \oplus \C E.
\end{equation*}
In practice, one decomposes divisors on $X$ as a sum of strict transform of divisors on $Y$ with a multiple of $E$ (see \cite[p.475]{griffiths_harris}).

One crucial fact is that the intersection product on $X$ can be deduced from the intersection product on $Y$ together with the fundamental equality (see \cite[p. 475]{griffiths_harris}): 
\begin{equation*}
(E \cdot E) = -1. 
\end{equation*}

The intersection of a strict transform of a curve with the exceptional divisor can be deduced geometrically or in coordinates. For the self-intersection of a strict transform of a curve, one uses the following statement. 
\begin{prop} If $X$is the blow-up of $Y$ at a point $p$. 
If $p$ is a smooth point on a curve $C$ in $Y$, then the strict transform $\tilde{C}$ satisfies:
\begin{equation*}
(\tilde{C} \cdot \tilde{C}) = (C\cdot C) - 1.
\end{equation*}
\end{prop}

\begin{example} Take $L$ a line in $\P^2$ and $p$ a point in $L$. Take $X$ to be the blow-up of $\P^2$ at $p$ and denote by $\pi\colon X \to \P^2$ the blow-down map, $E$ the exceptional divisor $\pi^{-1}(p)$. Then:
\begin{equation*}
\pi^* L = \tilde{L} + E,
\end{equation*}
where $\tilde{L}$ is the strict transform of the line at infinity by $\pi$.
Moreover, we have:
\begin{equation*}
(\tilde{L}\cdot E) =1 , (E\cdot E) = -1, (\tilde{L} \cdot \tilde{L}) = 0. 
\end{equation*}
\end{example}

\begin{example} If $X$ is the blow-up of $\P^2$ at a point on the line at infinity and $\pi \colon X \to \P^2$ is the blow-down map, 
then 
\begin{equation*}
K_X= \pi^* K_{\P^2} + E = - 3 (\tilde{L} + E) + E = - 3 \tilde{L} -2E,
\end{equation*}
where $\tilde{L}$ is the strict transform of the line at infinity.
\end{example}

We now summarize the formula involving pullback and pushforward of divisors under blow-ups and blow-down and general maps in the following proposition.

\begin{prop} \label{prop_formula_push_pull}
 Let $f \colon X \to Y$ be a regular map of smooth surfaces.  
 \begin{enumerate}
 
 \item[(i)] (\cite[Proposition 2.3 (c)]{fulton})  Fix $D$ a divisor on $Y$, $D'$ a divisor on $X$, one has:
 \begin{equation*}
 (f^* D \cdot D') = (D \cdot f_* D')
 \end{equation*}
 \item[(ii)] (\cite[Remark 1.1.13]{lazarsfeld_positivity_1}) Take $D,D'$ two divisors on $Y$, then one has:
 \begin{equation*}
 (f^* D \cdot f^* D') = d (D \cdot D'),
 \end{equation*}
 where $d$ is the topological degree of $f$.
 \item[(iii)] Take $D$  a divisor on $Y$. One has $f_* f^* D = d D \in H^{1,1}(Y)$, where $d$ is the topological degree of $f$. 
 \item[(iv)] (\cite[Proposition 2.5.5]{huybrechts}) Suppose that $X$ is the blow-up at one point of $Y$ and that $\pi$ is the map blowing down the exceptional divisor $E$ to that point, then one has in $H^{1,1}(X)$:
 \begin{equation*}
 K_X = \pi^* K_Y + E .
 \end{equation*}
 \item[(v)](\cite[Formula p.475]{griffiths_harris}) If $X$ is the blow-up of $Y$ at one point $p$ and $\pi \colon X \to Y$ is the blow-down map. Fix an irreducible curve $C$ on $Y$  and take $f=0$  the local equation defining $C$ near $p$.  The pullback $\pi^* C$ is equal to:
 \begin{equation*}
 \pi^* C =  \pi^{o}(C) + \ord_p(f) E,
 \end{equation*}
 where $E$ is the exceptional divisor above $p$, $\pi^o(C)$ is the strict transform of $C$ by $\pi$ and where $\ord_p(f)$ is the order of vanishing of the function $f$ at the point $p$.
 \end{enumerate}
\end{prop}

\begin{proof}
We only need to prove (iii), which is a consequence of (i) and (ii). 
Indeed, take  $D,D'$ two divisors on $Y$, we have using (i):
\begin{equation*}
(f_* f^* D \cdot D') = (f^* D \cdot f^* D').
\end{equation*}
By assertion (ii), we obtain:
\begin{equation*}
(f_* f^* D \cdot D' ) =  d (D \cdot D'),
\end{equation*}
for all divisors $D'$ on $Y$. 
Using the fact the the intersection product is non-degenerate in $H^{1,1}(Y)$, we have proved that $f_* f^* D = d D$.
\end{proof}

\bibliographystyle{alpha}
\bibliography{references}

\newcommand{\etalchar}[1]{$^{#1}$}
\begin{thebibliography}{BHPVdV04}

\bibitem[BD19]{berteloot_dinh}
Francois Berteloot and Tien-Cuong Dinh.
\newblock The mandelbrot set is the shadow of a julia set, 2019.

\bibitem[BG00]{bartholdi_grigorchuk_hecke}
Laurent Bartholdi and Rostislav~I. Grigorchuk.
\newblock On the spectrum of {H}ecke type operators related to some fractal
  groups.
\newblock {\em Tr. Mat. Inst. Steklova}, 231(Din. Sist., Avtom. i Beskon.
  Gruppy):5--45, 2000.

\bibitem[BGJ{\etalchar{+}}19]{teplyaev_rodgers_basilica}
Antoni Brzoska, Courtney George, Samantha Jarvis, Luke~G Rogers, and Alexander
  Teplyaev.
\newblock Spectral properties of graphs associated to the basilica group.
\newblock {\em arXiv preprint arXiv:1908.10505}, 2019.

\bibitem[BGS03]{BGS03}
Laurent Bartholdi, Rostislav~I. Grigorchuk, and Zoran {S}uni\'{k}.
\newblock Branch groups.
\newblock In {\em Handbook of algebra, {V}ol. 3}, volume~3 of {\em Handb.
  Algebr.}, pages 989--1112. Elsevier/North-Holland, Amsterdam, 2003.

\bibitem[BHPVdV04]{barth_peters_van}
Wolf~P. Barth, Klaus Hulek, Chris A.~M. Peters, and Antonius Van~de Ven.
\newblock {\em Compact complex surfaces}, volume~4 of {\em Ergebnisse der
  Mathematik und ihrer Grenzgebiete. 3. Folge. A Series of Modern Surveys in
  Mathematics [Results in Mathematics and Related Areas. 3rd Series. A Series
  of Modern Surveys in Mathematics]}.
\newblock Springer-Verlag, Berlin, second edition, 2004.

\bibitem[BLR17]{bleher_lyubich_roeder}
Pavel Bleher, Mikhail Lyubich, and Roland Roeder.
\newblock Lee-{Y}ang zeros for the {DHL} and 2{D} rational dynamics, {I}.
  {F}oliation of the physical cylinder.
\newblock {\em J. Math. Pures Appl. (9)}, 107(5):491--590, 2017.

\bibitem[BLS93]{bedford_lyubich_smillie_4}
Eric Bedford, Mikhail Lyubich, and John Smillie.
\newblock Polynomial diffeomorphisms of {${\bf C}^2$}. {IV}. {T}he measure of
  maximal entropy and laminar currents.
\newblock {\em Invent. Math.}, 112(1):77--125, 1993.

\bibitem[Bro65]{brolin}
Hans Brolin.
\newblock Invariant sets under iteration of rational functions.
\newblock {\em Ark. Mat.}, 6:103--144 (1965), 1965.

\bibitem[BS91]{bedford_smilie_polynomial_diffeo_currents}
Eric Bedford and John Smillie.
\newblock Polynomial diffeomorphisms of {${\bf C}^2$}: currents, equilibrium
  measure and hyperbolicity.
\newblock {\em Invent. Math.}, 103(1):69--99, 1991.

\bibitem[BT82]{bedford_taylor}
Eric Bedford and B.~A. Taylor.
\newblock A new capacity for plurisubharmonic functions.
\newblock {\em Acta Math.}, 149(1-2):1--40, 1982.

\bibitem[Can01]{cantat_k3}
Serge Cantat.
\newblock Dynamique des automorphismes des surfaces {$K3$}.
\newblock {\em Acta Math.}, 187(1):1--57, 2001.

\bibitem[CR19]{chio_roeder}
Ivan Chio and Roland Roeder.
\newblock Chromatic zeros on hierarchical lattices and equidistribution on
  parameter space.
\newblock {\em arXiv preprint arXiv:1904.02195}, 2019.

\bibitem[Dan20]{dang_degrees}
Nguyen-Bac Dang.
\newblock Degrees of iterates of rational maps on normal projective varieties.
\newblock {\em Proceedings of the London Mathematical Society},
  121(5):1268--1310, 2020.

\bibitem[Dem]{demailly_agbook}
Jean-Pierre Demailly.
\newblock Complex analytic and differential geometry, 2012.
\newblock {\em This book is--and will be--available as an
  “OpenContentBook”, see https://www-fourier. ujf-grenoble. fr/~
  demailly/manuscripts/agbook. pdf}.

\bibitem[DF01]{diller_favre}
Jeffrey Diller and Charles Favre.
\newblock Dynamics of bimeromorphic maps of surfaces.
\newblock {\em Amer. J. Math.}, 123(6):1135--1169, 2001.

\bibitem[DG17]{dudko_grigorchuk}
Artem Dudko and Rostislav Grigorchuk.
\newblock On spectra of {K}oopman, groupoid and quasi-regular representations.
\newblock {\em J. Mod. Dyn.}, 11:99--123, 2017.

\bibitem[DG18]{dudko_grigorchuk_diagonal}
Artem Dudko and Rostislav Grigorchuk.
\newblock On diagonal actions of branch groups and the corresponding
  characters.
\newblock {\em J. Funct. Anal.}, 274(11):3033--3055, 2018.

\bibitem[DJ08]{dabija_jonsson}
Marius Dabija and Mattias Jonsson.
\newblock Endomorphisms of the plane preserving a pencil of curves.
\newblock {\em Internat. J. Math.}, 19(2):217--221, 2008.

\bibitem[dlHGCS99]{harpe_grigorchuk}
Pierre de~la Harpe, Rostislav~I. Grigorchuk, and Tullio Chekerini-Sil'berstain.
\newblock Amenability and paradoxical decompositions for pseudogroups and
  discrete metric spaces.
\newblock {\em Tr. Mat. Inst. Steklova}, 224(Algebra. Topol. Differ. Uravn. i
  ikh Prilozh.):68--111, 1999.

\bibitem[DNT12]{dinh_nguyen_truong}
Tien-Cuong Dinh, Vi{\^e}t-Anh Nguy{\^e}n, and Tuyen~Trung Truong.
\newblock On the dynamical degrees of meromorphic maps preserving a fibration.
\newblock {\em Communications in Contemporary Mathematics}, 14(06):1250042,
  2012.

\bibitem[DS05]{dinh_sibony_une_borne_sup}
Tien-Cuong Dinh and Nessim Sibony.
\newblock Une borne sup\'erieure pour l'entropie topologique d'une application
  rationnelle.
\newblock {\em Ann. of Math. (2)}, 161(3):1637--1644, 2005.

\bibitem[DS08]{dinh_sibony_equidistribution_green_current}
Tien-Cuong Dinh and Nessim Sibony.
\newblock Equidistribution towards the {G}reen current for holomorphic maps.
\newblock {\em Ann. Sci. \'{E}c. Norm. Sup\'{e}r. (4)}, 41(2):307--336, 2008.

\bibitem[Duj14]{dujardin_bifurcation}
Romain Dujardin.
\newblock Bifurcation currents and equidistribution in parameter space.
\newblock In {\em Frontiers in complex dynamics}, volume~51 of {\em Princeton
  Math. Ser.}, pages 515--566. Princeton Univ. Press, Princeton, NJ, 2014.

\bibitem[Fek23]{fekete}
M.~Fekete.
\newblock \"{U}ber die {V}erteilung der {W}urzeln bei gewissen algebraischen
  {G}leichungen mit ganzzahligen {K}oeffizienten.
\newblock {\em Math. Z.}, 17(1):228--249, 1923.

\bibitem[FJ03]{favre_jonsson_brolin}
Charles Favre and Mattias Jonsson.
\newblock Brolin's theorem for curves in two complex dimensions.
\newblock {\em Ann. Inst. Fourier (Grenoble)}, 53(5):1461--1501, 2003.

\bibitem[FLMn83]{FLM}
Alexandre Freire, Artur Lopes, and Ricardo Ma\~{n}\'{e}.
\newblock An invariant measure for rational maps.
\newblock {\em Bol. Soc. Brasil. Mat.}, 14(1):45--62, 1983.

\bibitem[FP11]{favre_pereira}
Charles Favre and Jorge~Vit\'{o}rio Pereira.
\newblock Foliations invariant by rational maps.
\newblock {\em Math. Z.}, 268(3-4):753--770, 2011.

\bibitem[FS92]{fukushima_shima}
M.~Fukushima and T.~Shima.
\newblock On a spectral analysis for the {S}ierpi\'{n}ski gasket.
\newblock {\em Potential Anal.}, 1(1):1--35, 1992.

\bibitem[FS95]{fornaess_sibony_II}
John~Erik Fornaess and Nessim Sibony.
\newblock Complex dynamics in higher dimension. {II}.
\newblock In {\em Modern methods in complex analysis ({P}rinceton, {NJ},
  1992)}, volume 137 of {\em Ann. of Math. Stud.}, pages 135--182. Princeton
  Univ. Press, Princeton, NJ, 1995.

\bibitem[Ful98]{fulton}
William Fulton.
\newblock {\em Intersection theory}, volume~2 of {\em Ergebnisse der Mathematik
  und ihrer Grenzgebiete. 3. Folge. A Series of Modern Surveys in Mathematics
  [Results in Mathematics and Related Areas. 3rd Series. A Series of Modern
  Surveys in Mathematics]}.
\newblock Springer-Verlag, Berlin, second edition, 1998.

\bibitem[GH94]{griffiths_harris}
Phillip Griffiths and Joseph Harris.
\newblock {\em Principles of algebraic geometry}.
\newblock Wiley Classics Library. John Wiley \& Sons, Inc., New York, 1994.
\newblock Reprint of the 1978 original.

\bibitem[Giz80]{gizatullin}
M.~H. Gizatullin.
\newblock Rational {$G$}-surfaces.
\newblock {\em Izv. Akad. Nauk SSSR Ser. Mat.}, 44(1):110--144, 239, 1980.

\bibitem[GLN16]{grigorchuk_leemann_nagnibeda}
Rostislav Grigorchuk, Paul-Henry Leemann, and Tatiana Nagnibeda.
\newblock Lamplighter groups, de {B}rujin graphs, spider-web graphs and their
  spectra.
\newblock {\em J. Phys. A}, 49(20):205004, 35, 2016.

\bibitem[GLN17]{grigorchuk_lenz_nagnibeda_subshift}
Rostislav Grigorchuk, Daniel Lenz, and Tatiana Nagnibeda.
\newblock Schreier graphs of {G}rigorchuk's group and a subshift associated to
  a nonprimitive substitution.
\newblock In {\em Groups, graphs and random walks}, volume 436 of {\em London
  Math. Soc. Lecture Note Ser.}, pages 250--299. Cambridge Univ. Press,
  Cambridge, 2017.

\bibitem[GLN18]{grigorchuk_lenz_nagnibeda_schroedinger}
Rostislav Grigorchuk, Daniel Lenz, and Tatiana Nagnibeda.
\newblock Spectra of {S}chreier graphs of {G}rigorchuk's group and
  {S}chroedinger operators with aperiodic order.
\newblock {\em Math. Ann.}, 370(3-4):1607--1637, 2018.

\bibitem[GN07]{grigorchuk_nekrashevych_self_similar_operator}
Rostislav Grigorchuk and Volodymyr Nekrashevych.
\newblock Self-similar groups, operator algebras and {S}chur complement.
\newblock {\em J. Mod. Dyn.}, 1(3):323--370, 2007.

\bibitem[GNS00]{psim}
R.~I. Grigorchuk, V.~V. Nekrashevich, and V.~I. Sushchanski\u{\i}.
\newblock Automata, dynamical systems, and groups.
\newblock {\em Tr. Mat. Inst. Steklova}, 231(Din. Sist., Avtom. i Beskon.
  Gruppy):134--214, 2000.

\bibitem[GNS15]{grigorchuk_nekrashevych_sunic_self_similar}
Rostislav Grigorchuk, Volodymyr Nekrashevych, and Zoran {S}uni\'{c}.
\newblock From self-similar groups to self-similar sets and spectra.
\newblock In {\em Fractal geometry and stochastics {V}}, volume~70 of {\em
  Progr. Probab.}, pages 175--207. Birkh\"{a}user/Springer, Cham, 2015.

\bibitem[Gri80]{grigorchuk_80}
Rostislav~I. Grigor\v{c}uk.
\newblock On {B}urnside's problem on periodic groups.
\newblock {\em Funktsional. Anal. i Prilozhen.}, 14(1):53--54, 1980.

\bibitem[Gri83]{grigorchuk_83}
Rostislav~I. Grigorchuk.
\newblock On the {M}ilnor problem of group growth.
\newblock {\em Dokl. Akad. Nauk SSSR}, 271(1):30--33, 1983.

\bibitem[Gri84]{grigorchuk_84}
Rostislav~I. Grigorchuk.
\newblock Degrees of growth of finitely generated groups and the theory of
  invariant means.
\newblock {\em Izv. Akad. Nauk SSSR Ser. Mat.}, 48(5):939--985, 1984.

\bibitem[Gri00a]{just_infinite}
R.~I. Grigorchuk.
\newblock Just infinite branch groups.
\newblock In {\em New horizons in pro-{$p$} groups}, volume 184 of {\em Progr.
  Math.}, pages 121--179. Birkh\"{a}user Boston, Boston, MA, 2000.

\bibitem[Gri00b]{G00}
Rostislav~I. Grigorchuk.
\newblock Branch groups.
\newblock {\em Mat. Zametki}, 67(6):852--858, 2000.

\bibitem[Gri05]{solved}
Rostislav Grigorchuk.
\newblock Solved and unsolved problems around one group.
\newblock In {\em Infinite groups: geometric, combinatorial and dynamical
  aspects}, volume 248 of {\em Progr. Math.}, pages 117--218. Birkh\"{a}user,
  Basel, 2005.

\bibitem[Gri11]{grigorchuk_11}
Rostislav~I. Grigorchuk.
\newblock Some problems of the dynamics of group actions on rooted trees.
\newblock {\em Tr. Mat. Inst. Steklova}, 273(Sovremennye Problemy
  Matematiki):72--191, 2011.

\bibitem[GS08]{grigorchuk_sunic_hanoi}
Rostislav Grigorchuk and Zoran {S}uni\'{c}.
\newblock Schreier spectrum of the {H}anoi {T}owers group on three pegs.
\newblock In {\em Analysis on graphs and its applications}, volume~77 of {\em
  Proc. Sympos. Pure Math.}, pages 183--198. Amer. Math. Soc., Providence, RI,
  2008.

\bibitem[Gue04]{guedj_decay}
Vincent Guedj.
\newblock Decay of volumes under iteration of meromorphic mappings.
\newblock {\em Ann. Inst. Fourier (Grenoble)}, 54(7):2369--2386 (2005), 2004.

\bibitem[GY17]{grigorchuk_yang}
Rostislav Grigorchuk and Rongwei Yang.
\newblock Joint spectrum and the infinite dihedral group.
\newblock {\em Proceedings of the Steklov Institute of Mathematics},
  297(1):145--178, 2017.

\bibitem[GY19]{goldberg_yang}
Bryan Goldberg and Rongwei Yang.
\newblock Self-similarity and spectral dynamics.
\newblock {\em Research Gate preprint}, 12 2019.

\bibitem[GZ99]{grigorchuk_zuk_99}
Rostislav~I. Grigorchuk and Andrzej \.{Z}uk.
\newblock On the asymptotic spectrum of random walks on infinite families of
  graphs.
\newblock In {\em Random walks and discrete potential theory ({C}ortona,
  1997)}, Sympos. Math., XXXIX, pages 188--204. Cambridge Univ. Press,
  Cambridge, 1999.

\bibitem[GZ01]{grigorchuk_zuk_lamplighter}
Rostislav~I. Grigorchuk and Andrzej \.{Z}uk.
\newblock The lamplighter group as a group generated by a 2-state automaton,
  and its spectrum.
\newblock {\em Geom. Dedicata}, 87(1-3):209--244, 2001.

\bibitem[GZ02]{grigorchuk_zuk_spectral}
Rostislav~I. Grigorchuk and Andrzej \.{Z}uk.
\newblock Spectral properties of a torsion-free weakly branch group defined by
  a three state automaton.
\newblock In {\em Computational and statistical group theory ({L}as {V}egas,
  {NV}/{H}oboken, {NJ}, 2001)}, volume 298 of {\em Contemp. Math.}, pages
  57--82. Amer. Math. Soc., Providence, RI, 2002.

\bibitem[GZ04]{grigorchuk_zuk_ihara}
Rostislav~I. Grigorchuk and Andrzej \.{Z}uk.
\newblock The {I}hara zeta function of infinite graphs, the {KNS} spectral
  measure and integrable maps.
\newblock In {\em Random walks and geometry}, pages 141--180. Walter de
  Gruyter, Berlin, 2004.

\bibitem[Har77]{hartshorne}
Robin Hartshorne.
\newblock {\em Algebraic geometry}, volume~52.
\newblock Springer Science \& Business Media, 1977.

\bibitem[Huy05]{huybrechts}
Daniel Huybrechts.
\newblock {\em Complex geometry}.
\newblock Universitext. Springer-Verlag, Berlin, 2005.
\newblock An introduction.

\bibitem[Kes59]{kesten}
Harry Kesten.
\newblock Symmetric random walks on groups.
\newblock {\em Trans. Amer. Math. Soc.}, 92:336--354, 1959.

\bibitem[Kig89]{kigami_harmonic}
Jun Kigami.
\newblock A harmonic calculus on the {S}ierpi\'{n}ski spaces.
\newblock {\em Japan J. Appl. Math.}, 6(2):259--290, 1989.

\bibitem[Kig01]{kigami}
Jun Kigami.
\newblock {\em Analysis on fractals}, volume 143 of {\em Cambridge Tracts in
  Mathematics}.
\newblock Cambridge University Press, Cambridge, 2001.

\bibitem[KM98]{kollar_mori}
J{\'a}nos Koll{\'a}r and Shigefumi Mori.
\newblock {\em Birational geometry of algebraic varieties}, volume 134 of {\em
  Cambridge Tracts in Mathematics}.
\newblock Cambridge University Press, Cambridge, 1998.
\newblock With the collaboration of C. H. Clemens and A. Corti, Translated from
  the 1998 Japanese original.

\bibitem[Laz04]{lazarsfeld_positivity_1}
Robert Lazarsfeld.
\newblock {\em Positivity in algebraic geometry. {I}}, volume~48 of {\em
  Ergebnisse der Mathematik und ihrer Grenzgebiete. 3. Folge. A Series of
  Modern Surveys in Mathematics [Results in Mathematics and Related Areas. 3rd
  Series. A Series of Modern Surveys in Mathematics]}.
\newblock Springer-Verlag, Berlin, 2004.
\newblock Classical setting: line bundles and linear series.

\bibitem[Lyu]{L:book}
Mikhail Lyubich.
\newblock {\em Conformal geometry and Dynamics of quadratic polynomials vol
  I,II}.
\newblock http://www.math.stonybrook.edu/~mlyubich/book.pdf.

\bibitem[Lyu82]{L:note}
M.~Yu. Lyubich.
\newblock The measure of maximal entropy of a rational endomorphism of a
  {R}iemann sphere.
\newblock {\em Funktsional. Anal. i Prilozhen.}, 16(4):78--79, 1982.

\bibitem[Lyu83]{lyubich_entropy_properties_riemann_sphere}
Mikhail~Ju. Lyubich.
\newblock Entropy properties of rational endomorphisms of the {R}iemann sphere.
\newblock {\em Ergodic Theory Dynam. Systems}, 3(3):351--385, 1983.

\bibitem[Mil06]{milnor}
John Milnor.
\newblock {\em Dynamics in one complex variable}, volume 160 of {\em Annals of
  Mathematics Studies}.
\newblock Princeton University Press, Princeton, NJ, third edition, 2006.

\bibitem[Nek05]{nekrashevych_self_similar}
Volodymyr Nekrashevych.
\newblock {\em Self-similar groups}, volume 117 of {\em Mathematical Surveys
  and Monographs}.
\newblock American Mathematical Society, Providence, RI, 2005.

\bibitem[Ram84]{rammal}
R.~Rammal.
\newblock Spectrum of harmonic excitations on fractals.
\newblock {\em J. Physique}, 45(2):191--206, 1984.

\bibitem[RS80]{reed_simon}
Michael Reed and Barry Simon.
\newblock {\em Methods of modern mathematical physics. {I}}.
\newblock Academic Press, Inc. [Harcourt Brace Jovanovich, Publishers], New
  York, second edition, 1980.
\newblock Functional analysis.

\bibitem[RS97]{russakovskii_shiffman}
Alexander Russakovskii and Bernard Shiffman.
\newblock Value distribution for sequences of rational mappings and complex
  dynamics.
\newblock {\em Indiana Univ. Math. J.}, 46(3):897--932, 1997.

\bibitem[Sab03a]{sabot_smf}
Christophe Sabot.
\newblock Spectral properties of self-similar lattices and iteration of
  rational maps.
\newblock {\em M\'{e}m. Soc. Math. Fr. (N.S.)}, (92):vi+104, 2003.

\bibitem[Sab03b]{sabot_spectral}
Christophe Sabot.
\newblock Spectral properties of self-similar lattices and iteration of
  rational maps.
\newblock {\em M\'{e}m. Soc. Math. Fr. (N.S.)}, (92):vi+104, 2003.

\bibitem[Sab04]{sabot_electrical}
Christophe Sabot.
\newblock Electrical networks, symplectic reductions, and application to the
  renormalization map of self-similar lattices.
\newblock In {\em Fractal geometry and applications: a jubilee of {B}enoit
  {M}andelbrot. {P}art 1}, volume~72 of {\em Proc. Sympos. Pure Math.}, pages
  155--205. Amer. Math. Soc., Providence, RI, 2004.

\bibitem[Sib99]{sibony_1999}
Nessim Sibony.
\newblock Dynamique des applications rationnelles de {$\bold P^k$}.
\newblock In {\em Dynamique et g\'eom\'etrie complexes ({L}yon, 1997)},
  volume~8 of {\em Panor. Synth\`eses}, pages ix--x, xi--xii, 97--185. Soc.
  Math. France, Paris, 1999.

\bibitem[Str99]{strichartz}
Robert~S. Strichartz.
\newblock Some properties of {L}aplacians on fractals.
\newblock {\em J. Funct. Anal.}, 164(2):181--208, 1999.

\bibitem[Tep98]{teplyaev_gasket}
Alexander Teplyaev.
\newblock Spectral analysis on infinite {S}ierpi\'{n}ski gaskets.
\newblock {\em J. Funct. Anal.}, 159(2):537--567, 1998.

\bibitem[VV]{vorobets_notes}
Mariya Vorobets and Yaroslav Vorobets.
\newblock Notes on the {G}rigorchuk renormalization transformation.

\end{thebibliography}

\end{document}